\ifdefined\singlecolumn
\documentclass[11pt,draftclsnofoot,onecolumn]{IEEEtran}
\else
\documentclass[10pt,twocolumn]{IEEEtran}
\fi

\usepackage{cite}
\usepackage{graphicx} % DO NOT CHANGE THIS
\graphicspath{{figures/}}

\usepackage[tbtags]{amsmath}
\usepackage{amssymb,color,amsthm}
\usepackage[none]{hyphenat}
\usepackage{subfigure}
\usepackage{bbm}
\usepackage{etoolbox}
\usepackage[dvipsnames]{xcolor}
\usepackage{url}

\newtoggle{twocolumn}
\ifCLASSOPTIONtwocolumn\toggletrue{twocolumn}\else\togglefalse{twocolumn}\fi
\def\onlytwo#1{\iftoggle{twocolumn}{#1}{}}

\def\L{{\cal L}}

\def\bx{\mathbf x}

\def\cA{\mathcal A}

\def\cB{\mathcal B}

\def\minimize{\mathop{\text{minimize}}}

\usepackage[utf8]{inputenc}

\usepackage{float}
\restylefloat{table}

\usepackage{caption}

\usepackage[ruled,noline]{algorithm2e}
\usepackage{multirow}
\usepackage{textcomp}

\newtheorem{lemma}{Lemma}
\newtheorem{theorem}{Theorem}

\newtheorem{define}{Definition}
\newtheorem{remark}{Remark}
\newtheorem{prop}{Proposition}
\newtheorem{assump}{Assumption}

\newcommand{\bigo}[1]{\mathcal{O}(#1)}
\newcommand{\eg}{{\it e.g.}}
\newcommand{\ie}{{\it i.e.}}
\newcommand{\norm}[1]{\|#1\|}
\newcommand{\argmin}{\text{argmin}}
\newcommand{\argmax}{\text{argmax}}
\newcommand{\st}{\text{subject to}}

\newcommand{\reals}{{\mbox{$\mathbf{R}$}}}

\newcommand{\size}[1]{\left|#1\right|}
\newcommand{\abs}[1]{\size{#1}}
\newcommand{\inn}[2]{\langle#1, #2\rangle}
\newcommand{\crit}{\mathop{\rm crit}}
\newcommand{\dist}{\mathop{\rm dist}}
\newcommand{\dom}{\mathop{\rm dom}}

\newcommand{\prox}{\text{prox} }

\def\x{{\mathbf x}}
\def\y{{\mathbf y}}
\def\z{{\mathbf z}}
\def\u{{\mathbf u}}
\def\bv{{\mathbf v}}
\def\w{{\mathbf w}}
\def\d{{\mathbf d}}

\def\A{{\mathbf A}}
\def\X{{\mathbf X}}
\def\U{{\mathbf U}}
\def\V{{\mathbf V}}

\def\E{{\mathbb E}}

\def\cW{{\mathcal W}}
\def\cA{{\mathcal A}}
\def\cB{{\mathcal B}}
\def\cS{{\mathcal S}}

\begin{document}

\title{Leveraging Two Reference Functions in Block Bregman Proximal Gradient Descent for Non-convex and Non-Lipschitz Problems}
\author{Tianxiang~Gao, \emph{Student Member, IEEE},
	Songtao~Lu, \emph{Member, IEEE},
	Jia~Liu, \emph{Senior Member, IEEE},
    \\and Chris~Chu, \emph{Fellow, IEEE}

%\thanks{Manuscript received May 15, 2016; revised October 6, 2016,
%January 6, 2017, and February 16, 2017; accepted February 20, 2017. Date of
%publication XXX YY, 2017. The associate editor coordinating the review of this
%manuscript and approving it for publication was Marco Moretti. Part of the
%paper was presented at the 42nd IEEE International Conference on Acoustics, Speech, and Signal Processing (ICASSP), New Orleans, March 5--9, 2017. This work was
%supported in part by NSF under Grants No.~1523374 and No.~1526078, and by
%AFOSR under Grant No.~15RT0767.}

\thanks{
    Tianxiang Gao and Chris Chu are with the Department of Electrical and Computer Engineering,
    Iowa State University, Ames, IA 50011, USA (emails: \{gaotx,cnch\}@iastate.edu).}
\thanks{
    Songtao Lu is with IBM Research AI, IBM Thomas J. Waston Research Center, Yorktown Heights, New York 10562,  USA (email: songtao@ibm.com). }
\thanks{
	Jia Liu is with the Department of Computer Science,
	Iowa State University, Ames, IA 50011, USA (email: jialiu@iastate.edu).}
}

\maketitle

\begin{abstract}
In the applications of signal processing and data analytics, there is a wide class of non-convex problems whose objective function is freed from the common global Lipschitz continuous gradient assumption (\eg, the nonnegative matrix factorization (NMF) problem). Recently, this type of problem with some certain special structures has been solved by Bregman proximal gradient (BPG). This inspires us to propose a new {\bf B}lock-wise {\bf two}-references {\bf B}regman proximal gradient (B2B) method, which adopts two reference functions so that a closed-form solution in the Bregman projection is obtained. Based on the relative smoothness, we prove the global convergence of the proposed algorithms for various block selection rules. In particular, we establish the global convergence rate of $\bigo{\frac{\sqrt{s}}{\sqrt{k}}}$ for the greedy and randomized block updating rule for B2B, which is $\bigo{\sqrt{s}}$ times faster than the cyclic variant, \ie, $\bigo{\frac{s}{\sqrt{k}} }$, where $s$ is the number of blocks, and $k$ is the number of iterations. Multiple numerical results are provided to illustrate the superiority of the proposed B2B compared to the state-of-the-art works in solving NMF problems.

\end{abstract}

\section{Introduction} \label{sec:intro}
In this paper, we consider the following problem
\begin{align}
\minimize \; f(\x) \quad \st\; \x\geq 0,\label{opt:problem}
\end{align}
where the function $f$ is continuously differentiable and possibly non-convex. In the literature, Problem \eqref{opt:problem} is usually reformulated as
\begin{align}
	\minimize F(\x):=f(\x_1,\cdots,\x_s) + \sum_{b=1}^{s} r_b(\x_b),\label{opt:block}
\end{align}
where $\x$ is partitioned into $s$ blocks, and the function $r(\x) = \sum_{b=1}^{s}r_b(\x_b)$ is a block-structured nonsmooth regularizer. Problem \eqref{opt:block} is equivalent to Problem \eqref{opt:problem} when the regularizer is the indicator function of the nonnegative orthant, \ie, $r_b = \delta_+$.

Due to the block structure, Problem \eqref{opt:block} is usually solved by a \textit{block coordinate descent} (BCD) method, where $F$ is minimized over the $b$-th block exactly \cite{bertsekas1997nonlinear} or inexactly \cite{tseng2009coordinate,razaviyayn2013unified}. Since the update is block-wise, different block selection strategies usually result in various convergence behavrious/rates. In \cite{nesterov2012efficiency}, the authors provide the first convergence rate result of BCD by adopting the \textit{randomized} rule for convex and smooth optimization problems. Later, the same convergence rate is obtained in \cite{richtarik2014iteration} for nonsmooth convex problems, while a relatively slower sublinear convergence rate is proved by \cite{patrascu2015efficient} proves guarantee for the nonconvex setting. In \cite{nutini2015coordinate}, the authors show that a better convergence rate can be obtained by using Gauss-Southwell (G-So) or \textit{greedy} rule when the problem is unconstrained and strongly convex. In recent years, the convergence rate of the Gauss-Seidel (G-S) or \textit{cyclic} rule has also been extensively studied in the convex setting \cite{beck2013convergence,saha2013nonasymptotic,sun2015improved}. However, the convergence rate for the cyclic rule is usually the same or even slower than the randomized and greedy rule. In the non-convex settings, the previous work \cite{xu2017globally} estimates the convergence rate of the cyclic rule based on the assumption that $F$ satisfies Kurdyka-\L ojasiewicz (K\L) property.

A commonly used assumption in showing the convergence of BCD methods in the literature is that the gradient of $f$ is globally Lipschitz-continuous. However, this could be a restrictive assumption violated in diverse applications in practice, such as matrix factorization \cite{lee1999learning,pagd19}, tensor decomposition \cite{kim2007nonnegative}, matrix/tensor completion \cite{xu2012alternating}, Poisson likelihood models \cite{he2016fast}, etc. Although this assumption may be relaxed by adopting conventional line search methods, the efficiency and computational complexity of the BCD methods are unavoidably distorted, especially when the size of the problem is large. To overcome this longstanding isssue, existing works in \cite{bauschke2016descent,lu2018relatively} develop a new framework by adapting the geometry of $f$ through the Bregman distance paradigm, which helps to derive a descent lemma to quantify the decrease of the objective value by Bregman proximal gradient (BPG) instead of classical proximal gradient (PG). As a result, the convergence behaviour of BPG can be characterized without assuming globally Lipschitz-continuous gradient of the objective function. Further, this framework is extended in \cite{bolte2018first} to the case of nonconvex optimization.

Despite a cyclic Bregman BCD (CBBCD) method had been proposed in \cite{ahookhosh2019multi,wang2018block} by leveraging this framework, the convergence rate results of the Bregman BCD methods remains unknown. In this paper, we bridge this gap by conducting rigorous convergence rate analysis for different rules of block selection. We note that a main drawback of the (block) Bregman-proximal-based methods is that the Bregman projection problem (i.e., a constrained convex optimization that will be specified later) has no closed form solution, which necessitates an iterative algorithm involved to solve, resulting in increased computational complexity. To address this challenge, we propose a new {\bf b}lock-wise {\bf two} references {\bf B}regman proximal gradient descent (B2B) method by leveraging \textit{two} reference functions, where the original problem is split into two parts to induce a closed-form solution to the Bregman projection subproblem. Further, we show that the proposed B2B method is $\bigo{\sqrt{s}}$ times faster than the CBBCD method if the greedy or randomized block updating rule is used. The main contributions of this paper are highlighted as follows.

%In this paper, we also adopt this framework but propose a new block-wise Bregman proximal gradient descent method by leveraging two reference functions, so that a closed-form solution to the Bregman projection can be obtained easily. Moreover, we conduct rigorous convergence rate analysis for different rules of block selection. As a result, we show that CBBCD is $\bigo{\sqrt{s}}$ times slower than our proposed {\bf b}lock-wise {\bf two} references {\bf B}regman proximal gradient descent (B2B) in achieving a stationary point of Problem \eqref{opt:problem} if the greedy or randomized block updating rule is used.

\noindent{\bf Convergence analysis}: we establish a rigorous convergence rate analysis of  the block-wise cyclic Bregman BCD (CBBCD) method, showing that its convergence rate to the stationary points is $\bigo{s/\sqrt{k}}$.

\noindent{\bf Implementation efficiency}: a new block-wise Bregman proximal gradient descent method is proposed by leveraging \textit{two} reference functions such that the Bregman projection of this method has a closed-form solution.

\noindent{\bf Faster convergence rate}: we prove that  if the greedy or randomized rule of updating blocks is adopted, B2B with a constant stepsize achieves a $\bigo{\sqrt{s}}$ times faster convergence rate than CBBCD, i.e., $\bigo{\sqrt{s}/\sqrt{k}}$.

\noindent{\bf Numerical discovery}: extensive experimental results reveal the superiority of the B2B method compared with the state-of-the-art counterparts implemented on a diverse dataset for nonnegative matrix factorization (NMF).

%{\bf Notation:} Bold upper case letters without subscripts (e.g., $\bX,\bY$)
%denote matrices and bold lower case letters without subscripts (e.g.,
%$\bx,\by$) represent vectors. The notation $\bZ_{i,j}$ denotes the $(i,j)$-th
%entry of matrix $\bZ$. Vector $\bX_i$ denotes the $i$th row of matrix $\bX$
%and $\bX'_m$ denotes the $m$th column of the matrix. Calligraphic letters such
%as $\mathcal{Y}$ are used to denote sets.

\section{Preliminaries}
In this section, we review the well-know proximal gradient method and its variant, namely, the Bregman proximal gradient method, as they are one of the baisc methods to minimize a composite objective function.

\subsection{Notation}
Throughout this paper, bold upper case letters denote matrices (\eg. $\X$), bold lower case letters denote vectors (\eg, $\x$), and Calligraphic letters (\eg, $\mathcal{X}$) are used to denote sets. We use $\norm{\cdot}$ to denote the Euclidean norm. $\delta_{\mathcal{X}}(\x)$ represents the indicator function: $ \delta_{\mathcal{X}}(\x) = 0$ if $\x\in \mathcal{X}$; otherwise, $ \delta_{\mathcal{X}}(\x) =\infty$. If $\mathcal{X}=\reals^N_+$, the indicator function becomes $\delta_+(\x)$. For a function $f$, $\nabla f(\x)$ denotes its the gradient, while $\nabla_b f(\x)$ is the partial gradient with respect to the $b$-th block. We also denote $f_b(\x_b)$ as a function of the $b$-th block, while the rest of the blocks are fixed. Clearly, $\nabla _b f(\x) = \nabla f_b(\x_b)$. If $f$ is not differentiable, $\partial f$ denotes the subdifferential of $f$.

Given a convex function $\phi$, the \textit{proximal mapping} of $\phi$ at a point $\bx$ is defined as
\begin{align}
%\resizebox{0.85\hsize}{!}{
%	$
	\prox_\phi (\x) = \argmin_\u \phi(\u) + \frac{1}{2}\norm{\u-\x}^2.\label{eq:proximal mapping}
%	$
%}
\end{align}
This mapping is well-defined due to the convexity of $\phi$. If $\phi =\delta_{\mathcal{X}}$, \eqref{eq:proximal mapping} reduces to \textit{orthogonal projection }
\begin{align}
P_{\mathcal{X}}(\x) = \argmin\{\norm{\u-\x}: \u\in \mathcal{X}\}.
\end{align}
If, in addition, $\mathcal{X}=\reals_+^n$, the projection mapping has a  closed-form solution as the following
\begin{align}
[\x]_+ = \argmin\{\norm{\u-\x}: \u\geq \textbf{0}\} = \max\{\x,\textbf{0}\},\label{eq:orthogonal projection}
\end{align}
where the max operation is taken componentwise.

Similarly, the \textit{Bregman proximal mapping} is defined by replacing the Euclidean distance with the Bregman distance
\begin{align}
T_{\phi}(\x) = \argmin_\u \phi( \u) + D_h(\u, \x),
\end{align}
where $D_h(\u,\x) = h(\u) - h(\x) - \inn{\nabla h(\x)}{\u-\x}$ is the \textit{Bregman distance} with the reference convex function $h$. This mapping is also well-defined since the functions $\phi$ and $h$ are convex. The convexity of $h$ also implies $D_h(\x,\y)\geq 0, \forall \x,\y$. If, in addition, $h$ is strictly convex, $D_h(\x,\y) = 0$ if and only if $\x=\y$. In the rest of this paper, we assume $h$ is strictly convex. Note that $D_h(\x,\y)$ is not symmetric in general. Therefore, we use \textit{symmetric coefficient} $\beta(h)=\inf \left\{\frac{D_h(\x,\y)}{D_h(\y,\x)}: \x \neq \y \right\}$ to measure the symmetry. When $\phi=\delta_{\mathcal{\x}}$, the Bregman proximal mapping reduces to the \textit{Bregman projection}
\begin{align}
P_{\mathcal{X}}^h(\x)=\argmin\{D_h(\u,\x):\u\in \mathcal{X}\}.
\end{align}
Clearly, the Bregman projection is much harder to solve in general compared to the orthogonal projection for $\mathcal{X} = \reals_+^n$. Note that if we choose the energy function as the reference function, \ie, $h(\cdot)=\frac{1}{2}\norm{\cdot}^2$, the Bregman proximal mapping and the Bregman projection boils down to the classical proximal mapping and orthogonal projection.

\subsection{Proximal Gradient Method}
%\begin{algorithm}[h]
%	\caption{(Bregman) Proximal Gradient Descent Method.}
%	\label{alg:PG}
%	
%	Choose $\x^0\in \reals^n_+$.\\
%	\Repeat{Some stopping criterion is satisfied}{
%		Set $\alpha^k$ and obtain $\x^{k+1}$ by \eqref{eq:proximal} or \eqref{eq:bregman mapping solution}.
%	}
%\end{algorithm}
Next, we review the standard proximal gradient (PG) method since it is a fundamental method for minimizing the sum of a smooth function $f$ with a nonsmooth one $r$, \ie,
\begin{align}
\minimize \;F(\x) := f(\x) + r(\x).\label{opt:PG}
\end{align}
The following assumptions are made for Problem \eqref{opt:PG} throughout the paper.
\begin{assump}\label{assump:lower bounded}
	\begin{itemize}
		\item[(\emph{i})] $f$ is continuously differentiable.
		\item[(\emph{ii})] $r$ is a proper and lower semicontinuous.
		\item[(\emph{iii})] $F^* = \inf_{\x} F(\x) > -\infty$.
	\end{itemize}
\end{assump}
Clearly, Problem \eqref{opt:problem} satisfies the first two assumptions above. The last assumption is equivalent to assuming $f^* = \inf_{\x\geq 0} f(\x) > -\infty$.

At the $k$-th iteration, by linearizing the smooth function $f$, the PG method minimizes the following subproblem,
%\begin{subequations}
\begin{equation}\label{eq:proximal}
\resizebox{0.91\hsize}{!}{%
	$
	\x^{k+1}=\argmin_\u \; r(\u) + \inn{\nabla f(\x^k) }{\u-\x^k} + \frac{1}{2\alpha^k}\norm{\u - \x^k}^2,
	$
}
\end{equation}
where $\alpha^k$ is some positive stepsize. In the notation of proximal mapping, Eq. \eqref{eq:proximal} can be rewritten as
\begin{align}
\x^{k+1}=\prox_{\alpha^k r}(\x^k - \alpha^k \nabla f(\x^k))\label{eq:PG}.
\end{align}
For $r = \delta_+$, the update rule in \eqref{eq:PG} becomes
\begin{align}
\x^{k+1} = [\x^k - \alpha^k \nabla f(\x^k)]_+. \label{eq:bcd x_i}
\end{align}
Due to the simplicity of the orthogonal projection, PG is broadly used to solve \eqref{opt:problem}.

Eq. \eqref{eq:PG} can be further expressed in a more concise form as
\begin{align}
\x^{k+1} = \x^k - \alpha^k G(\x^k),
\end{align}
where $G(\x^k)$ is called the \textit{generalized gradient} and defined by
\begin{align}
G(\x^k) = \frac{\x^k - \prox(\x^k - \alpha^k \nabla f(\x^k))}{\alpha^k}.\label{eq:generalized gradient}
\end{align}
Similar to the norm of the gradient for unconstrained problems, $\|G(\x^k)\|$ can be used to measure the optimality since
\begin{align}
G(\x^k)\in \nabla f(\x^k) + \partial r(\x^{k+1}),
\end{align}
and it is easy to show that $\norm{G(\x^*)} = 0$ if and only if $\x^*$ is a critical point of \eqref{opt:PG} defined by
$
0\in \partial F(\x^*).
$

The convergence results can be established by using a conventional line search method. However, a line search strategy is inefficient since it may need to evaluate the objective values multiple times so as to ensure a descent in the objective value. For many large-scale optimization problems, evaluating the objective function is inefficient or even impossible in some cases. Therefore, a constant stepsize with a predefined value is favored in practice. To establish the convergence results of the PG method with a constant stepsize, a common and crucial assumption is that $\nabla f(\x)$ is globally Lipschitz-continuous, \ie, there exists a constant $\ell > 0$ such that
\begin{align}
\norm{\nabla f(\y) - \nabla f(\x)} \leq \ell \norm{\y- \x}, \quad\forall \x,\y\geq 0.
\end{align}
With $0 < \alpha^k  \leq \frac{1}{\ell}$, the classic convergence result indicates that $\{ \|G(\x^k)\|\}$ converges to zero at the rate of $\bigo{1/\sqrt{k}}$. However, the global Lipschitz-continuity is a restrictive assumption. In the past, many objective functions in modern optimization problems do not satisfy this assumption. Another limitation of PG is that, similar to gradient descent, it suffers a very slow rate of convergence as it approaches a critical point in a zig-zag manner.

\subsection{Bregman Proximal Gradient}
The limitations of PG discussed above has recently been solved in \cite{bauschke2016descent,lu2018relatively}, which proposed the \textit{Bregman proximal gradient} (BPG) method that does not require the global Lipschitz-continuous gradient in objective functions.

%To review BPG, we start with the definition of the \textit{Bregman proximal mapping}. Replacing the Euclidean distance in \eqref{eq:proximal mapping} with the Bregman distance yields
%\begin{align}
%T_{\phi}(\x) = \argmin_u \phi(\u) + D_h(u,x),
%\end{align}
%where $D_h(u,x) = h(\u) - h(\x) - \inn{\nabla h(\x)}{u-x}$ is the \textit{Bregman distance }with the reference convex function $h$. This mapping is well-defined since the functions $\phi$ and $h$ are all convex. The convexity of $h$ also implies $D_h(x,y)\geq 0$ for all $x,y$. If, in addition, $h$ is strictly convex, then $D_h(x,y) = 0$ if and only if $x=y$. For the rest of this paper, we assume $h$ is strictly convex. Note that $D_h(x,y)$ is not symmetric in general. Therefore, we use \textit{symmetric coefficient} $\beta(h)=\inf \left\{\frac{D_h(x,y)}{D_h(y,x)}:x\neq y \right\}$ to measure the symmetry. Similarly, when $\phi=\delta_{\mathcal{\x}}$ , the Bregman proximal mapping reduces to the \textit{Bregman projection} defined as
%\begin{align}
%P_{\mathcal{\x}}^h=\argmin\{D_h(u,x):u\in \mathcal{\x}\}.
%\end{align}
%Clearly, the Bregman projection is much harder to solve compared to the orthogonal projection for $\mathcal{\x} = \reals_+^n$. Note that if we choose the energy function as the reference function, \ie, $h=\frac{1}{2}\norm{\cdot}^2$, the Bregman proximal mapping and the Bregman projection boils down to the classical proximal mapping and orthogonal projection.

In the $k$-th iteration, the BPG method constructs a similar subproblem as in \eqref{eq:proximal} by replacing the Euclidean distance with the Bregman distance, \ie,
\begin{align}
\resizebox{0.88\hsize}{!}{$
\x^{k+1}=\argmin_\u\; r(\u)  + \inn{\nabla f(\x^k) }{\u-\x^k} + \frac{1}{\alpha^k}D_h (\u,\x^k).
$}
%=&\argmin_u r(\u) + \frac{1}{2\alpha^k}\norm{\x^k - \alpha^k \nabla f(\x^k)}^2\\
%=&\prox_{\alpha^k r}(\x^k - \alpha^k \nabla f(\x^k)),
\end{align}
In the view of the Bregman proximal mapping, we have
\begin{align}
\x^{k+1} = T_{\alpha^k r} [T_{\alpha^k\hat{f}}(\x^k) ],\label{eq:bregman mapping solution}
\end{align}
where $\hat{f}(\u)=f(\x^k) + \inn{\nabla f(\x^k)}{\u-\x^k}$ is the linear approximation of $f$ at $\x^k$. The update rule in \eqref{eq:bregman mapping solution} involves a two-step operation, which can be written explicitly by
\begin{subequations}\label{eq:two steps}
\begin{align}
\y^{k+1} &= \argmin_\u \;\inn{\nabla f(\x^k)}{\u-\x^k} + \frac{1}{\alpha^k}D_h(\u, \x^k),\\
\x^{k+1} &= \argmin_\u \; r(\u) +  \frac{1}{\alpha^k}D_h(\u, \y^{k+1}).\label{eq:two steps Bregman projection}
\end{align}
\end{subequations}
Based on relative smoothness (to be explained in the next section), the convergence results of BPG are obtained in \cite{bauschke2016descent,lu2018relatively} for convex and in \cite{bolte2018first} for nonconvex settings, respectively. Although using Bregman proximal mapping overcomes the global Lipschitz-continuous gradient issue, the two subproblems in \eqref{eq:two steps} are in general not easy to be solved, even for $r= \delta_+$. For example, if $h(\x) = \frac{1}{2}\inn{\x}{\A \x}$ for some positive definite matrix $\A$, then the Bregman projection becomes a quadratic optimization problem under a componentwise nonnegative constraint, which does not have closed-form solutions for subproblem \eqref{eq:two steps Bregman projection}. In Section~\ref{section: two reference}, we propose a new algorithm that uses a different reference function for the Bregman projection so that a closed-form solution can be obtained.

\section{Block-wise Bregman Proximal Gradient}\label{section:bcd}
In this section, we first propose a (cyclic) Bregman BCD (BBCD) method, which extends the previous cyclic BCD by using the Bregman distance \footnote{While writing this paper, \cite{ahookhosh2019multi} proposes a similar cyclic BBCD method, but they did not provide the convergence rate.}.

Instead of updating all coordinates simultaneously, the BBCD method selects and updates a subset of blocks in each iteration while the rest of the blocks are fixed. At the $k$-th iteration, BBCD selects an index set $\mathcal{C}^k\subseteq\{1, \cdots, s\}$ such that (such that) if $b\in \mathcal{C}^k$, the $b$-th block can be updated by
\begin{align}\label{opt:GS}
\resizebox{.88\hsize}{!}
{$
	\x_{b}^{k+1}=\argmin_\u\; r_{b}(\u) + \inn{\nabla f^k_{b} (\x^k_{b})}{\u-\x_{b}^k}
	+ \frac{1}{\alpha^k}D_h(\u, \x^k_{b}),
	$}
\end{align}
otherwise it remains the same, \ie, $\x_b^{k+1} = \x_b^k$ for all $b\notin \mathcal{C}^k$.
Note that we simplify the notation by dropping the index $b$ in $D_{h_b}$. In a BPG fashion, Eq. \eqref{opt:GS} can be also rewritten as a two-step operation
\begin{subequations}
\begin{align}
\y_b^{k+1} &= \argmin \;\inn{\nabla f_b^k(\x^k)}{\u-\x_b^k} + \frac{1}{\alpha^k}D_h(\u, \x_b^k),\label{eq:two steps block}\\
\x_b^{k+1} &= \argmin \;r_b(\u) +  \frac{1}{\alpha^k}D_h(\u, \y_b^{k+1}).\label{eq:cyc bcd projection}
\end{align}
\end{subequations}

In contrast to the PG and BPG methods, $\norm{G(\x^k)}$ or $\norm{\x^{k+1}-\x^k}$ is not appropriate for measuring the optimality, because it is possible that only a subset of blocks are selected and updated in the whole process. Instead, \textit{projected gradient} $\nabla^P f(\x^k)$ is commonly used to measure optimality. In the case of $r = \delta_+$, $\nabla^P f(\x^k)$ is defined by \cite{lin2007projected}
\begin{align}\label{eq:projected gradient}
\nabla^P f(\x) \triangleq \begin{cases}
\nabla_i f(\x^k),	&\text{if $\x_i > 0$},\\
\min\{0, \nabla_i f(\x^k)\}, &\text{if $\x_i = 0$}.
\end{cases}
\end{align}
Similar to $\|G(\x^k)\|$, we have $\norm{\nabla^P f(\x^*) } = 0$ if and only if $\x^*$ is a critical point. Therefore, one needs to keep track of $\norm{\nabla^P f(\x^k)}$ as the algorithm proceeds, and stop the algorithm when $\norm{\nabla^P f(\x^k)}$ is small enough. Therefore, we define $\norm{\nabla^P f(\x^k)}$ as the \emph{optimality gap}.

The Gauss-Seidel (G-S) or cyclic rule used in Algorithm~\ref{alg:GS} is a special rule since it includes all blocks in $\mathcal{C}^k$ and updates them in the cyclic manner. As a result, $\norm{\x^{k+1}-\x^k}$ can be used to measure the optimality. Next, we provide a series of analyses and convergence results for the cyclic BBCD method.
\begin{algorithm}[h]
	\caption{Cyclic BBCD method.}
	\label{alg:GS}
	
	Choose $\x^0\in \reals^n_+$.\\
	\Repeat{Some stopping criterion is satisfied}{
		\For{$b= 1$ to $s$}{
			Set $\alpha^k$ and update $\x_{b}^k$ by \eqref{opt:GS}
		}
	}
\end{algorithm}

We start with the definition of \textit{relative smoothness} \cite{lu2018relatively,bauschke2016descent}, by which a new descent lemma is obtained without the assumption of the global Lipschitz-continuity of $\nabla f(\x)$.
\begin{define}\cite[Definition~1.1]{lu2018relatively}
	A pair of functions $(g, h)$ are said to be relatively smooth if $h$ is convex and there exists a scalar $L>0$ such that $Lh-g$  and $Lh+g$ are convex.
\end{define}
Note that the above definition holds for every convex function $h$, even $g$ is nonconvex. Moreover, the relative smoothness nicely translates the Bregman distance to produce a non-Lipschitz descent lemma \cite{lu2018relatively,bauschke2016descent}.
\begin{lemma}\cite[Lemma~2.1]{bolte2018first}\label{lemma:nolip}
	The pair of functions $(g,h)$ is relatively smooth if and only if for all $\x$ and $\y$, it holds that
	\begin{align}
	\abs{g(\y)- g(\x) - \inn{\nabla g(\x)}{\y-\x}} \leq LD_h(\y,\x).\label{lemma:nolip eq0}
	\end{align}
\end{lemma}
\begin{remark}\label{remark:nolip}
	\begin{itemize}
		\item[(\emph{i})] For the purpose of this paper, it is sufficient to only consider the convex condition of $Lh-g$ and the corresponding descent lemma, \ie, $g(\y)- g(\x) - \inn{\nabla g(\x)}{y-x} \leq LD_h(y,x)$.
		\item[(\emph{ii})] In abuse of the definition of Bregman distance $D_g$ (since $g$ is not convex), the non-Lipschitz descent lemma in Lemma~\ref{lemma:nolip} can be written as $\abs{D_g(y,x)}\leq D_h(y,x)$.
		\item[(\emph{iii})] In the special case with $h=\frac{1}{2}\norm{\cdot}^2$, the classical descent lemma is recovered
		\begin{align}
		\abs{g(\y)- g(\x) - \inn{\nabla g(\x)}{\y-\x}} \leq \frac{L}{2}\norm{\y-\x}^2.\nonumber
		\end{align}
		\item[(\emph{iv})] The relative smooth property is invariant when $h$ is $m$-strongly convex \cite{bolte2018first}.
		%		, \ie, $D_h(y,x)\geq \frac{m}{2}\norm{y-x}^2$. Indeed, since no convexity is assumed in $g$, let $\omega = \frac{m}{2}\norm{\cdot}$, then we have
		%		\begin{align}
		%		Lh - g = L(h-\omega) - (g - L\omega) = L\overline{h}-\overline{g},
		%		\end{align}
		%		namely, relatively smooth holds for new pair $(\overline{h},\overline{g})$.
	\end{itemize}
\end{remark}

For the rest of this paper, we additionally make the following assumptions.
\begin{assump}\label{assume: convex}
	\begin{itemize}
		\item[(\emph{i})] $(f_b, h_b)$ are relatively smooth with constant $L_b > 0$ and let $L = \underset{b}{\max}\{L_b\}$.
		\item[(\emph{ii})] $h_b$ is $m_b$-strongly convex and set $m = \underset{b}{\min}\{m_b\}$.
	\end{itemize}
\end{assump}

With the relative smoothness between $(f_b, h_b)$, the following proposition shows the basic convergence results. A similar results can be found in \cite{ahookhosh2019multi,wang2018block}.
\begin{prop}\label{prop:cyclic prop}
	Let $\{\x^k\}$ be the sequence generated by Algorithm~\ref{alg:GS} with G-S block selection rule and $\alpha^k = \alpha$ such that $0 < \alpha  < \frac{1}{L}$. The following assertions hold:
	\begin{itemize}
		\item[(\emph{i})] The sequence $\{F(\x^k)\}$ is nonincreasing, \ie,
		$
			\sum_{b=1}^s \left(\frac{1}{\alpha^k} - L_b\right)D_h(\x_b^{k+1},\x_b^k)\leq F(\x^k)-	F(\x^{k+1}).
		$
		\item[(\emph{ii})] $\sum_{k=0}^{\infty} \left[\sum_{b=1}^s D_h(\x_b^{k+1}, \x_b^k)\right] < \infty$, and hence for all $b$ $\lim\limits_{k\rightarrow \infty}D_h(\x_b^{k+1},\x_b^k) = 0$.
		\item[(\emph{iii})] $
		\min_{0\leq k\leq N} \left[\sum_{b=1}^s D_h(\x_b^{k+1},\x_b^k)\right] \leq \frac{\alpha(F(\x^0) - F^*)}{(N+1) (1 - \alpha L)}.
		$
	\end{itemize}
\end{prop}
\begin{remark}
	\begin{itemize}
		\item[(\emph{i})] If $r$ is convex, then we have $0 < \alpha L < 1 + \beta$, where $\beta = \underset{b}{\min}\{\beta(h_b)\}$. In particular, using the subgradient of $r$, we obtain a stronger inequality
		$
			r_b(\x_b^{k+1})\leq r_b(\x_b^k) + \frac{1+\beta(h_b)}{\alpha^k} D_h(\x_b^{k+1},\x_b^k),
		$
		where we use $\beta(h_b)D_h(\x_b^{k+1},\x_b^k) \leq D_h(\x_b^k, \x_b^{k+1}) $. Hence, we have
		$
			\sum_{b=1}^s \left(\frac{1 + \beta(h_b)}{\alpha^k} - L_b\right)D_h(\x_b^{k+1},\x_b^k)\leq F(\x^k)-	F(\x^{k+1}).
		$
		\item[(\emph{ii})] The reference function $h_b$ for each block could be varied in different iterations. As a result, the coefficient $L_b$ should be written as $L_b^k$ since it could also change in different iterations. To simplify the expression, however, we assume the same reference function for each block in different iterations, so that we can set $L_b = \underset{k}{\max}\{L_b^k\}$ and the resulting analysis is the same.
	\end{itemize}
	
\end{remark}

In order to show the sequence $\{\x^k\}$ approaching to a critical point, we first show the subgradient of $F$ is upper bounded. For that purpose, we make the following additional assumption for this section.
\begin{assump}\label{assume:cyclic lipschitz}
	$\nabla f(\x)$ and $\nabla h_b(\x)$ are Lipschitz-continuous with constant $\ell > 0$ in any bounded set.
\end{assump}

\begin{prop}\label{prop:cyclic convergence}
	Let $\{\x^k\}$ be the sequence generated by Algorithm~\ref{alg:GS} that is assumed to be bounded. Let $0 < \alpha^k L < 1, \forall k$. For all $k\geq 0$, we have
	\begin{align}
	\norm{\w^{k+1}} \leq  s \ell \left(1+\frac{1}{\alpha^k}\right)  \norm{\x^{k+1} - \x^k},
	\end{align}	
	where $\w^{k+1}\in \partial F(\x^{k+1})$. Then every limit point of $\{\x^k\}$ is a critical point of $F$.
\end{prop}
The boundedness of the sequence $\{\x^k\}$ is a common assumption in the literature (\eg, \cite{bolte2014proximal,attouch2010proximal,attouch2009convergence}), because the function $f(\x)$ in many applications has bounded level sets and the descent in the objective function is guaranteed. For more details please see \cite{attouch2010proximal}.

The following proposition says that $\nabla^P f(\x^k) $ defined by \eqref{eq:projected gradient} is a subgradient of $F$ when $r = \delta_+$.
\begin{prop}\label{prop:subgradient equiavlent}
	If $r(\x) = \delta_+(\x)$, then $\nabla^P f(\x)\in \partial F(\x)$.
\end{prop}
Combining Proposition~\ref{prop:subgradient equiavlent} with Propositions~\ref{prop:cyclic prop} (\emph{iii}) and \ref{prop:cyclic convergence}, we immediately obtain the following inequality for CBBCD:
\begin{align*}
\min_{0\leq k\leq N} \norm{\nabla^P f(\x^k)}^2
%	\leq \min_{0\leq k\leq N} s^2M^2\left(1+\frac{1}{\alpha^k}\right)^2\norm{\x^{k}-\x^{k-1} }^2
\leq \frac{2s^2 \ell^2 (1+\alpha)\left(f(\x^0) -f^*\right)}{(N+1)(1-\alpha L)m}.
\end{align*}
Therefore, we have $\norm{\nabla^P f(\x^k)}$ converges to zero at the rate of $\bigo{s/\sqrt{k}}$.

Similar to BPG, computing the Bregman projection \eqref{eq:cyc bcd projection} is expensive in general. In the following section, we propose another BCD-type method that uses two different reference functions so that projection operation admits a closed-form solution. Further, the proposed method is also applicable for greedy or randomized rules so that it achieves a faster convergence rate.
\section{Block-wise Two References Bregman Proximal Gradient Descent}\label{section: two reference}
As we discussed in the previous sections, a stronger convergence result can be obtained by using the Bregman distance. However, the projection operation \eqref{eq:two steps Bregman projection} or \eqref{eq:cyc bcd projection} might be computationally expensive. To resolve this issue, we use a \textit{different} reference function $g$ for the projection subproblem so that the projection operation can be easily solved. We call this method \textit{Block-wise Two references Bregman proximal gradient} (B2B) method. With two different reference functions $h$ and $g$, the update rule \eqref{eq:two steps block}-\eqref{eq:cyc bcd projection} becomes
\begin{subequations}%\label{opt:two metric}
\begin{align}
\d_b^{k} &= \argmin_\d\; \inn{\nabla f_b(\x_b^k)}{\d_b} + D_h(\x_b^k + \d_b,\x_b^k), \label{opt:two metric mapping}\\
\x_b^{k+1} &= \argmin_\u \; r_b(\u) +  D_g(\u,\x_b^k + \alpha^k \d_b^k).\label{opt:two metric projection}
\end{align}
\end{subequations}
Here we first compute the search direction $\d^k$, where $\d_b^k$ is given by \eqref{opt:two metric mapping} and the rest entries are set to zero. The search direction is intuitive, since we have $\d^k = -\nabla f(\x^k)$ for PG, $\d^k=(\y^{k+1}-\x^k)/\alpha^k$ for BPG, and $\d_b^k=(\y_b^{k+1}-\x_b^k)/\alpha^k$ for CBBCD.

In the case of $r = \delta_+$, we set $g = \frac{1}{2}\norm{\cdot}^2$, then the $b$-th block update can be written as:
\begin{align}
\x_b^{k+1} = [\x_b^k + \alpha^k \d_b^k]_+,\label{eq:update}
\end{align}
where we use the fact that orthogonal projection has a closed-form solution \eqref{eq:orthogonal projection}. Clearly, the projection operation is much cheaper than the Bregman projection used in \eqref{eq:two steps block}-\eqref{eq:cyc bcd projection}. However, the obtained direction $\d^k$ is not always a descent direction. In \cite[Figure 1.2]{bertsekas2014constrained}, a counterexample with $h(\x) = \frac{1}{2}\inn{\x}{\A\x}$ was provided \cite{bertsekas2014constrained}, where one can obtain $f(\x^{k+1}) > f(\x^k)$ for \textit{all} $\alpha^k > 0$ with an unfavored positive definite matrix $\A$. By leveraging the special structure of $\delta_+$, we identify a class of valid blocks by which the descent of the objective value is guaranteed.

\subsection{Feasible descent direction and line search}
We define the notion of \textit{valid} coordinate by which a feasible descent direction is found so that the objective value is continuously decreased in each iteration for an appropriate stepsize.
\begin{define}
	A coordinate $\x_i$ is valid if it satisfies
	\begin{align}
	-\nabla_i f(\x)\notin \partial \delta_+(\x_i).
	\end{align}
\end{define}
In our B2B method, we enforce only using the valid coordinates in each block. As a result, the following lemma shows that the obtained direction $\d^k$ can always induce a feasible descent direction that guarantees a descent in the objective value.

\begin{lemma}\label{lemma:critical point}\label{lemma:descent}
	Define a uni-variate variable function of $\alpha$ as
	\begin{align}
	\x_b^k(\alpha) = [\x_b^k + \alpha \d_b^k]_+, \quad \forall \alpha > 0.
	\end{align}
	
	\begin{enumerate}
		\item[(\emph{i})] The following assertions are equivalent:
		\begin{itemize}
			\item[(1)] A vector $\x^k$ is a critical point;
			\item[(2)] $\norm{\nabla^P f(\x^k)} = 0$;
			\item[(3)] $-\nabla_b f(\x) \in \delta_+(\x_b)$, $\forall b$;
			\item[(4)] $\x_b^k(\alpha) = \x_b^k$, $\forall \alpha, b$.
		\end{itemize}
		\item[(\emph{ii})] 	If $\x^k$ is not a critical point and the selected block is valid, then there exists a stepsize $\overline{\alpha}_k$ such that
		\begin{align}
		f(\x^{k}(\alpha)) < f(\x^k), \quad \forall \alpha \in (0,\overline{\alpha}_k]\label{lemma:descent eq0}.
		\end{align}
	\end{enumerate}

\end{lemma}

With Lemma~\ref{lemma:descent}, we can establish the stationary convergence result by using an Armijo-like line search rule. Here scalars $\tau$, $\sigma$ and $\alpha_0$ are fixed. Choosing $\tau\in(0,1)$ and $\sigma\in (0,1/2)$, and we set $\alpha^k = \tau^{m_k}\alpha_0$, where $m_k$ is the smallest positive integer that satisfies
\begin{align}
\resizebox{0.88\hsize}{!}{
	$
	f(\x^k) - f(\x^k(\tau^{m} \alpha_0)) \geq -\sigma \inn{\nabla_b f(\x^k)}{\x_b^k(\tau^{m} \alpha_0) - \x_b^k}.\label{eq:line search}
	$
}
\end{align}

%The algorithm is summarized in Algorithm~\ref{alg:line search}. The following theorem gives the main convergence result.
%\begin{algorithm}[h]
%	\caption{B2B method with line search.}
%	\label{alg:line search}
%	Choose $\x^0\in \reals^n_+$.\\
%	\Repeat{Some stopping criterion is satisfied}{
%		Select a block $b\in \{1,2,\cdots, s\}$\\
%		Remove the invalid coordinates from the $b$-th block\\
%		Compute $\d^k$ by \eqref{opt:two metric mapping}\\
%		Set the stepsize $\alpha^k$ by the line search method \eqref{eq:line search}
%		Obtain $\x^{k+1}$ by \eqref{eq:update}
%	}
%\end{algorithm}

\begin{theorem}[Convergence of the line search method]\label{thm:line search}
	Let $\{\x^k\}$ be the sequence generated by \eqref{eq:update} with line search \eqref{eq:line search}. Then every limit point of $\{\x^k\}$ is a critical point.
\end{theorem}
The global convergence result is obtained without the assumption of global Lipschitz-continuous gradient. However, a line search strategy may be inefficient since it has to evaluate the objective function values multiple times to ensure the sufficient descent in the objective value. In the next subsection, we establish the convergence results for the constant stepsize strategy under mild conditions.

\subsection{Constant stepsize}
In practice, a line search strategy is not computational efficient, especially for high-dimensional problems, since evaluating the objective function is expensive or even impossible in many applications. Therefore, using a predefined constant stepsize is preferred in practice. The generic B2B algorithm with a constant stepsize is given in Algorithm~\ref{alg:B2B}. Note that the B2B method in Algorithm~\ref{alg:B2B} uses either the greedy or randomized rule.
\begin{algorithm}[h]
	\caption{B2B method with a constant stepsize}
	\label{alg:B2B}
	Choose $\x^0\in \reals^n_+$ and $\alpha$.\\
	\Repeat{Some stopping criterion is satisfied}{
		Select a block $b\in \{1,2,\cdots, s\}$ by \eqref{eq:greedy} or uniformly at random\\
		Remove the invalid coordinates \\
		Compute $\d^k$ by \eqref{opt:two metric mapping}\\
		Set $\alpha^k = \alpha$\\
		Obtain $\x^{k+1}$ by \eqref{eq:update}
	}
\end{algorithm}

We make the following additional assumption for the reference function $h_b$ in the rest of this section.
\begin{assump}\label{assume: smooth}
	The function $\nabla h_b$ is $M_b$-smooth on any bounded set and let $M = \underset{b}{\max}\{M_b\}$.
\end{assump}
By Lemma~\ref{lemma:nolip}, we can easily obtain the following fundamental inequality, which will play a crucial role in establishing the main convergence result.
\begin{lemma}\label{lemma:decrease}
	Let $\{\x^k\}$ be the sequence generated by Algorithm~\ref{alg:B2B} that is assumed to be bounded. Then we have
	\begin{align}
%	\resizebox{0.88\hsize}{!}{
%		$
		f(\x^{k+1})\leq f(\x^k) - \frac{\alpha^k(1+\beta)m_b}{2}\left(1-\frac{L_b M_b\alpha^k}{(1+\beta)m_b}\right)\norm{\d_b^k}^2 .\label{lemma:decrease eq inequality}
%		$
%	}
	\end{align}
	In particular, with $0 < \alpha^k < \frac{ (1+\beta)m_b}{L_bM_b}$, a sufficient descent in the objective value of $f$ is ensured.
\end{lemma}

Maximizing the function $\theta(\alpha) = \alpha\left( (1+\beta)m -L M \alpha \right)$ with respect to $\alpha$ yields the optimal stepsize $\alpha^* = \frac{(1+\beta)m}{2LM}$, which gives the following convergence results.
\begin{prop}\label{prop:subsequence}
	Let $\{\x^k\}$ be the sequence generated by Algorithm~\ref{alg:B2B} that is assumed to be bounded. Set $\alpha^k\equiv\alpha$, where $0 < \alpha \leq \frac{(1+\beta)m}{2LM}$. Then the following assertions hold:
	\begin{itemize}
		\item[(\emph{i})] The sequence $\{f(\x^k)\}$ is nonincreasing, and satisfies $\forall k$, $
		f(\x^{k+1})\leq f(\x^k)-\frac{LM}{2} \norm{\x^{k+1} -\x^k}^2.\label{prop:subsequence inequality}
		$
		\item[(\emph{ii})] $\sum_{k=0}^{\infty}  \norm{\x^{k+1} -\x^k}^2 < \infty$, and hence the sequence $\left\{ \norm{\x^{k+1} -\x^k}\right\}$ converges to zero.
		\item[(\emph{iii})]
		$
		\min_{0\leq k\leq N} \norm{\x^{k+1} - \x^k}^2 \leq \frac{2(f(\x^0) - f^*)}{(N+1) LM}, \forall N\geq 0.
		$
	\end{itemize}
\end{prop}
%In the case of PG or BPG methods, we update the whole coordinates together, and the reference function $h$ is chosen to be the energy function (with $m=M=\beta = 1$ and $0 < \alpha^k \leq 1/L$). Using Proposition~\ref{prop:subsequence}, we immediately recover the classical convergence result
%\begin{align}
%\min_{0\leq k\leq N} \norm{G(\x^k)}^2
%%\leq \min_{0\leq k\leq N-1} \frac{\norm{\x^{k+1} - \x^k}^2}{(\alpha^k)^2}
%\leq \frac{2(f(\x^0) - f^*)}{(N+1) (\alpha^k)^2 L}.
%\end{align}
%Therefore, the optimality $\norm{G(\x^k)}$ convergence to zero at the rate of $\bigo{1/\sqrt{N}}$.

To establish the convergence rate of $\{\x^k\}$, the main idea is to show $\nabla^P f(\x^k)$ (or $\partial F(\x^k)$) can be upper bounded in each iteration, and those upper bounds converge to zero. Proposition~\ref{prop:subsequence}(ii) implies $\{\norm{\x^{k+1} - \x^k}\}$ converges to zero. As we discussed in the previous section, however, $\norm{\x^{k+1} - \x^k}$ cannot be used to bound $\nabla^P f(\x^k)$ to obtain an asymptotic convergence rate, because Algorithm~\ref{alg:B2B} only selects one block at a time, while all blocks are required to satisfy the conditions in Lemma~\ref{lemma:critical point} (\emph{iii}). The result is not easy to obtained, even we use the cyclic rule since B2B method uses two different reference functions in \eqref{opt:two metric mapping}-\eqref{opt:two metric projection}. In the following, we show the convergence results for B2B by leveraging the greedy and randomized rules, which can further improve the convergence rate of B2B with the cyclic rule by one order in terms of the number of blocks.

\subsection{Greedy and randomized rule}
Without loss generality, we assume that each bock only contains valid coordinates. For the Gauss-Southwell (G-So) or greedy rule, a block is selected in the $k$-th iteration if it has the maximum magnitude of the partial gradient, \ie,
\begin{align}
b_k = \argmax_{1\leq b\leq s} \{\norm{\nabla_b f(\x^k) }\}.\label{eq:greedy}
\end{align}
The following proposition establishes the main convergence results for the \textit{greedy B2B} (GB2B) method.
\begin{theorem}[Convergence of greedy B2B]\label{prop:greedy}
	Let $\{\x^k\}$ be the sequence generated by Algorithm~\ref{alg:B2B} with greedy rule and is assumed to be bounded. Set $\alpha^k\equiv\alpha$ with $0 < \alpha \leq \frac{(1+\beta)m}{2LM}$. The following assertions hold:
	\begin{itemize}
		\item[(\emph{i})] For all $k\geq 0$, the projected gradient $\nabla^p f(\x^k)$ satisfies
		$
		f(\x^{k+1})\leq f(\x^k)  - \frac{\alpha^k(1+\beta)m}{4sM}\norm{ \nabla^P f(\x^k) }^2.\label{prop:greedy eq descent}
		$
		\item[(\emph{ii})]
		$
		\min_{0\leq k\leq N} \norm{\nabla^P f(\x^k)}^2 \leq \frac{4sM\left(f(\x^0) - f^*\right)}{(N+1)\alpha(1+\beta) m},
		\forall N\geq 0.
		$
		\item[(\emph{iii})] Every limit point of $\{\x^k\}$ is a critical point.
	\end{itemize}
\end{theorem}
From Theorem~\ref{prop:greedy} (\emph{ii}), it immediately follows that $\norm{\nabla^P f(\x^k)}$ converges to zero at the rate of $\bigo{\sqrt{s}/\sqrt{k}}$.

In the randomized rule, a block is selected uniformly at random. We use $\E_{b_k}$ to denote the expectation with respect to a single random index $b_k$. We use $\E$ to denote the expectation with respect to all random variables $\{b_0, b_1, \cdots\}$. The following proposition establishes the main convergence result for the \textit{randomized variant of B2B} (RB2B).

\begin{theorem}[Convergence of randomized B2B]\label{prop:randomized}
	Let $\{\x^k\}$ be the sequence generated by Algorithm~\ref{alg:B2B} with randomized rule and is assumed to be bounded. Set $\alpha^k=\alpha$, where $0 < \alpha \leq \frac{(1+\beta)m}{2LM}$. The following assertions hold:
	\begin{itemize}
		\item[(\emph{i})] For all $k\geq 0$, the projected gradient $\nabla^p f(\x^k)$ satisfies
		$
		\E_{b_k}f(\x^{k+1})\leq f(\x^k)  - \frac{\alpha^k(1+\beta)m}{4sM}\norm{ \nabla^P f(\x^k) }^2
		\label{prop:randomized eq descent}
		$
		\item[(\emph{ii})]
		$
		\min_{0\leq k\leq N} \E \norm{\nabla^P f(\x^k)}^2 \leq \frac{4sM\left(f(\x^0) - f^*\right)}{(N+1)\alpha(1+\beta) m}, \; \forall N\geq0.
		$
		\item[(\emph{iii})] Every limit point of $\{\x^k\}$ is a critical point.
	\end{itemize}
\end{theorem}
Similar to the GB2B method, we also obtain the convergence rate $\bigo{\sqrt{s}/\sqrt{k}}$ for the RB2B method. Both of these two methods are $\bigo{\sqrt{s}}$ times faster than the CBBCD method, \ie, $\bigo{s/\sqrt{k}}$. Moreover, the randomized variant is even more efficient than the greedy variant, since performing \eqref{eq:greedy} in GB2B needs to search all blocks to determine the desired block, while the randomized method selects a block randomly.

\subsection{Global convergence}
In this subsection, we establish the global convergence of the B2B method. For this purpose, we outline three ingredients of the methodology \cite{bolte2014proximal,bolte2018first}, which has broad range of applications.

\begin{define}\cite[Definiton~4.1]{bolte2018first}
	A sequence $\{\x^k\}$ is called a gradient-like descent sequence for $F$ if the following three conditions holds.
	\begin{itemize}
		\item[(\emph{i})] \textit{Sufficient decrease property}: There exists a scalar $\rho_1 > 0$ such that for $k\geq 0$
		\begin{align}
			\rho_1\norm{\x^{k+1} - \x^k}\leq F(\x^{k} ) - F(\x^{k+1}).
		\end{align}
		\item[(\emph{ii})] A subgradient lower bound for the iterate gap: There exists another scaler $\rho_2 > 0$ such that for $k\geq 0$
		\begin{align}
			\norm{\w^{k+1}}\leq \rho_2 \norm{\x^{k+1} - \x^k}
		\end{align}
		for some $\w^{k+1}\in \partial F(\x^{k+1})$.
		\item[(\emph{iii})] Let $\overline{\x}$ be a limit point of a subsequence $\{\x^{k_q}\}$, then $\limsup_{q\rightarrow \infty} F(\x^{k_q})\leq F(\overline{\x})$.
	\end{itemize}
\end{define}
Clearly, it follows from Proposition~\ref{prop:subsequence} (\emph{ii}) that the sufficient descent property is obtained. Combining Proposition~\ref{prop:subgradient equiavlent} with Theorem~\ref{prop:greedy} (\emph{ii}) or Theorem~\ref{prop:randomized} (\emph{ii}) implies the subgradient bound property. Since $f$ is continuously differentiable and $r$ is an indicator function of $\reals^n_+$, the third continuity condition holds trivially. Together with the Kurdyka-\L ojasiewicz (KL) property (see \cite{bolte2014proximal} and supplemental for details), we can prove the following theorem for the GB2B and RB2B methods.

\begin{theorem}[Global convergence]\label{thm:Global convergence}
	Let $\{\x^k\}$ be the sequence generated by Algorithm~\ref{alg:B2B} and is assumed to be bounded. Then the sequence $\{\x^k\}$ converges to a critical point of $f$.
\end{theorem}

\section{Applications and Numerical Experiments}
\begin{table*}
	\resizebox{\textwidth}{!}{
		\centering
		\begin{tabular}{|c|| r|r|r|r|r|r|r || r|r|r|r|r|r|r|}
			\hline
			\multirow{ 2}{*}{Dataset} &	\multicolumn{7}{c||}{Time(seconds)}&\multicolumn{7}{c|}{Number of iterations}\\  \cline{2-15}
			& GCD & FastHALS &  ANLSPivot& AOADMM&APG &GB2B & RB2B  & GCD & FastHALS &  ANLSPivot & AOADMM&APG &GB2B & RB2B \\ \hline
			ORL & 5.192& 19.267& 2.031& 3.101& 5.461& \textbf{1.592}& 11.130&
			189& 1000& 21& 22& 297& 76& 542\\ \hline
			COIL & 357.959& 600.970& 165.423& 281.719& 265.440& \textbf{67.315}& 334.971&
			493& 862& 88& 72& 409& 84& 478\\ \hline
			%			UMist & \textbf{16.7748}& 90.2561& 28.2418& 29.2847& 38.6203& 21.0996& 53.0146&
			%			112& 651& 80& 47& 276& 142& 363\\ \hline
			YaleB & 90.755& 78.008& 15.324& 24.543& 34.046& \textbf{11.178}& 88.798&
			1000& 911& 68& 83&406 & 119& 1000\\ \hline \hline
			News20 & 2.791& \textbf{2.648}& 25.273& 43.997& 4.319& 3.083& 4.730&
			72& 69& 54& 75& 123& 50& 96\\ \hline
			MNIST & 14.104& 11.732& 72.078& 214.376& 15.547& \textbf{8.159}& 42.818&
			119& 115& 136& 169& 156& 66& 383\\ \hline
			TDT2 & 14.608& \textbf{3.804}& 67.467& 344.238& 6.069& 27.833& 11.739&
			120& 28& 68& 81& 52& 128&60\\ \hline
			%			Reuters & 14.7146& \textbf{3.67342}& 99.753& 1008.31& 6.459& 42.1466& 5.9439
			%			& 143& 31& 137&183 &73 & 130& 40\\ \hline
		\end{tabular}
	}
	\caption{Performance comparison for algorithms on real datasets, where the fastest algorithm is highlighted for each dataset. GB2B is in general faster than the other algorithms. RB2B is theoretically faster than GB2B in terms of computation, but it is not the case in practice.}\label{tab:result}
\end{table*}
\begin{figure*}
	\centering
	\begin{minipage}{.24\linewidth}
		\centerline{\includegraphics[width=\columnwidth]{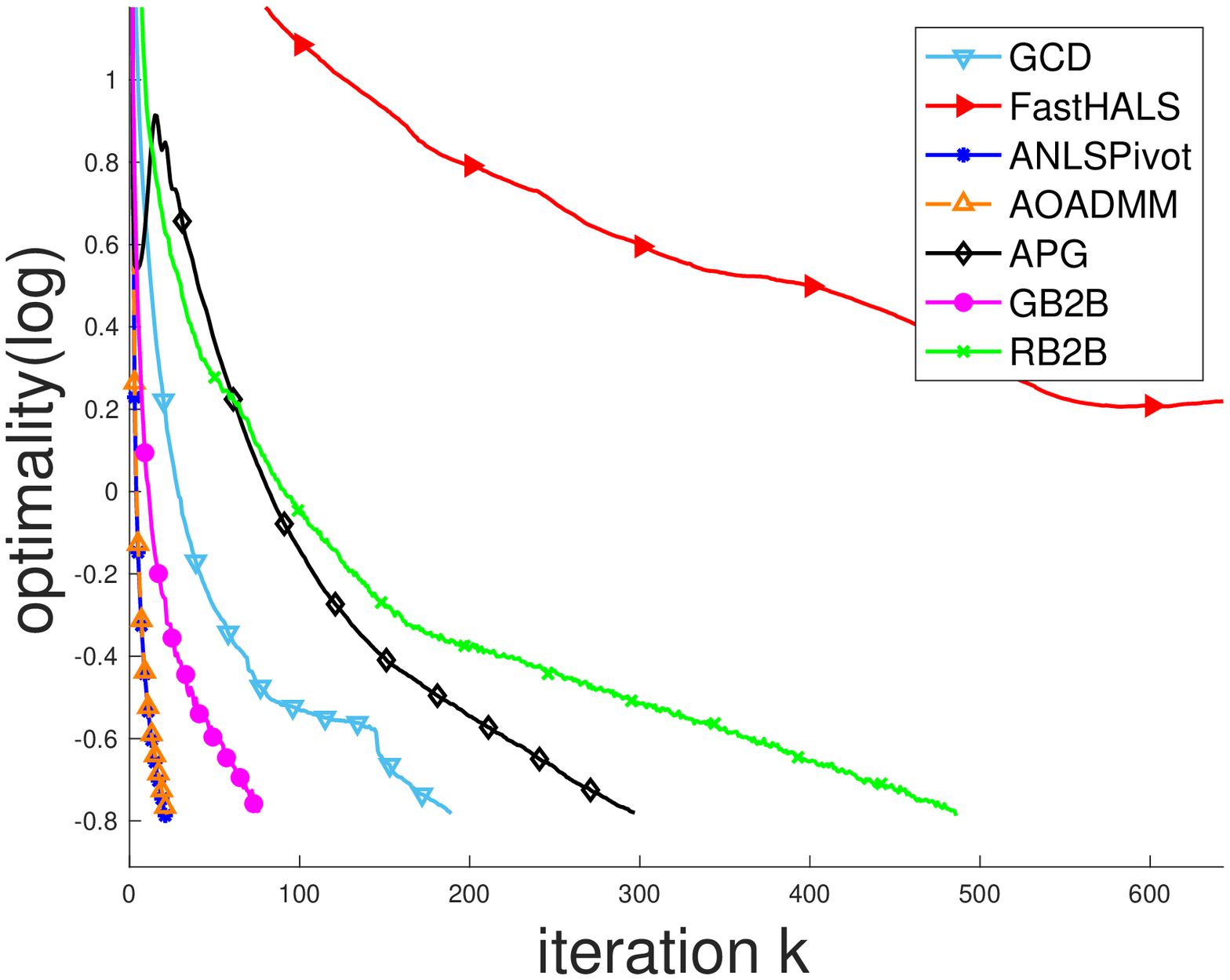}}
		%		\caption{Objective reduction}
		\caption*{(a) Optimality versus iteration.}
		%		\label{fig:obs}
	\end{minipage}
	\begin{minipage}{.24\linewidth}
		\centerline{\includegraphics[width=\columnwidth]{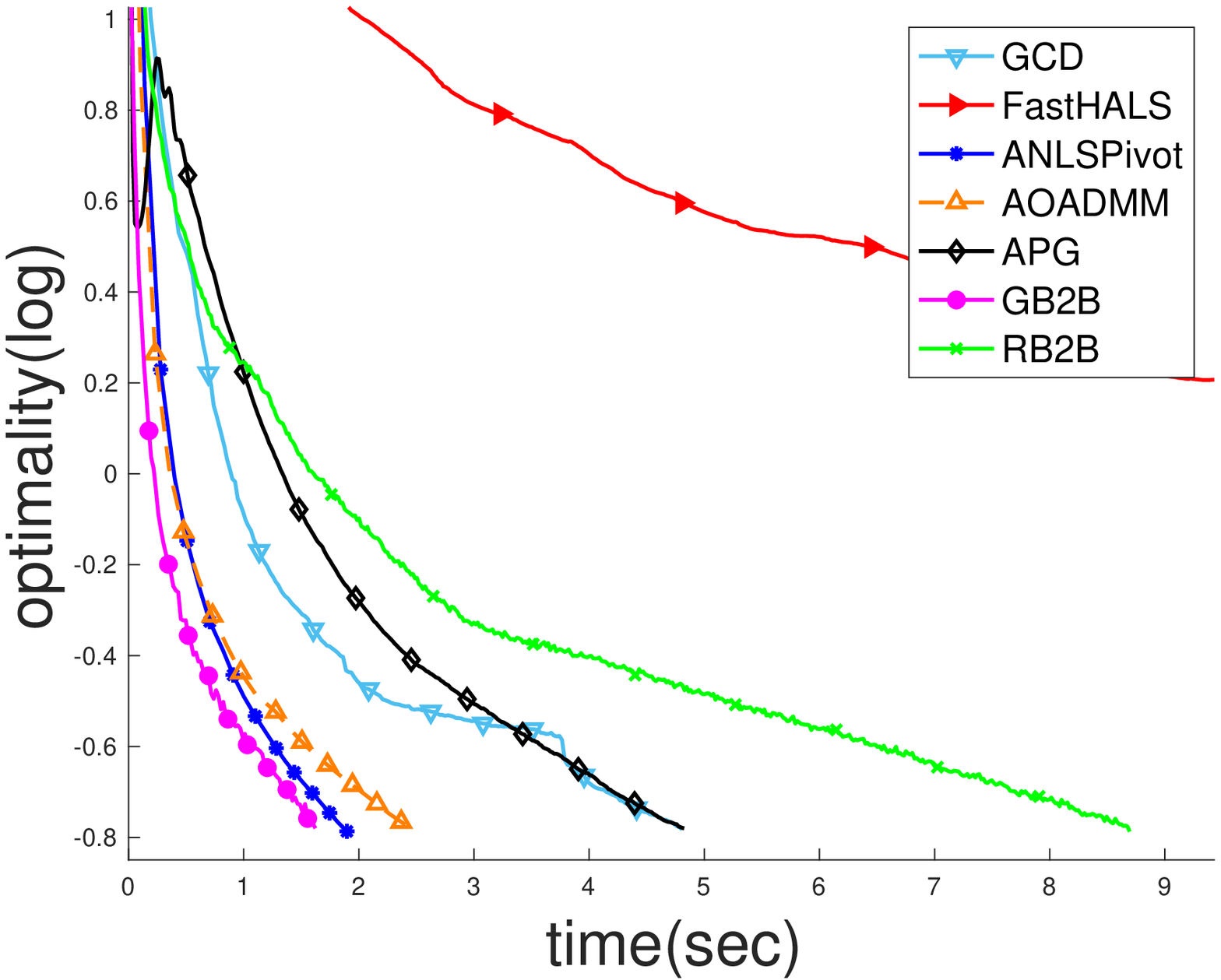}}
		%		\caption{Objective reduction}
		\caption*{(b) Optimality versus runtime.}
		%		\label{fig:obs}
	\end{minipage}
	\begin{minipage}{.24\linewidth}
		\includegraphics[width=\columnwidth]{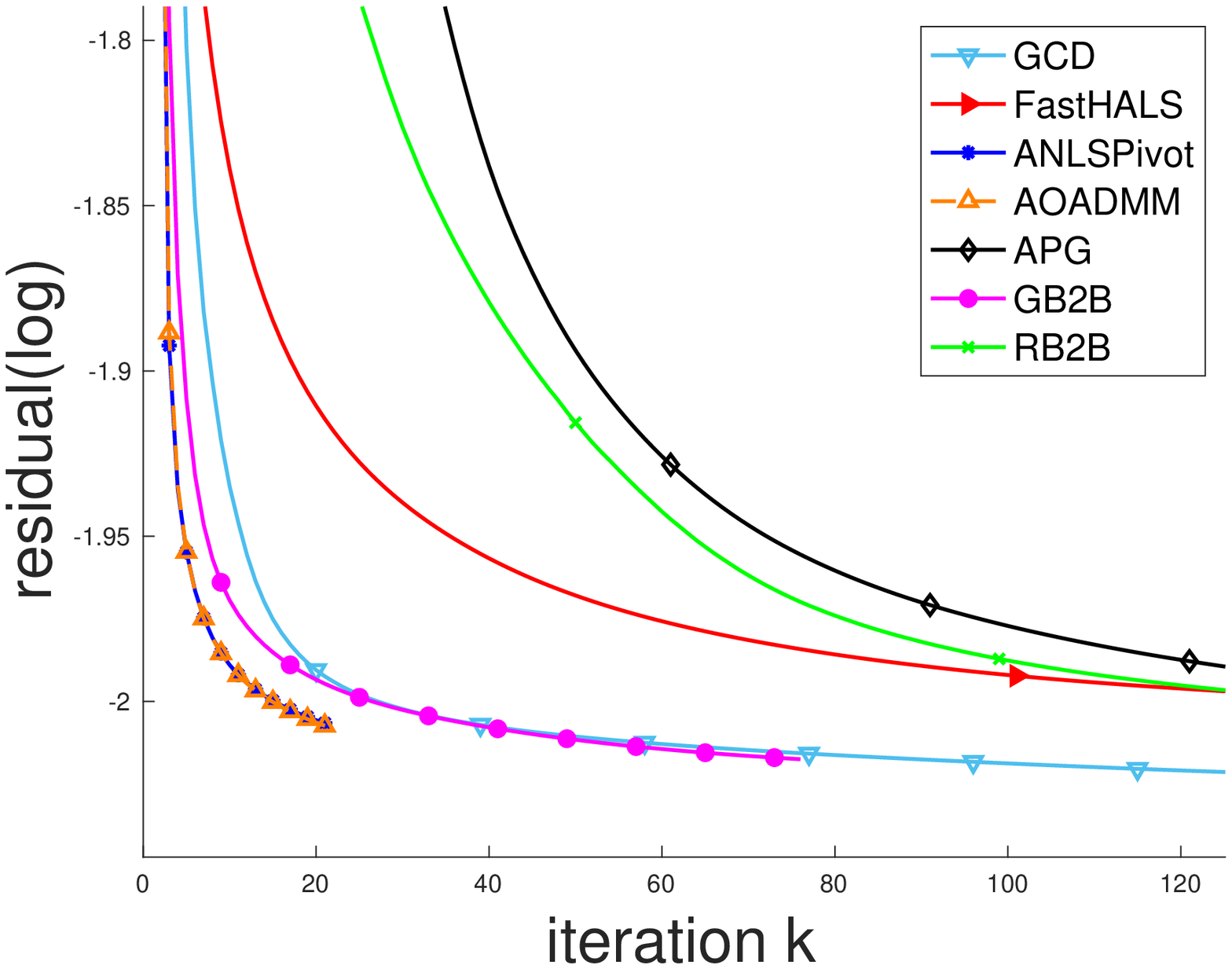}
		%		\caption{Runtime histogram}
		\caption*{(c) Residual versus iteration.}
		%		\label{fig:chol}
	\end{minipage}%
	\begin{minipage}{.24\linewidth}
		\centerline{\includegraphics[width=\columnwidth]{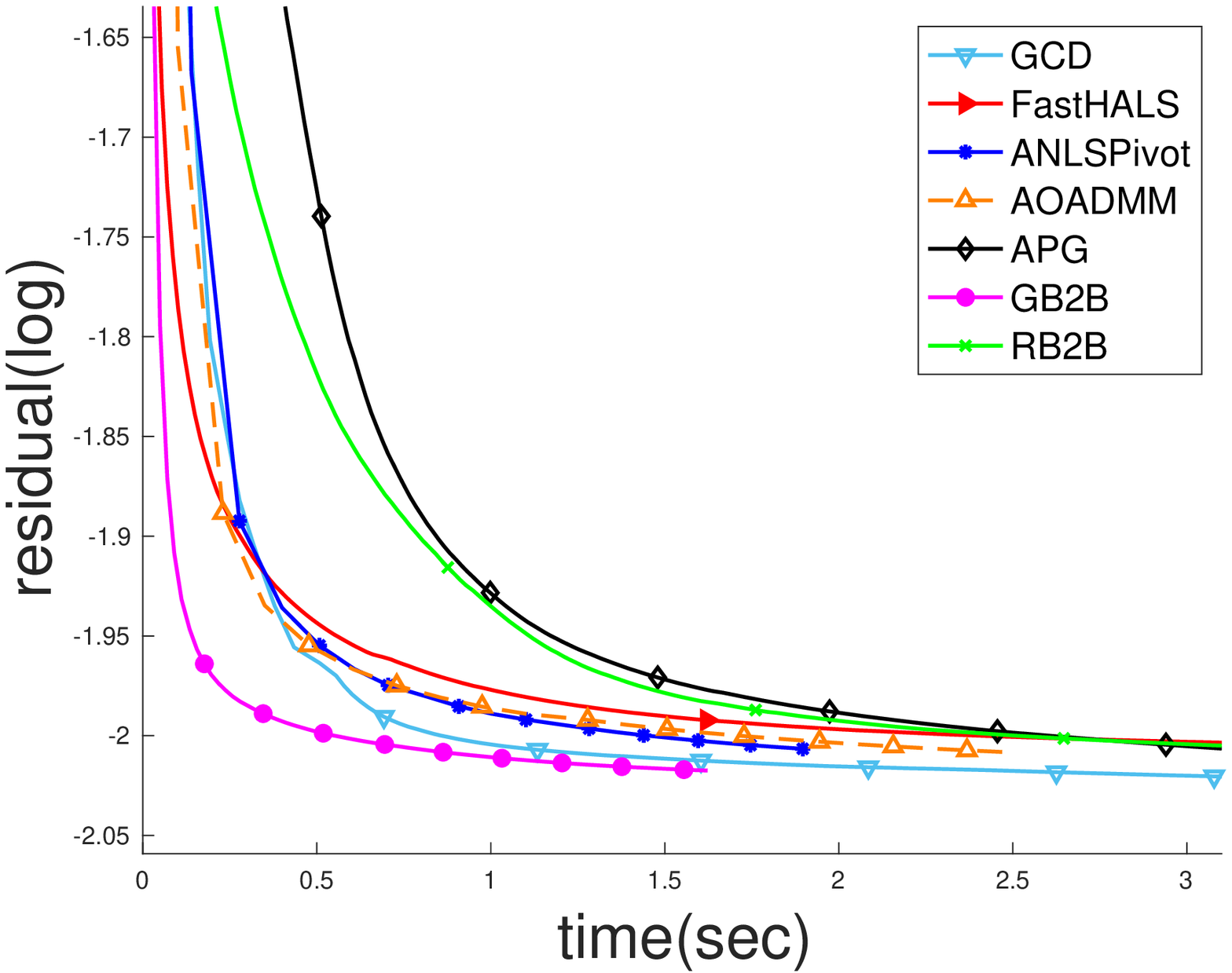}}
		%		\caption{Size of working set}
		\caption*{(d) Residual versus runtime.}
		%		\label{fig: residual exact}
	\end{minipage}
	\caption{Convergence behaviors of different algorithms on \textbf{ORL}: (a)-(b) illustrate the changes in optimality versus iterations and runtime; (c)-(d) illustrate the changes in the residual versus iterations and runtime.}
	\label{fig:orl}
\end{figure*}
\begin{figure*}
	\centering
	\begin{minipage}{.24\linewidth}
		\centerline{\includegraphics[width=\columnwidth]{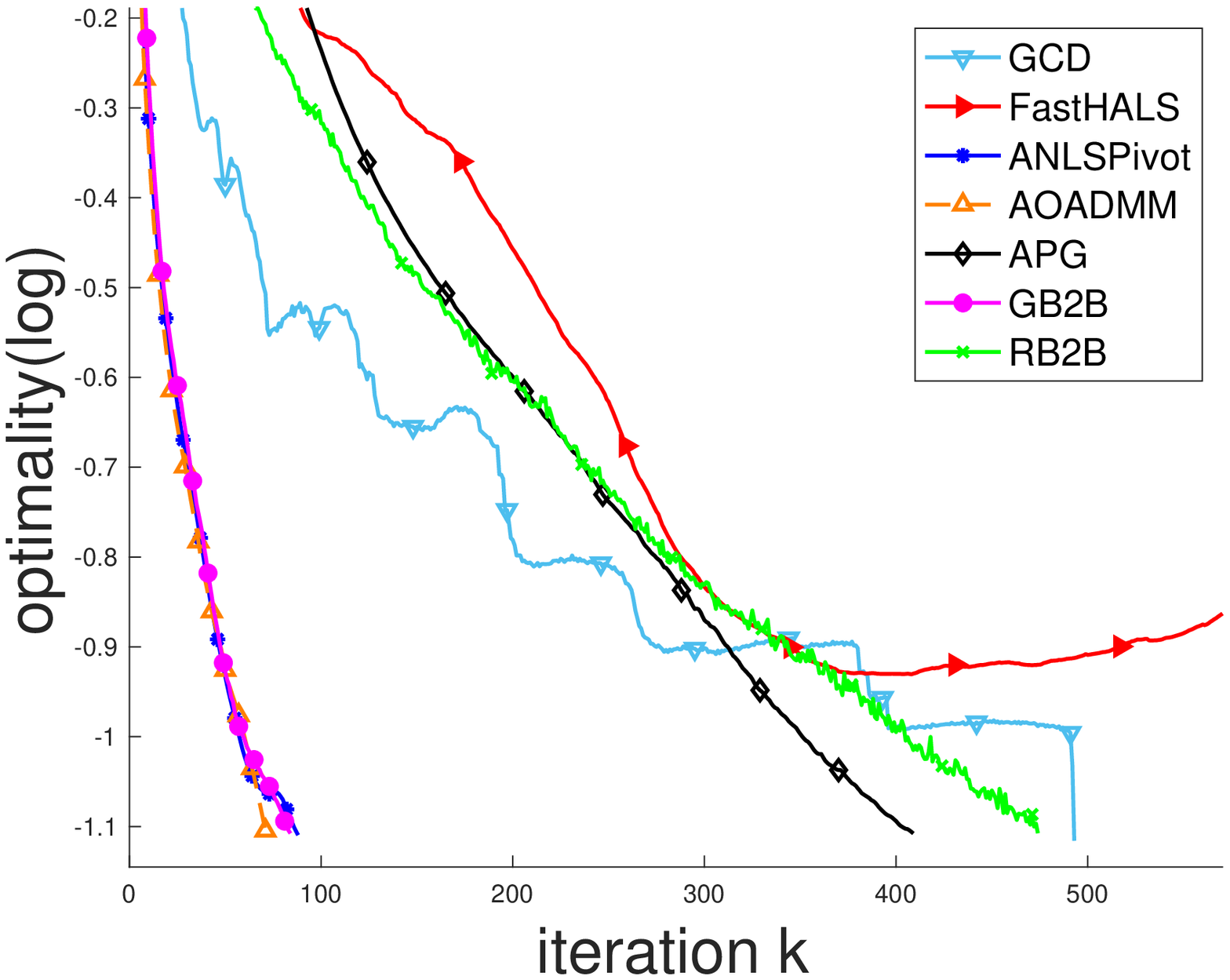}}
		%		\caption{Objective reduction}
		\caption*{(a) Optimality versus iteration.}
		%		\label{fig:obs}
	\end{minipage}
	\begin{minipage}{.24\linewidth}
		\centerline{\includegraphics[width=\columnwidth]{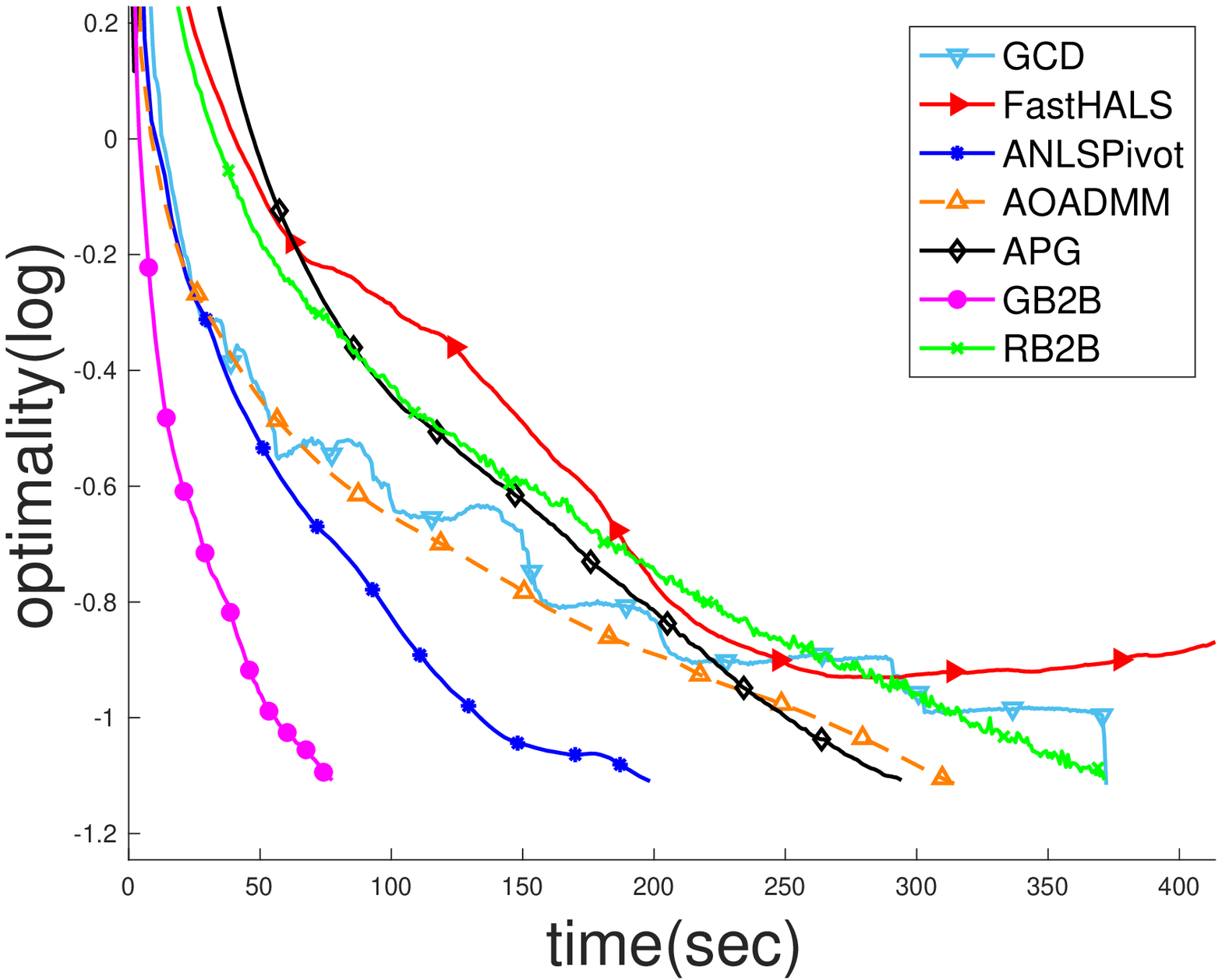}}
		%		\caption{Objective reduction}
		\caption*{(b) Optimality versus runtime.}
		%		\label{fig:obs}
	\end{minipage}
	\begin{minipage}{.24\linewidth}
		\includegraphics[width=\columnwidth]{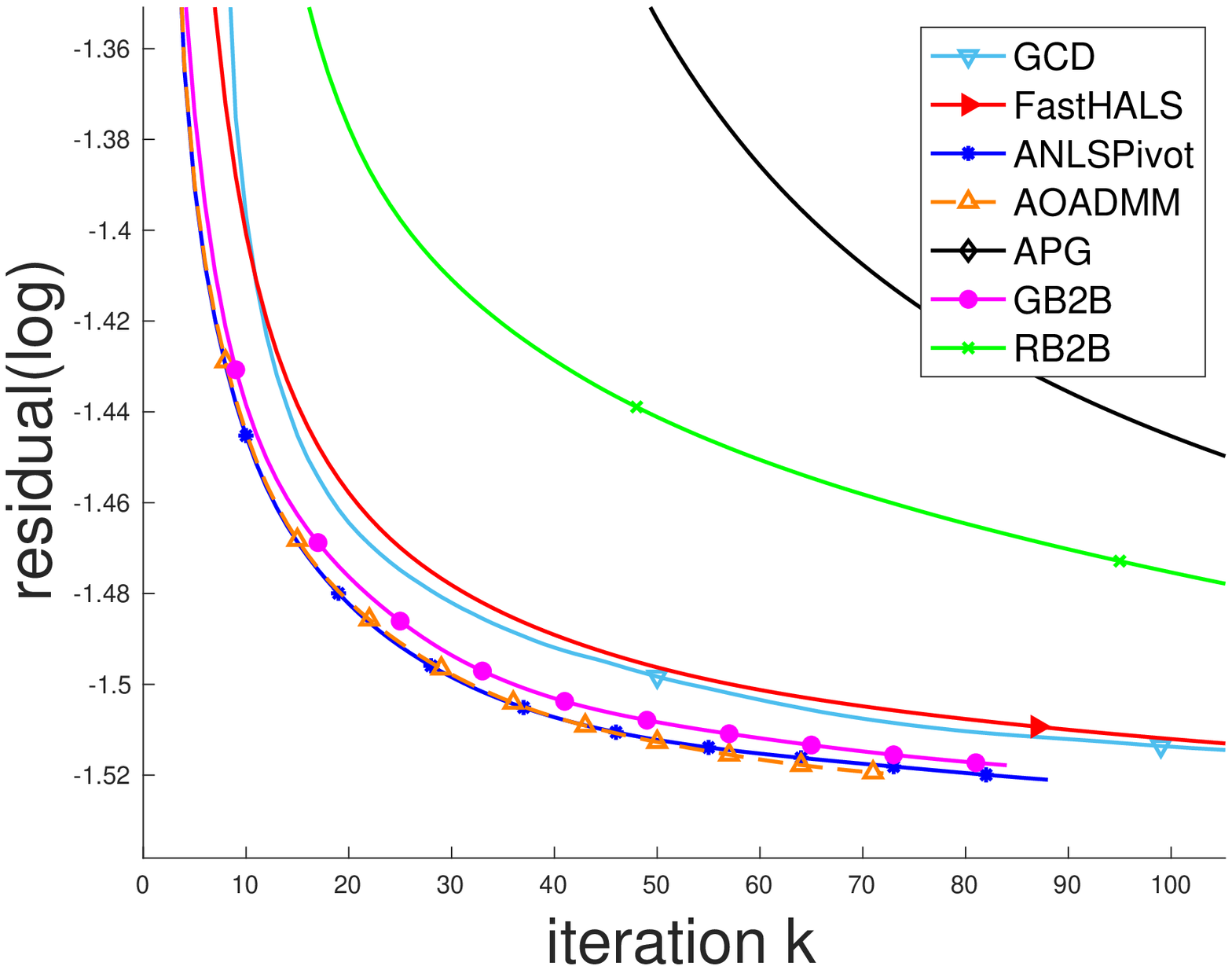}
		%		\caption{Runtime histogram}
		\caption*{(c) Residual versus iterations.}
		%		\label{fig:chol}
	\end{minipage}%
	\begin{minipage}{.24\linewidth}
		\centerline{\includegraphics[width=\columnwidth]{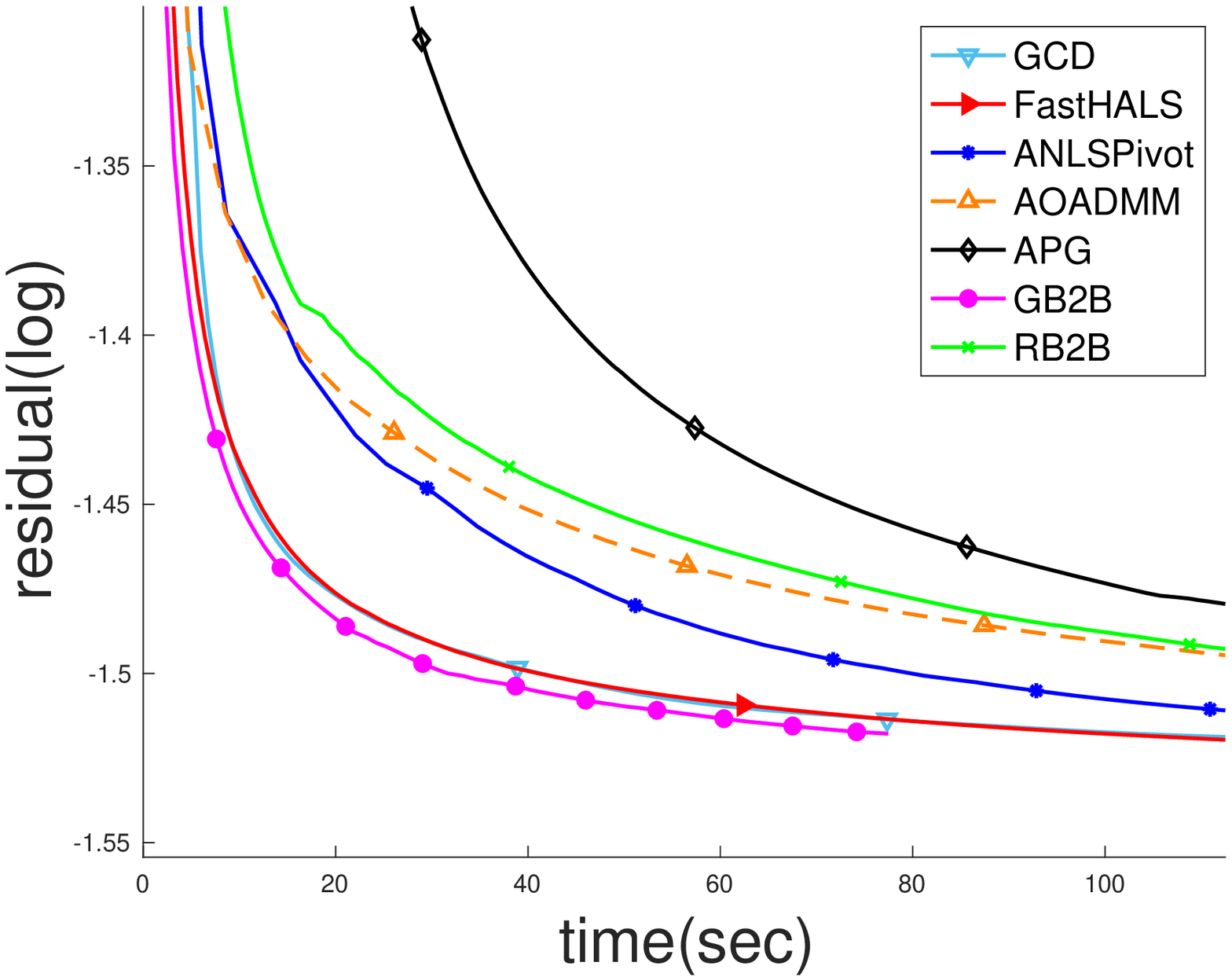}}
		%		\caption{Size of working set}
		\caption*{(d) Residual versus runtime.}
		%		\label{fig: residual exact}
	\end{minipage}
	\caption{Convergence behaviors of different algorithms on \textbf{COIL}: (a)-(b) illustrate the changes in optimality versus iterations and runtime; (c)-(d) illustrate the changes in the residual versus iterations and runtime.}
	\label{fig:coil}
\end{figure*}
\begin{figure*}
	\centering
	\begin{minipage}{.24\linewidth}
		\centerline{\includegraphics[width=\columnwidth]{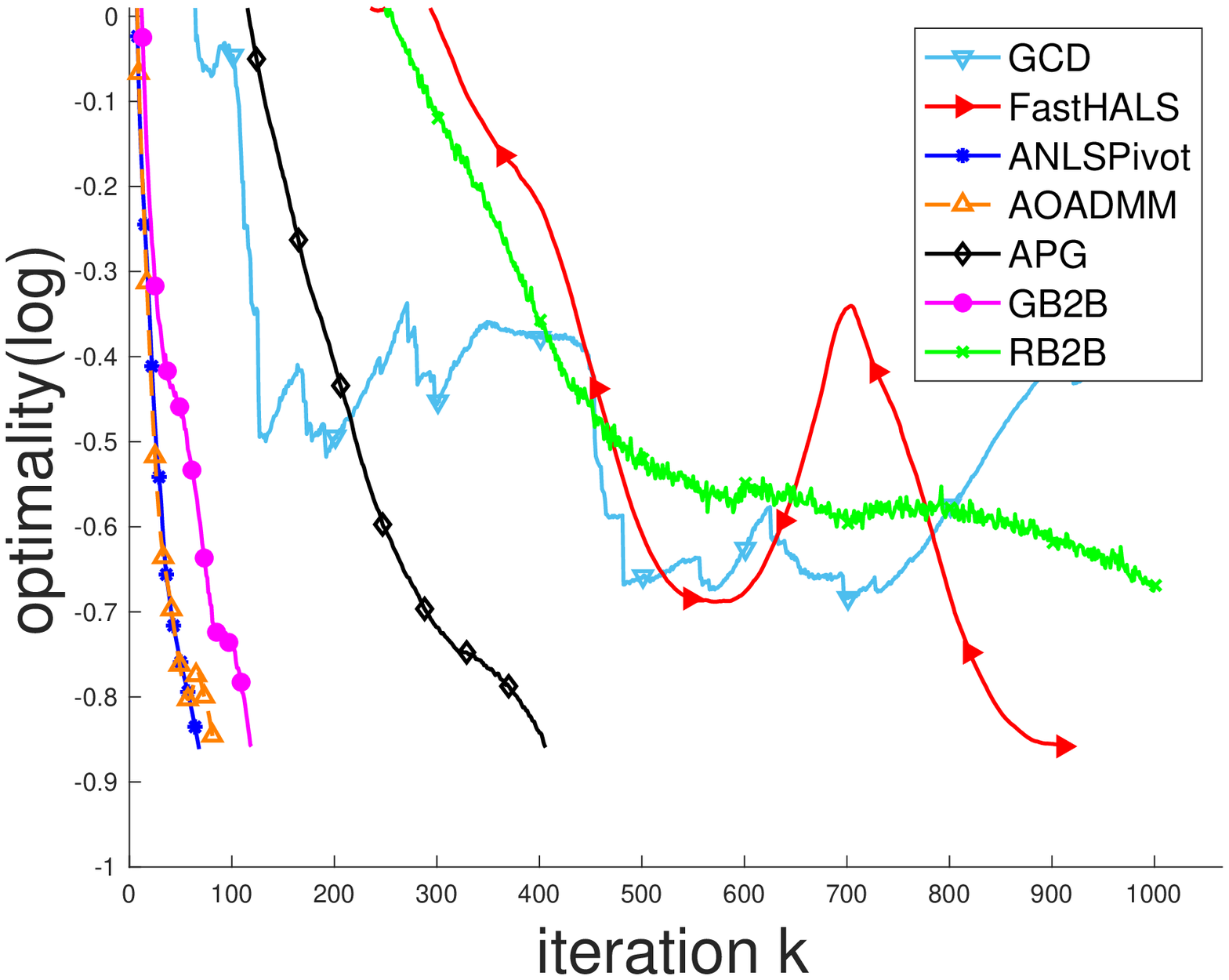}}
		%		\caption{Objective reduction}
		\caption*{(a) Optimality versus iteration.}
		%		\label{fig:obs}
	\end{minipage}
	\begin{minipage}{.24\linewidth}
		\centerline{\includegraphics[width=\columnwidth]{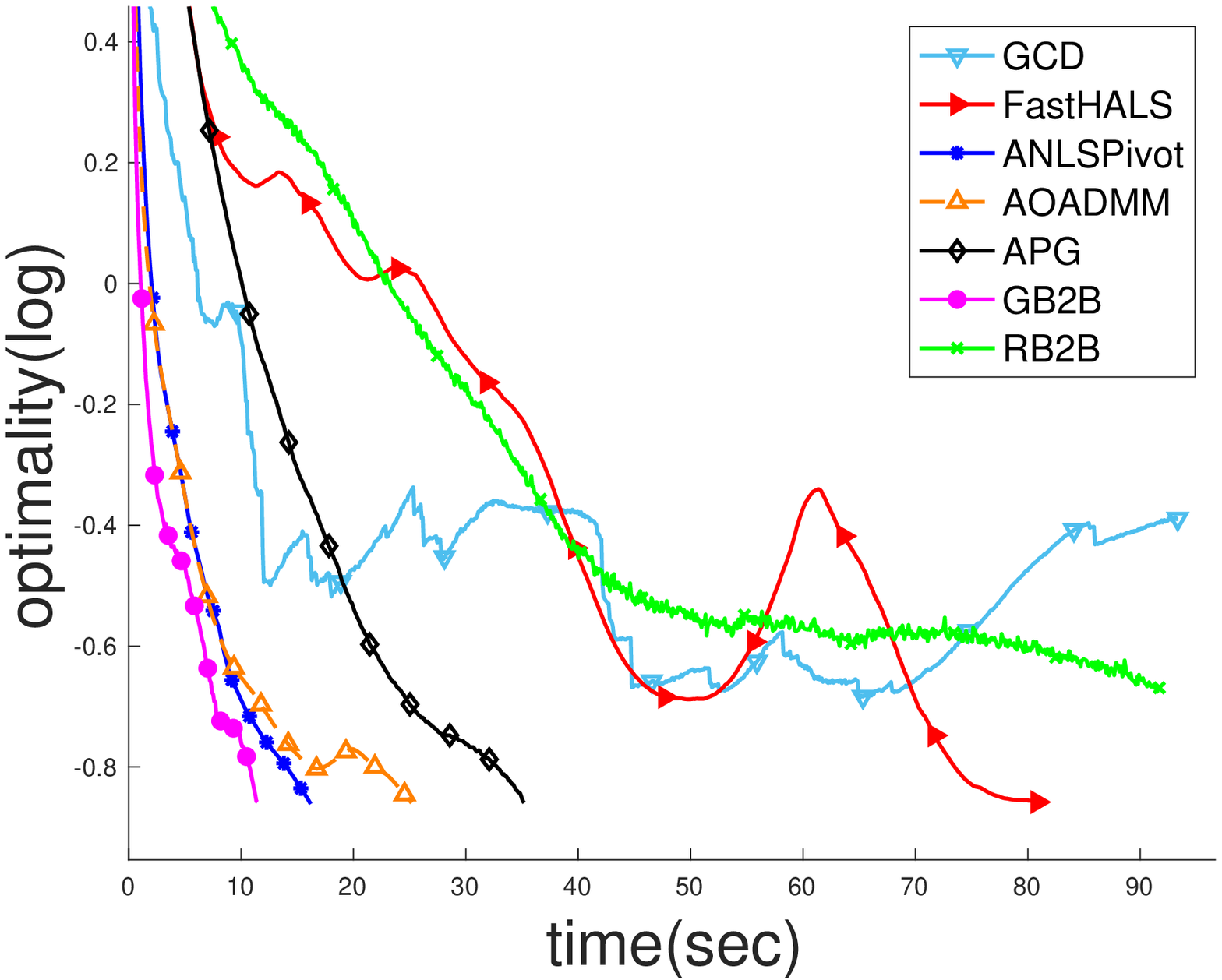}}
		%		\caption{Objective reduction}
		\caption*{(b) Optimality versus runtime.}
		%		\label{fig:obs}
	\end{minipage}
	\begin{minipage}{.24\linewidth}
		\includegraphics[width=\columnwidth]{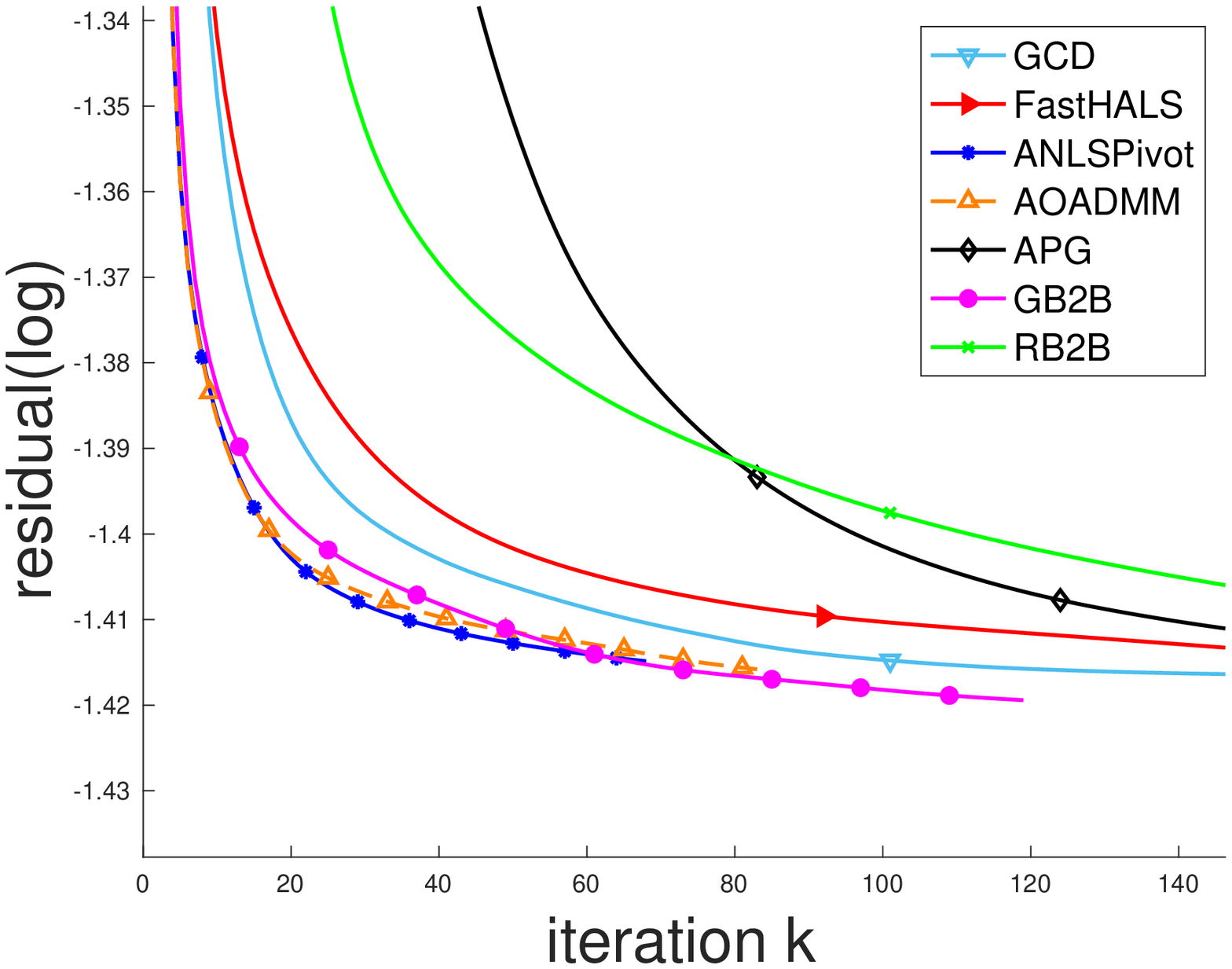}
		%		\caption{Runtime histogram}
		\caption*{(c) Residual versus iteration.}
		%		\label{fig:chol}
	\end{minipage}%
	\begin{minipage}{.24\linewidth}
		\centerline{\includegraphics[width=\columnwidth]{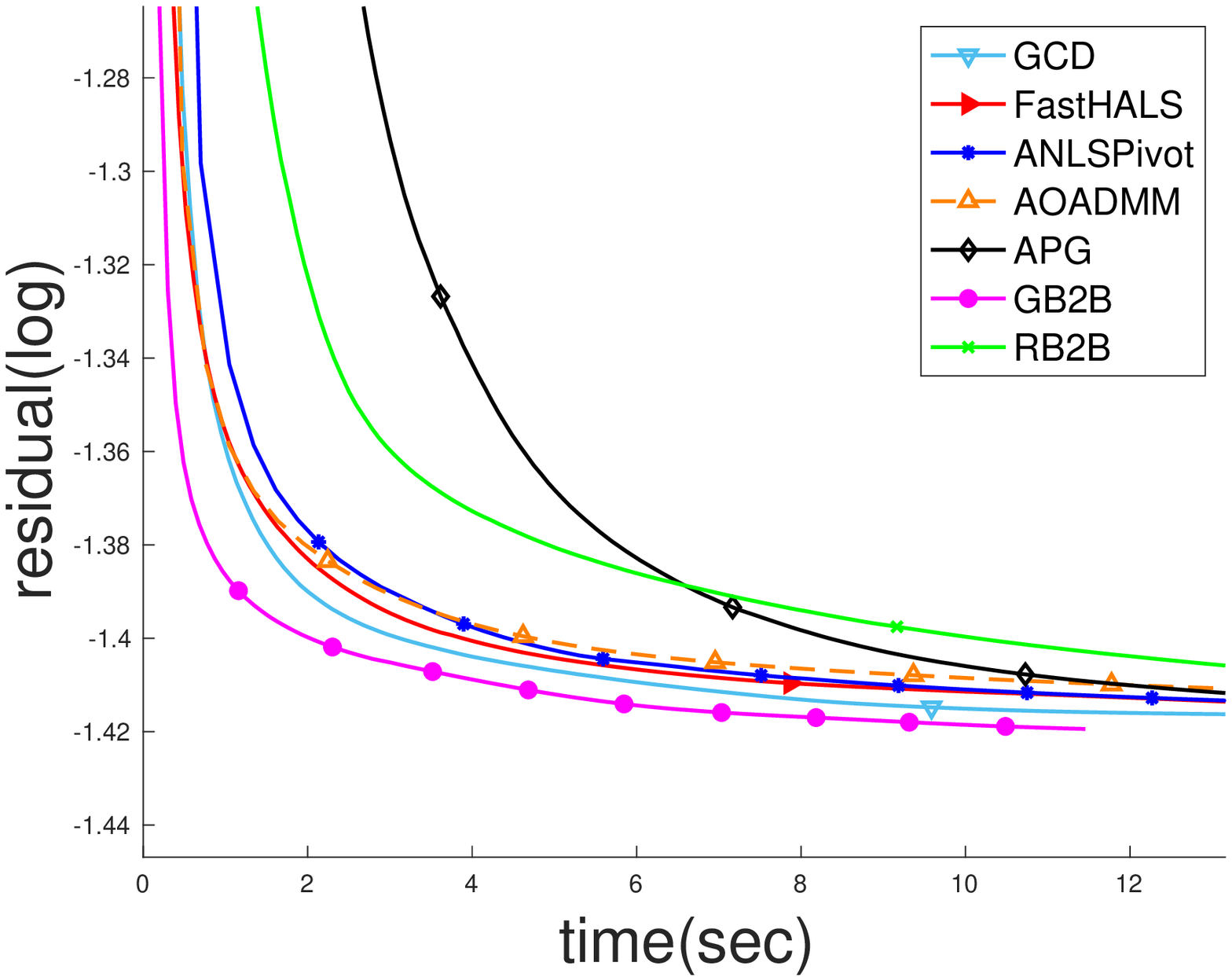}}
		%		\caption{Size of working set}
		\caption*{(d) Residual versus runtime.}
		%		\label{fig: residual exact}
	\end{minipage}
	\caption{Convergence behaviors of different algorithms on \textbf{YaleB}: (a)-(b) illustrate the changes in optimality versus iterations and runtime; (c)-(d) illustrate the changes in the residual versus iterations and runtime.}
	\label{fig:yaleb}
\end{figure*}
\begin{figure*}
	\centering
	\begin{minipage}{.24\linewidth}
		\centerline{\includegraphics[width=\columnwidth]{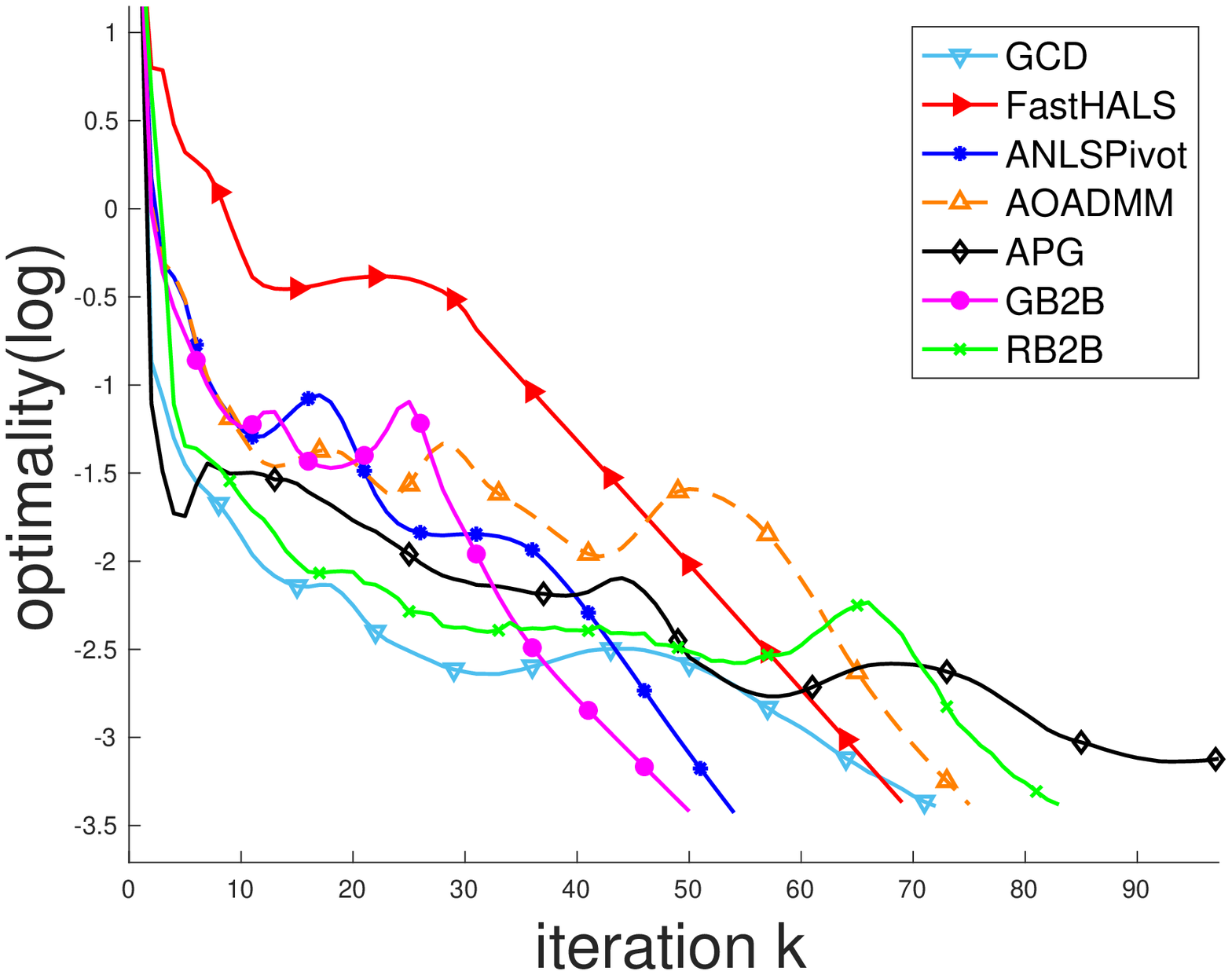}}
		%		\caption{Objective reduction}
		\caption*{(a) Optimality versus iteration.}
		%		\label{fig:obs}
	\end{minipage}
	\begin{minipage}{.24\linewidth}
		\centerline{\includegraphics[width=\columnwidth]{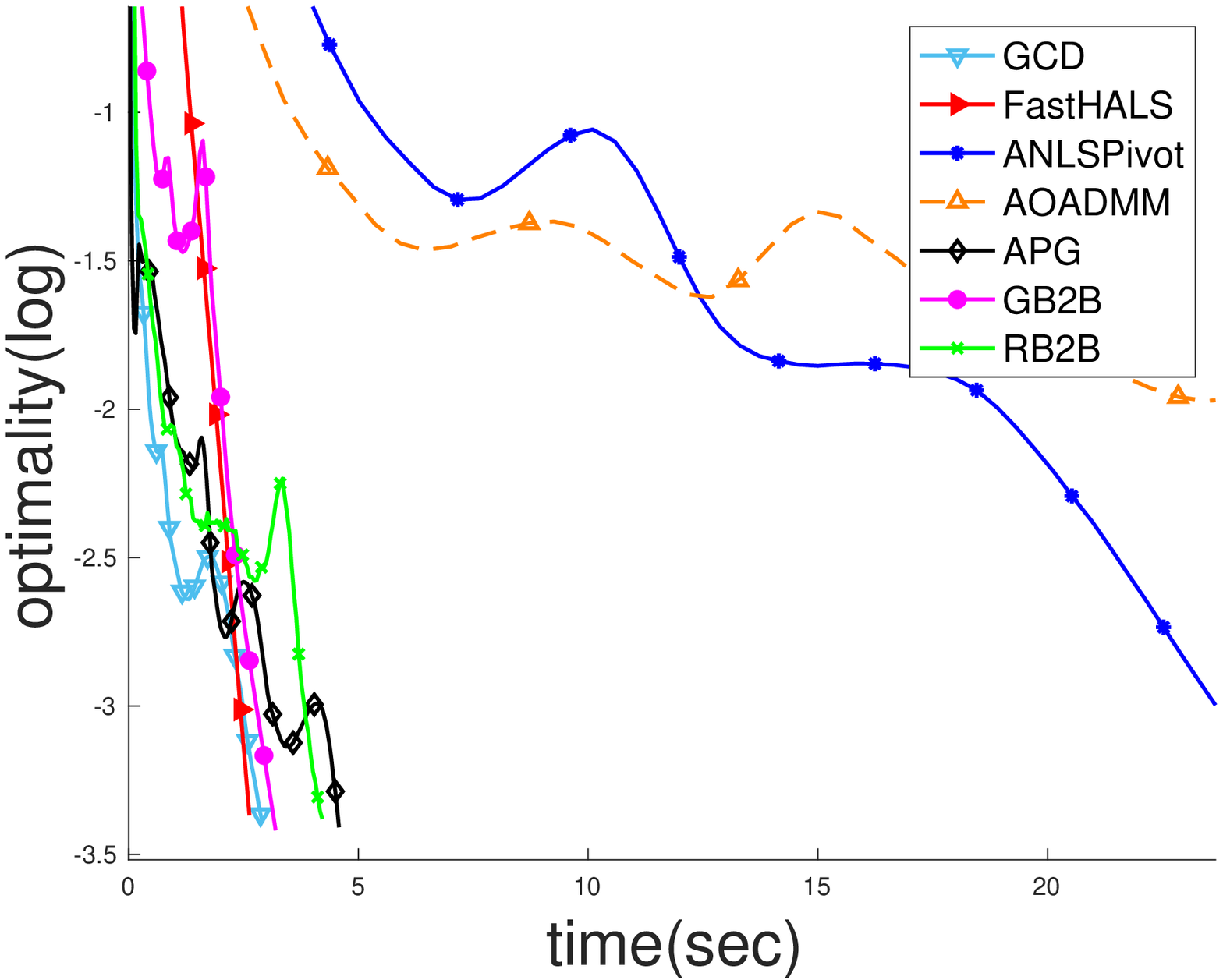}}
		%		\caption{Objective reduction}
		\caption*{(b) Optimality versus runtime.}
		%		\label{fig:obs}
	\end{minipage}
	\begin{minipage}{.24\linewidth}
		\includegraphics[width=\columnwidth]{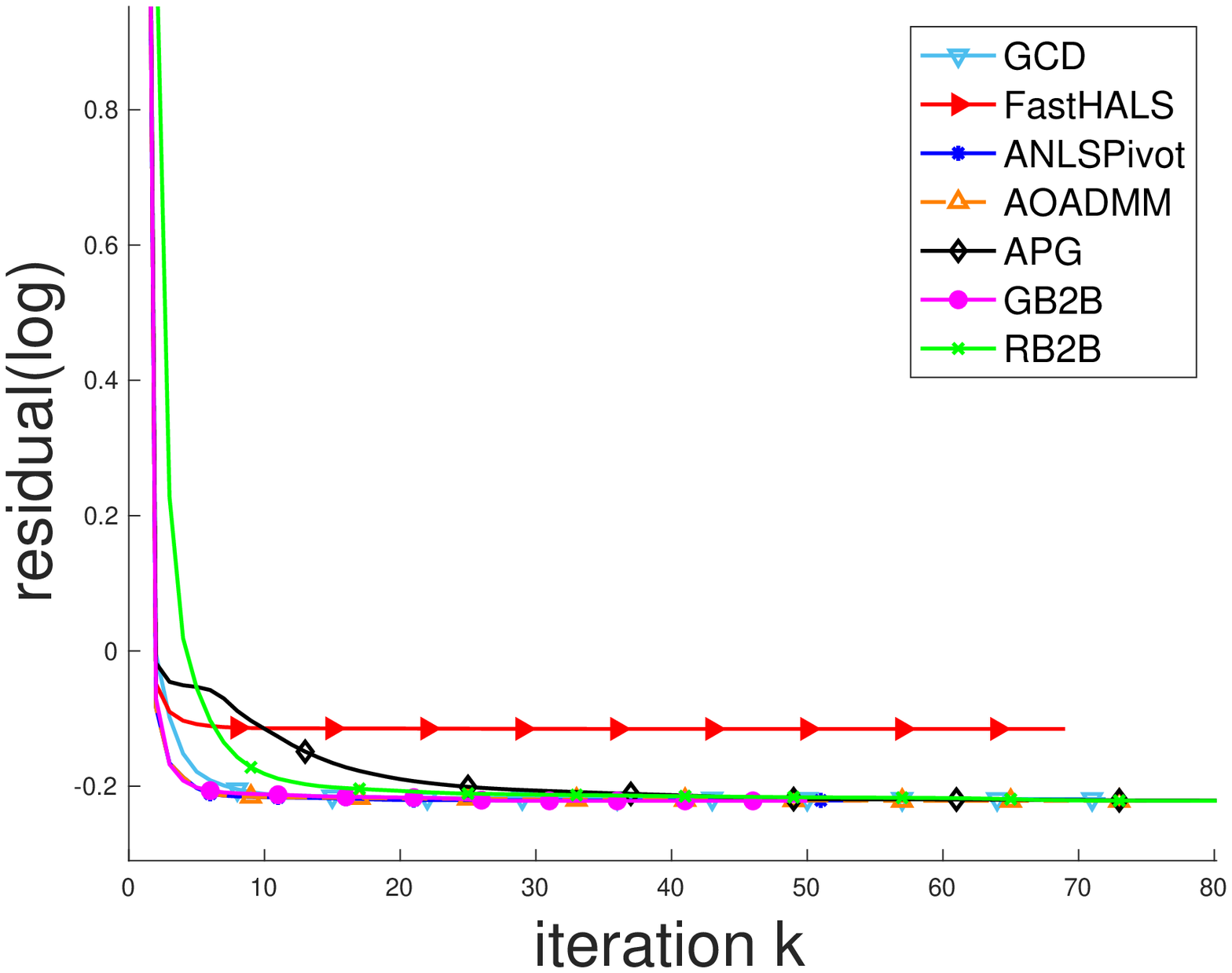}
		%		\caption{Runtime histogram}
		\caption*{(c) Residual versus iterations.}
		%		\label{fig:chol}
	\end{minipage}%
	\begin{minipage}{.24\linewidth}
		\centerline{\includegraphics[width=\columnwidth]{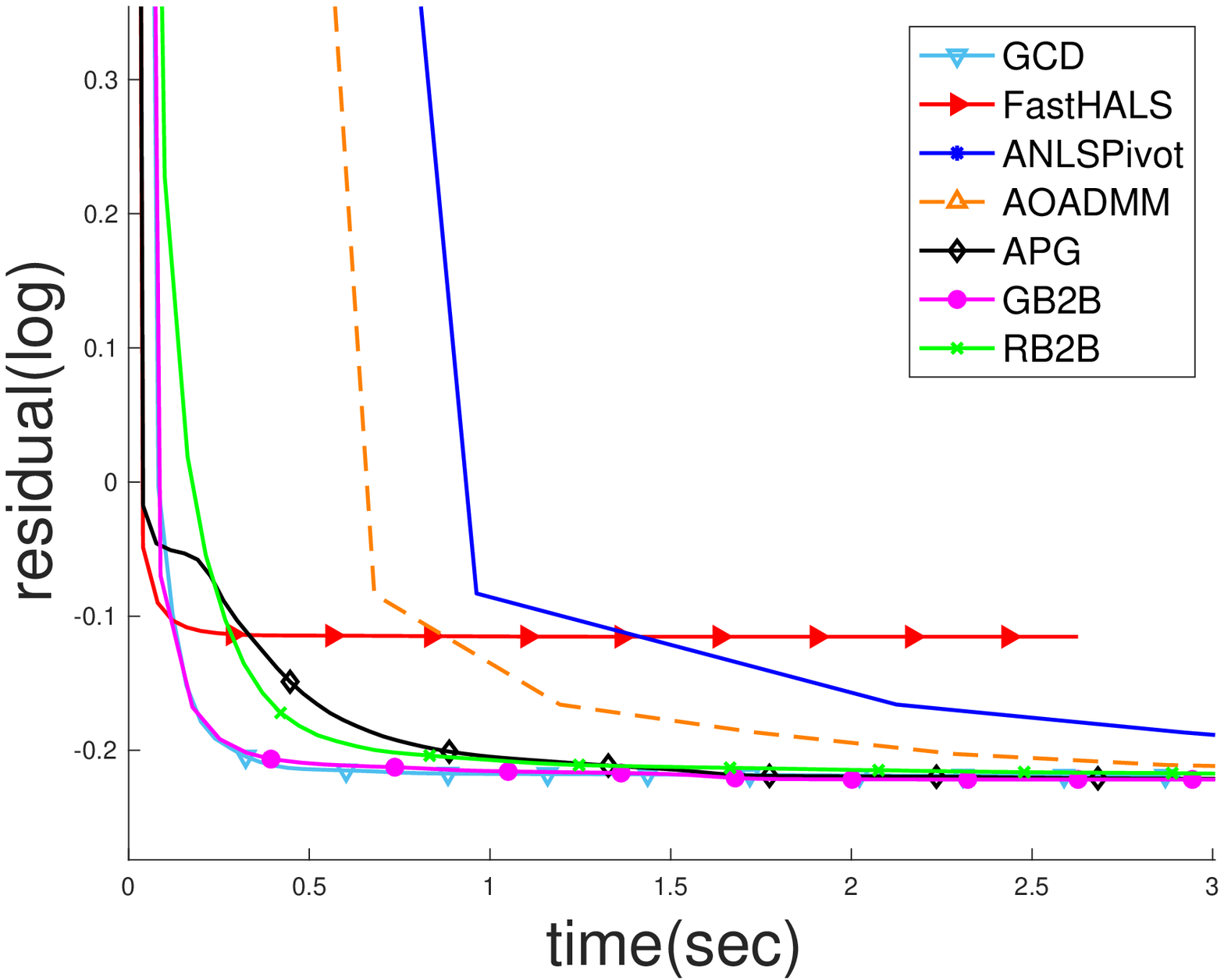}}
		%		\caption{Size of working set}
		\caption*{(d) Residual versus runtime.}
		%		\label{fig: residual exact}
	\end{minipage}
	\caption{Convergence behaviors of different algorithms on \textbf{News20}: (a)-(b) illustrate the changes in optimality versus iterations and runtime; (c)-(d) illustrate the changes in the residual versus iterations and runtime.}
	\label{fig:news20}
\end{figure*}
\begin{figure*}
	\centering
	\begin{minipage}{.24\linewidth}
		\centerline{\includegraphics[width=\columnwidth]{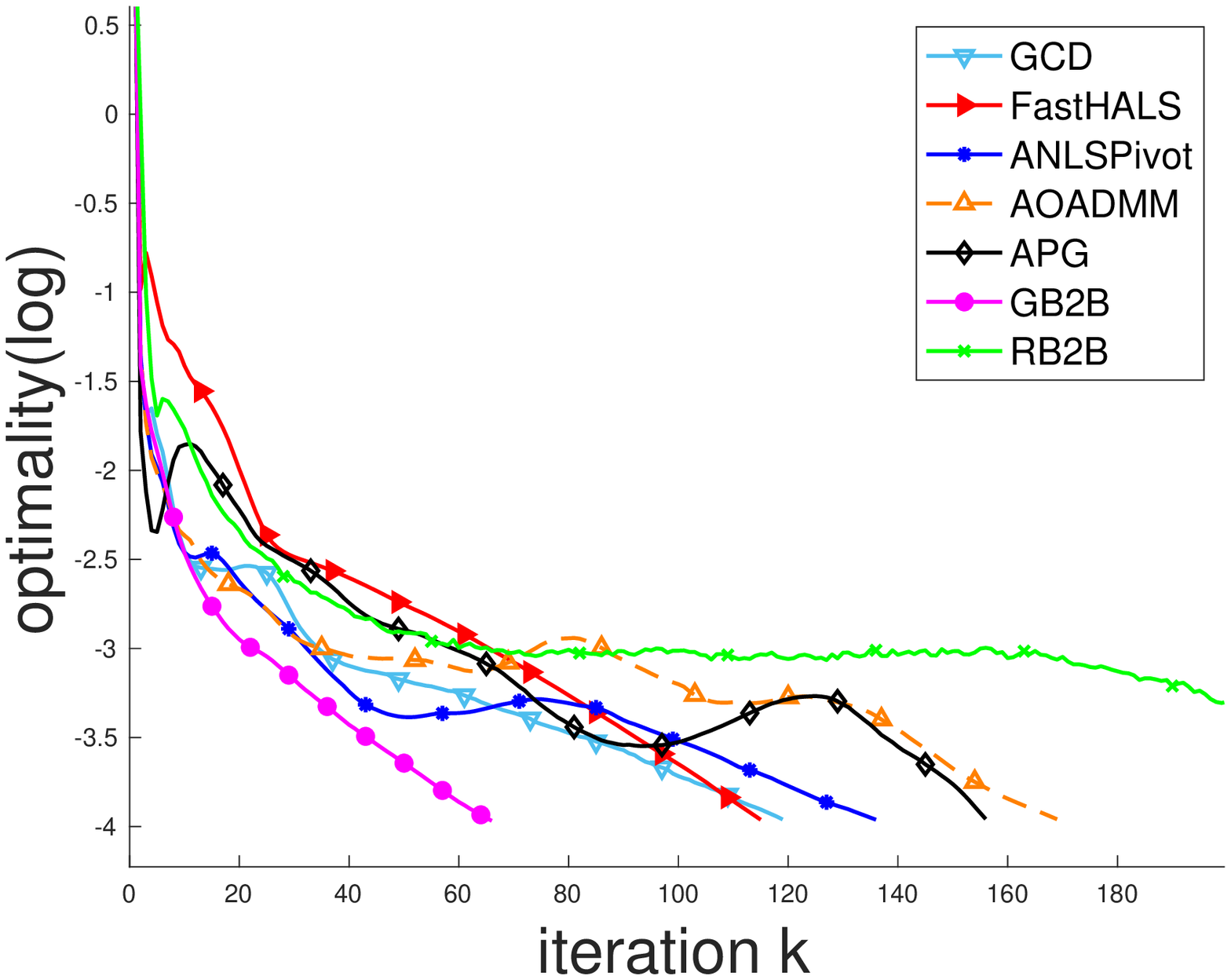}}
		%		\caption{Objective reduction}
		\caption*{(a) Optimality versus iteration.}
		%		\label{fig:obs}
	\end{minipage}
	\begin{minipage}{.24\linewidth}
		\centerline{\includegraphics[width=\columnwidth]{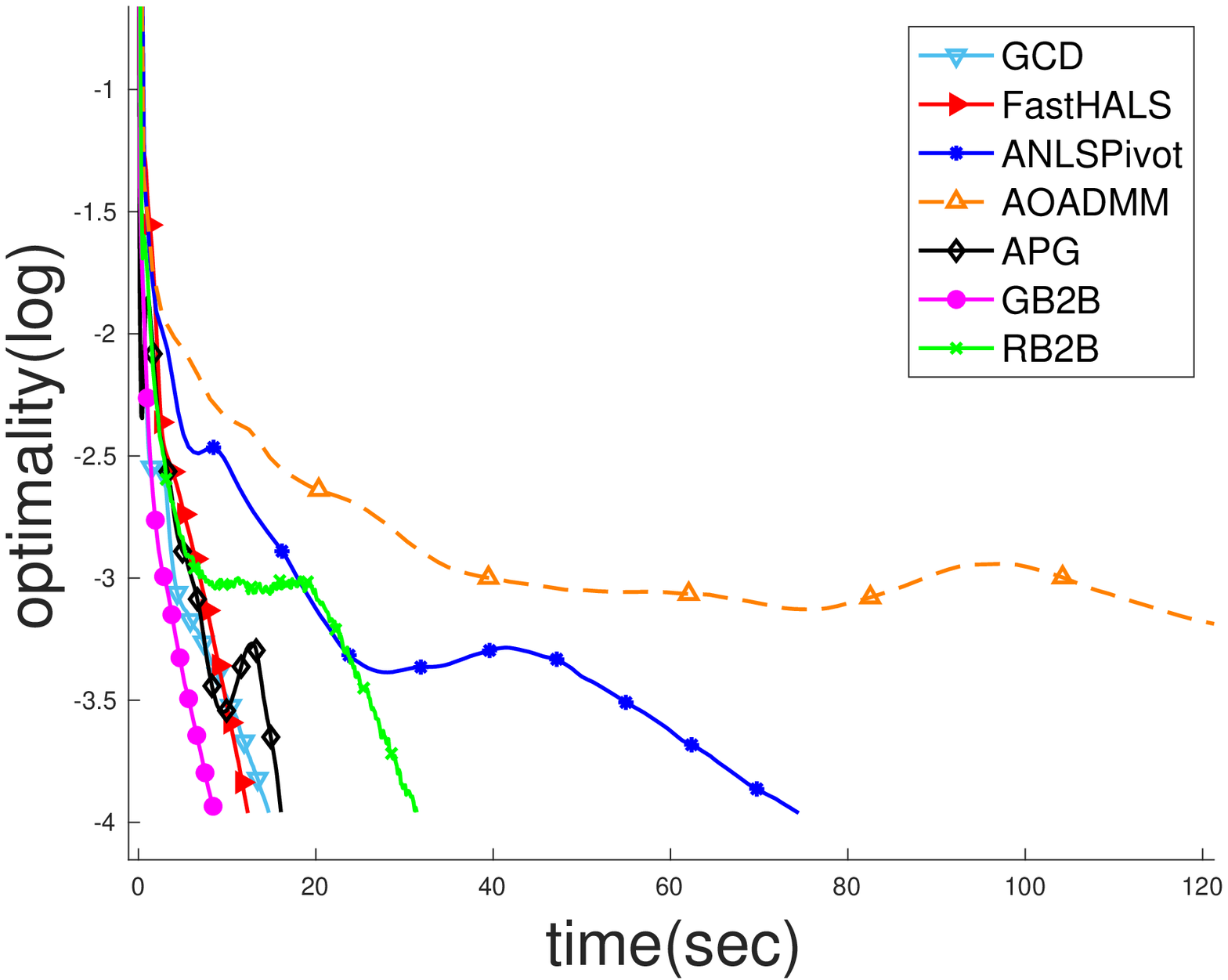}}
		%		\caption{Objective reduction}
		\caption*{(b) Optimality versus runtime.}
		%		\label{fig:obs}
	\end{minipage}
	\begin{minipage}{.24\linewidth}
		\includegraphics[width=\columnwidth]{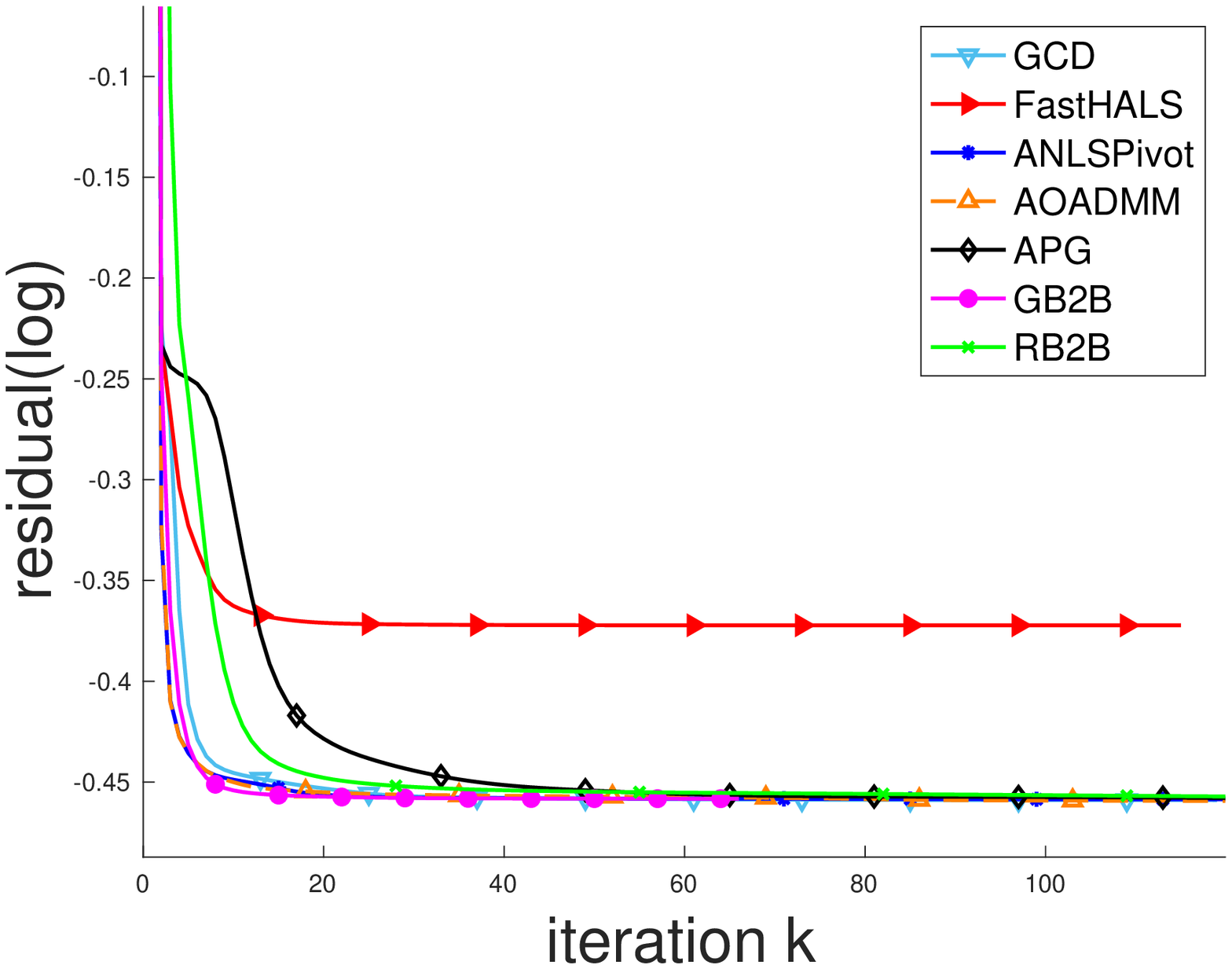}
		%		\caption{Runtime histogram}
		\caption*{(c) Residual versus iterations.}
		%		\label{fig:chol}
	\end{minipage}%
	\begin{minipage}{.24\linewidth}
		\centerline{\includegraphics[width=\columnwidth]{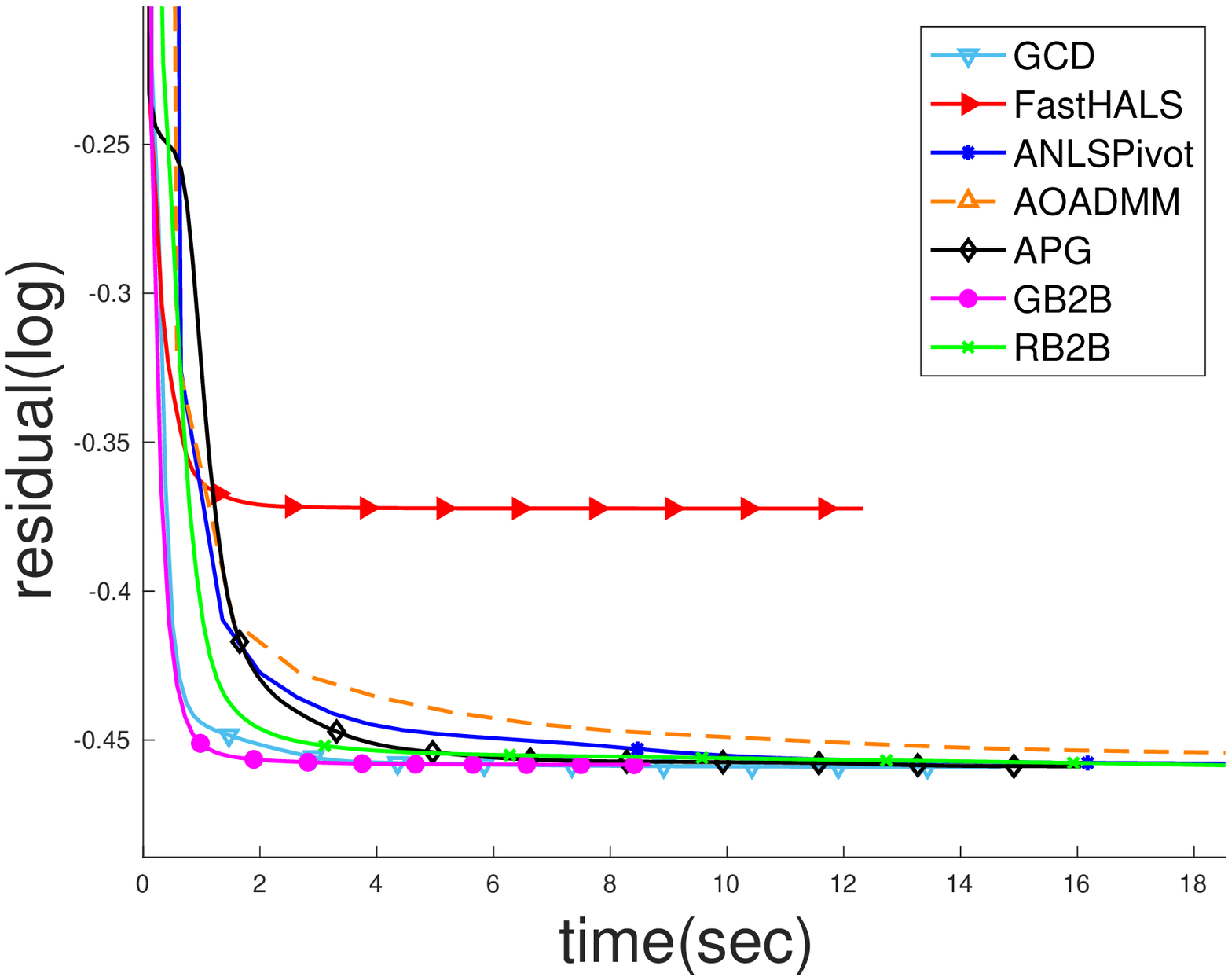}}
		%		\caption{Size of working set}
		\caption*{(d) Residual versus runtime.}
		%		\label{fig: residual exact}
	\end{minipage}
	\caption{Convergence behaviors of different algorithms on \textbf{MNIST}: (a)-(b) illustrate the changes in optimality versus iterations and runtime; (c)-(d) illustrate the changes in the residual versus iterations and runtime.}
	\label{fig:mnist}
\end{figure*}
\begin{figure*}
	\centering
	\begin{minipage}{.24\linewidth}
		\centerline{\includegraphics[width=\columnwidth]{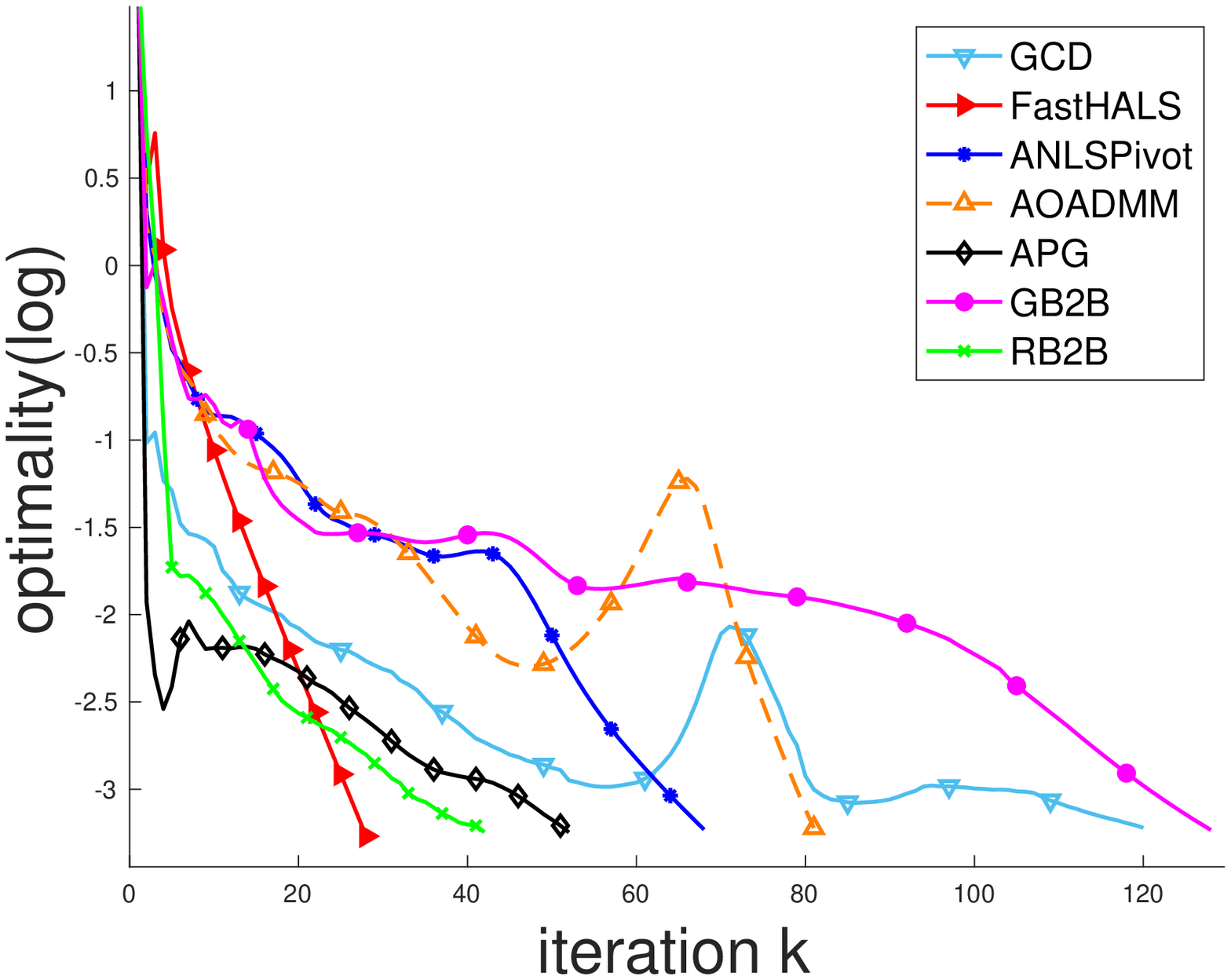}}
		%		\caption{Objective reduction}
		\caption*{(a) Optimality versus iteration.}
		%		\label{fig:obs}
	\end{minipage}
	\begin{minipage}{.24\linewidth}
		\centerline{\includegraphics[width=\columnwidth]{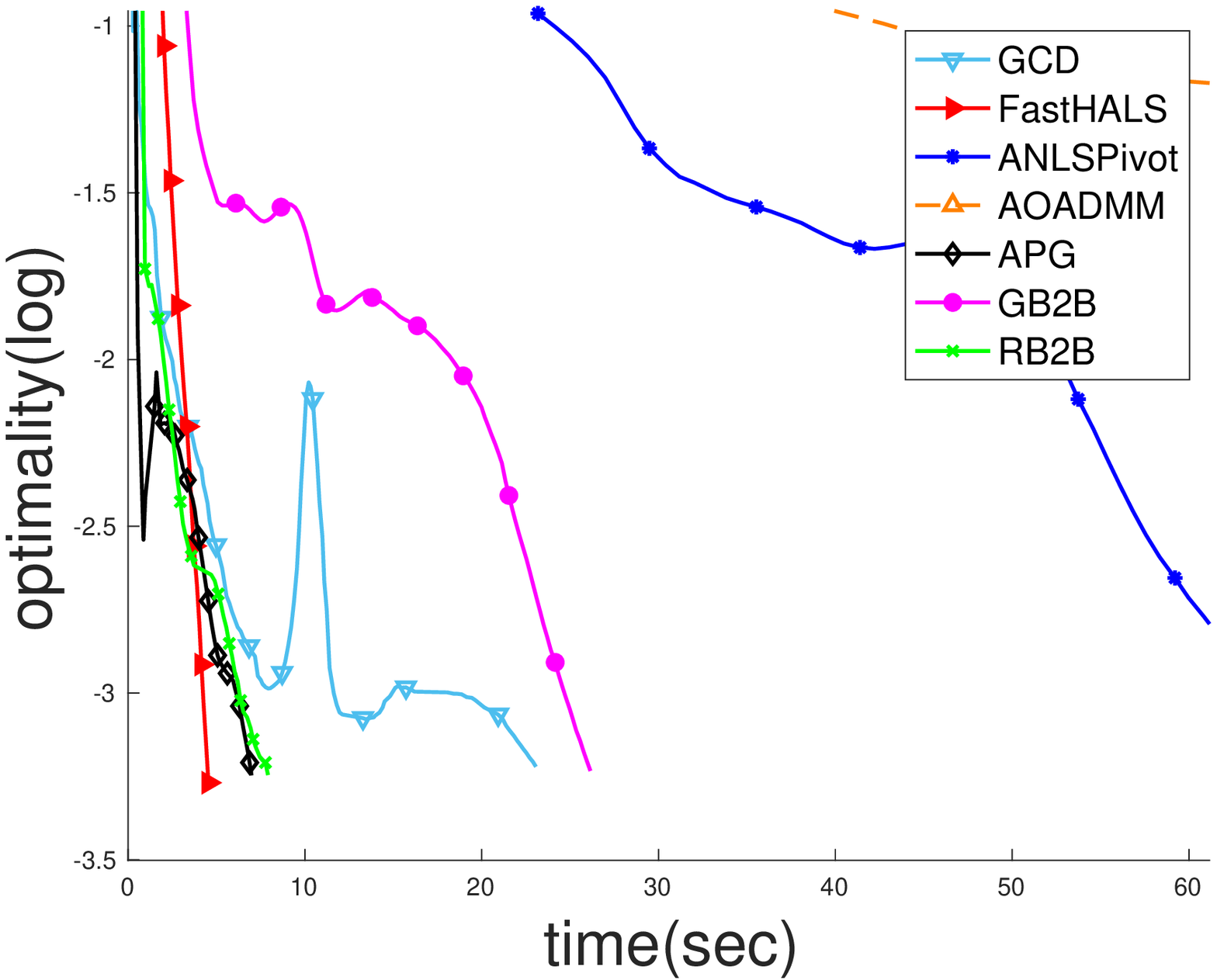}}
		%		\caption{Objective reduction}
		\caption*{(b) Optimality versus runtime.}
		%		\label{fig:obs}
	\end{minipage}
	\begin{minipage}{.24\linewidth}
		\includegraphics[width=\columnwidth]{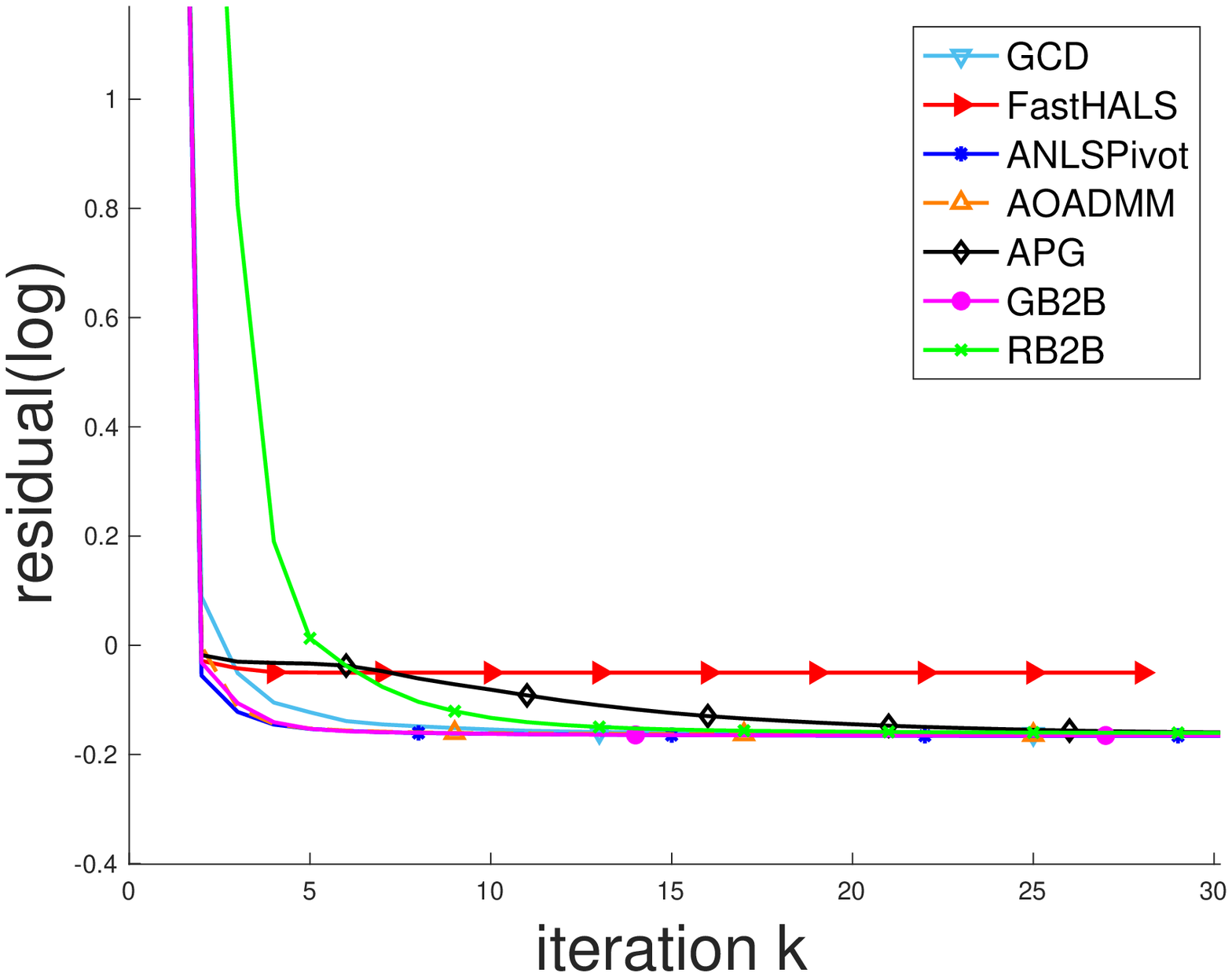}
		%		\caption{Runtime histogram}
		\caption*{(c) Residual versus iterations.}
		%		\label{fig:chol}
	\end{minipage}%
	\begin{minipage}{.24\linewidth}
		\centerline{\includegraphics[width=\columnwidth]{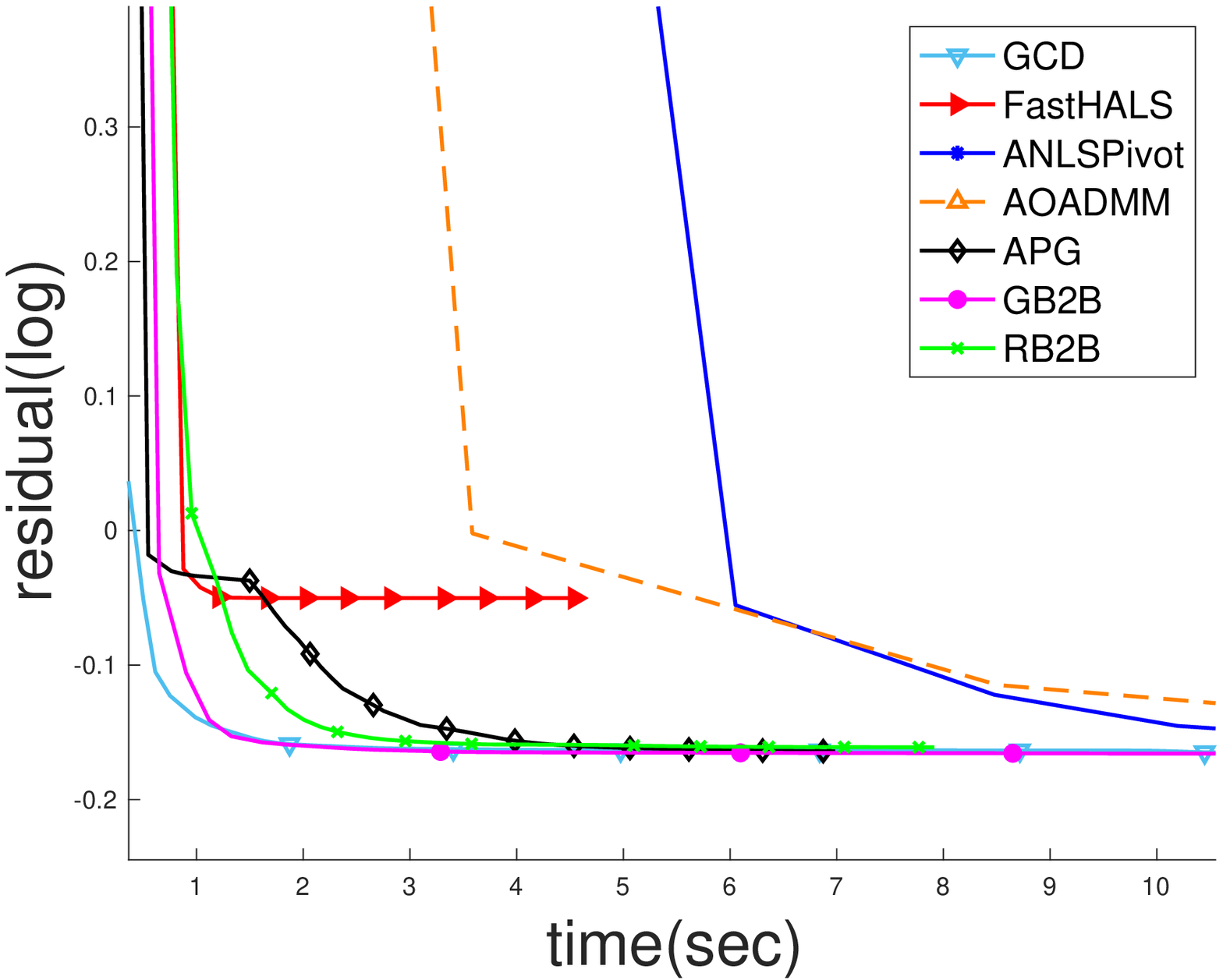}}
		%		\caption{Size of working set}
		\caption*{(d) Residual versus runtime.}
		%		\label{fig: residual exact}
	\end{minipage}
	\caption{Convergence behaviors of different algorithms on \textbf{TDT2}: (a)-(b) illustrate the changes in optimality versus iterations and runtime; (c)-(d) illustrate the changes in the residual versus iterations and runtime.}
	\label{fig:tdt2}
\end{figure*}
To showcase the strength of the B2B algorithm, we use B2B to solve the nonnegative matrix factorization (NMF) problem \cite{lee1999learning,gao2018did,gao2016minimum}. As an efficient dimension reduction method, NMF plays a crucial role in various areas, such as text mining \cite{cai2010graph}, face recognition \cite{zhao2017multi}, network detection \cite{wang2017community}, etc.

Given an elementwise nonnegative matrix $\A\in\reals_+^{M\times N}$ and a desired rank $R\leq \min\{M,N\}$, NMF seeks to approximate $\A$ by an outer product of two nonnegative matrices $U$ and $V$, \ie,
\begin{align}
\min \; \norm{\A-\U \V^T}^2_F\quad\text{such that}\; \U, \V\geq 0,\label{opt:nmf}
\end{align}
Clearly, this problem is nonconvex, and finding the exact NMF is NP-hard \cite{vavasis2009complexity}.

To apply Algorithm~\ref{alg:B2B}, we consider each column in $\U$ or $\V$ as a block. Due to space limitation, we only show the update rule for variable $\V$, since the update rule for $\U$ is similar. Define the corresponding function $f_b$ for $\bv_b$ as $f_b(\bv_b) = \frac{1}{2}\norm{\overline{\A} - \u_b \bv_b^T}^2_F$,
%\begin{align}
%	f_b(v_b) = \frac{1}{2}\norm{\overline{A} - u_bv_b^T}^2_F,
%\end{align}
where $\overline{\A} = \A - \sum_{c \neq b} \u_c \bv_c^T$.
Here, we define $h_b$ as
$
h_b(\bv_b) = \frac{1}{2}\u_b^T\u_b\norm{\bv_b}^2.
$
\begin{prop}\label{prop:nmf relative smooth}
	Let $(f_b, h_b)$ be defined as above. Then for any $L\geq 1$, the function $Lh_b- f_b$ is convex.
\end{prop}

In the definition of $h_b$, we have $m_b=M_b=1/\u_b^T \u_b$ and $\beta_b=1$. As a result, we have unit stepsize $\alpha^k = 1$. The main computational step requires to computing the search direction $\d_b^k$ in \eqref{opt:two metric mapping}, \ie,
$
\d_b = \nabla h_b^*[\nabla h_b(\x_b) - \nabla f_b(\x_b)]-\x_b.
$
It can be shown that
$
\d_b = \overline{\A}^T \u_b/\u_b^T \u_b - \bv_b.
$
With unit stepsize, we obtain $\bv_b^{+} = [\overline{\A}^T\u_b/\u_b^T \u_b ]_+$.
%\begin{align}
%	v_b^{+} = \left[\frac{\overline{A}^Tu_b}{u_b^T u_b} \right]_+.
%\end{align}
\begin{remark}
	The update rule for $\bv_b$ is not well-defined if $\u_b^T\u_b=0$. Thanks to the notion of ``valid block'', these blocks will not be selected as they are invalid. Indeed, suppose $\u_b^T\u_b = 0$. As $\u_b\geq0$, we must have $\u_b=0$, and further $\nabla f_b(\bv_b)=0$ indicating $\bv_b$ is not a valid block.
\end{remark}

%\section{Experiments}\label{sec4}

%\begin{table*}[ht]
%%	\resizebox{\textwidth}{!}{
%		\centering
%		\begin{tabular}{|c|c|c| c| c | c |c |c |c |c |}
%			\hline
%			& Algorithms& ORL & YaleB & UMist & COIL100 & News20 & MNIST & TDT2 & Reuters21578\\ \hline
%			\multirow{8}{*}{Time(Seconds)} & ANLSPivot & &  &  &  &  &  &  & \\ \cline{2-10}
%			&FastHALS & & &  &  &  &  &  &  \\ \cline{2-10}
%			&GCD & &  &  &  &  &  &  & \\ \cline{2-10}
%			&AOADMM & &  &  &  &  &  &  & \\ \cline{2-10}
%			&APG & &  &  &  &  &  &  & \\ \cline{2-10}
%			&GB2B & &  &  &  &  &  &  & \\ \cline{2-10}
%			&RB2B & &  &  &  &  &  &  & \\ \hline
%			\hline
%			\multirow{8}{*}{Iterations } & ANLSPivot & &  &  &  &  &  &  & \\ \cline{2-10}
%			&FastHALS & & &  &  &  &  &  &  \\ \cline{2-10}
%			&GCD & &  &  &  &  &  &  & \\ \cline{2-10}
%			&AOADMM & &  &  &  &  &  &  & \\ \cline{2-10}
%			&APG & &  &  &  &  &  &  & \\ \cline{2-10}
%			&GB2B & &  &  &  &  &  &  & \\ \cline{2-10}
%			&RB2B & &  &  &  &  &  &  & \\ \hline
%		\end{tabular}
%%	}
%	\caption{Performance comparison for algorithms on real datasets. $M$ is the dimension, $N$ is the number of samples, $K$ is the desired rank, and $\epsilon$ is the relative error.}\label{tab:result}
%\end{table*}

We compare the proposed algorithms GB2B and RB2B with five state-of-the-art algorithms:
\begin{enumerate}
	\item \textit{GCD}: A greedy BCD method \cite{hsieh2011fast}, where the block is selected based on the reduction in the objective function from the previous iteration.
	\item \textit{FastHALS}: A cyclic BCD method \cite{cichocki2009fast}. Here we use its fast implementation.
	\item \textit{ANLSPivot}: An alternating method and each subproblem is solved by \textit{block principal pivot} (BPP) \cite{{kim2011fast}}.
	\item \textit{AOADMM}: An alternating algorithm where the subproblems are solved by ADMM \cite{huang2016flexible}.
	\item \textit{APG}: An alternating proximal gradient method with extrapolation \cite{xu2017globally}.
\end{enumerate}
All algorithms are implemented in Matlab by the authors of the original works, except FastHALS.

We evaluate the algorithms using the following datasets.
\begin{enumerate}
	%	\item \textit{UMist}\footnote{https://cs.nyu.edu/\texttildelow roweis/data.html}: This dataset is an image dataset containing 575 images of 20 people with the size of $112\times 92$.
	\item \textit{ORL}\footnote{http://www.cad.zju.edu.cn/home/dengcai/Data/FaceData.html}: This dataset includes 40 distinct subjects which has 10 different images, where each image has $32\times32$ pixels.
	\item \textit{COIL}\footnote{http://www.cad.zju.edu.cn/home/dengcai/Data/MLData.html}: This dataset contains $7200$ images of size $32\times 32$ of $100$ objects.
	\item \textit{YaleB}\footnote{http://www.cad.zju.edu.cn/home/dengcai/Data/FaceData.html}: This dataset includes 2,414 images of 38 individuals of size 32$\times$32.
	\item MNIST\footnote{http://yann.lecun.com/exdb/mnist/}: This dataset contains handwritten digits, which has 70,000 samples of size 28$\times$28.
	\item News20:\footnote{http://qwone.com/\texttildelow jason/20Newsgroups/} A collection of 18,821 documents across 20 different newsgroups with 8,165 keywords in total.
	\item TDT2:\footnote{ http://projects.ldc.upenn.edu/TDT2/} A text dataset containing news articles from $30$ different topics.
\end{enumerate}
The detailed statistics of the datasets are given in Table~\ref{tab:dataset}.
\begin{table}[ht]
	\centering
	\begin{tabular}{|c|| c| c | c |c|}
		\hline
		Dataset	&$M$ & $N$ &$K$ & $\epsilon$\\ \hline
		ORL  & 1024 &  400 &  40 &$10^{-3}$ \\ \hline
		YaleB  & 1024 & 2414&  38 & $10^{-3}$ \\ \hline
		%		UMist  & 10304 &  575 & 20 &$10^{-3}$ \\ \hline
		COIL  & 1024 & 7200&  100 & $10^{-3}$ \\ \hline
		News20 & 8165 & 18821&  20 & $10^{-5}$ \\ \hline
		MNIST  & 784& 70000&  10&$10^{-5}$ \\ \hline
		TDT2 & 9394 & 36771& 30 & $10^{-5}$ \\ \hline
		%		Reuters21578 & 8293 & 18933& 65 & $10^{-5}$ \\ \hline
	\end{tabular}
	\caption{The parameters of the datasets, where $M$ is the dimension, $N$ is the number of samples, $K$ is the desired rank, and $\epsilon$ is the relative error.}\label{tab:dataset}
\end{table}

All algorithms start with the same initial point whose entries are uniformly distributed in the interval $[0,1]$. We stop each algorithm if the relative projected gradient is small enough, \ie,
$
\norm{\nabla^P f(\x^k)}_F \leq \epsilon \norm{\nabla^P f(\x^0)}_F,\label{eq.stopcond}
$
or a total number of 1000 iterations has been reached. Since NMF is \textit{nonconvex}, it may include multiple critical points. The quality of the critical point that an algorithm converges to is also important. Hence, we also record and compare the \textit{relative residual} defined by
$
e^k =\norm{\A-\U^k \V^{kT} }_F/\norm{\A}_F.
$
The results are averaged over 20 Monte Carlo trials and summarized in Table~\ref{tab:result}, and the convergence behaviors are illustrated in Figures~\ref{fig:orl}-\ref{fig:tdt2} in log scale. The standard deviations are small and so we does not include them in Table~\ref{tab:result}.

From Table~\ref{tab:result}, we can conclude that GB2B is consistently faster than the other algorithms in most cases. RB2B is also a good solver for NMF but is relatively slower than GB2B. In principle, RB2B should be faster than GB2B since GB2B needs to spend more time to select a block. Nevertheless, Table~\ref{tab:result} shows the exact opposite. In the use of greedy choice, the number of iterations used by GB2B is much fewer than the number of iterations used by RB2B. Consequently, the overall performance of GB2B is much better than RB2B, even each iteration in RB2B is cheaper. In fact, ANLSPiovt and AOADMM also use fewer number of iterations, but they are slower than GB2B in terms of runtime. Such superiority becomes more apparent when the size and the sparsity of the datasets increase.

In Figures~\ref{fig:orl}(a)-\ref{fig:tdt2}(a), the optimality of GB2B continuously decreases across iterations in most cases, while relatively large oscillations appears in RB2B and other methods. From Table~\ref{tab:result} and Figures~\ref{fig:tdt2}(a)-(b), it can be observed that FastHALS is also a fast solver for the text datasets, but Figures~\ref{fig:tdt2}(c)-(d) indicate that FastHALS may converge to a poor quality critical point. This problem is also observed in other text datasets Figures~\ref{fig:news20}-\ref{fig:mnist}. In summary, we can see that GB2B is the most efficient algorithm among the compared algorithms.

\section{Concluding Remarks}
In this paper, we proposed a block-wise Bregman proximal gradient descent algorithm for composite nonconvex problems, where the smooth part does not satisfy the global Lipschitz-continuous gradient property. With two reference functions, the Bregman projection reduces to the orthogonal projection so that a closed-form solution of the projection subproblem can be obtained. The global convergence of the proposed algorithms are proved for various block selection rules. In particular, we show that a global convergence rate of $\bigo{\frac{\sqrt{s}}{\sqrt{k}}}$ can be achieved by the greedy and randomized rule, which is $\bigo{\sqrt{s}}$ faster than the cyclic rule. We perform multiple numerical experiments based on real datasets to demonstrate the superiority of the proposed B2B algorithms for the NMF problem, which shows that the greedy B2B is faster than the compared algorithms and is able to converge to a better quality critical point.

\linespread{0.9}\normalsize
\section{Appendix}
\subsection{Proof of Proposition~\ref{prop:cyclic prop}}
	\begin{itemize}
	\item[(\emph{i})] From the optimality of $\x_b^{k+1}$ in \eqref{opt:GS}, we have
	\begin{align*}
	r_{b}(\x_b^{k+1}) +& \inn{\nabla f^k_{b} (\x^k_{b})}{\x_b^{k+1}-\x_{b}^k} \\
	&+ \frac{1}{\alpha^k}D_h(\x_b^{k+1}, \x^k_{b})\leq r_{b}(\x_b^k).
	\end{align*}
	Together with Lemma~\ref{lemma:nolip}, we then obtain
	\begin{align}
	&f^k_b(\x_b^{k+1}) + r_b(\x_b^{k+1})\\
	\leq& f^k_b(\x_b^k) + r_b(\x_b^k) - \left(\frac{1}{\alpha^k} - L_{b}\right)D_h(\x_b^{k+1},\x_b^k)\label{eq:cyclic prop 1}.
	\end{align}
	Summing over \eqref{eq:cyclic prop 1} for $b=1,\cdots, s$ yields
	\begin{align*}
	F(\x^{k+1}) \leq F(\x^k)-\sum_{b=1}^s \left(\frac{1}{\alpha^k} - L_b\right)D_h(\x_b^{k+1},\x_b^k)
	\end{align*}
	where we use the facts that $f^k_b(\x_b^{k}) = f^k_{b-1} (\x_{b-1}^{k+1})$. Thus, the sequence $\{F(\x^k)\}$ is nonincreasing.
	
	\item[(\emph{ii})] Noting that $L\geq L_b$ for all $b$, we obtain
	\begin{align*}
	F(\x^{k+1})
	\leq F(\x^k)-\sum_{b=1}^s \left(\frac{1}{\alpha^k} - L\right)D_h(\x_b^{k+1},\x_b^k).
	\end{align*}
	Taking the telescopic sum of the inequality above for $k=0,1,\cdots, N$ gives us
	\begin{align*}
	\sum_{k=0}^N \sum_{b=1}^s\left(\frac{1}{\alpha^k} - L\right)&D_h(\x_b^{k+1},\x_b^k) \\
	\leq& F(\x^0) - F(\x^{N+1})\leq F(\x^0) - F^*,
	\end{align*}
	where $F^*=\inf F  >\infty$. Since $\alpha^k$ is a constant, dividing both sides by $\frac{1}{\alpha^k} -L$ and taking the limit $N\rightarrow \infty$ yields the desired result.
	
	\item[(\emph{iii})] E.q. \eqref{eq:cyclic prop 1} further implies that
	\begin{align*}
	(N+1) \min_{0\leq k\leq N} \sum_{b=1}^s &D_h(\x_b^{k+1},\x_b^k) \\
	&\leq \frac{\alpha}{1 - \alpha L} (F(\x^0) - F^*),
	\end{align*}
	which yields the desired result by dividing $(N+1)$.
\end{itemize}

\subsection{Proof of Proposition~\ref{prop:cyclic convergence}}
%\begin{proof}
	The optimality condition of \eqref{opt:GS} is given by
	\begin{align*}
	0 \in\partial r_b^k(\x_b^{k+1}) + \nabla f_b^k(\x_b^k) + \frac{1}{\alpha^k}\left(\nabla h_b(\x_b^{k+1}) - \nabla h_b(\x_i^k)\right) .
	\end{align*}
	Therefore, by defining
	\begin{align*}
	\w_b^{k+1} =& \nabla_b f(\x^{k+1}) - \nabla f_b^k(\x_b^k) \\
	&+ \frac{1}{\alpha^k}\left(\nabla h_b(\x_b^{k}) - \nabla h_b(\x_b^{k+1})\right),
	\end{align*}
	we have that $\w_b^{k+1}\in \partial_b F(\x^{k+1})$. Since $\nabla f$ and $\nabla h_b$ are both $\ell$-Lipschitz-continuous on any bounded set and $\{\x^k\}$ is bounded, we have
	\begin{align*}
	\norm{\w_b^{k+1} }\leq&\norm{\nabla_b f(\x^{k+1}) - \nabla_b f_b^k(\x_b^k) }\\
	&+ \frac{1}{\alpha^k} \norm{\nabla h_b(\x_b^{k}) - \nabla h_b(\x_b^{k+1})}\\
	\leq& \ell \sum_{i = b}^s\norm{\x_i^{k+1} - \x_i^k} + \frac{\ell}{\alpha^k}\norm{\x_b^{k+1} - \x_b^k}\\
	\leq& \ell\left(1+\frac{1}{\alpha^k}\right)\norm{\x^{k+1} - \x^k}.
	\end{align*}
	Clearly, we have $\w^{k+1}\in \partial F(\x^{k+1})$. Summing over all $b=1,\cdots, s$ yields the desired result
	\begin{align*}
	\norm{\w^{k+1}} =& \sum_{b=1}^s\norm{\w_b^{k+1} } \leq \sum_{b=1}^s \ell\left(1+\frac{1}{\alpha^k}\right)\norm{\x^{k+1} - \x^k} \\
	=& s\ell\left(1+\frac{1}{\alpha^k}\right)  \norm{\x^{k+1} - \x^k}.
	\end{align*}	
	
	Let $\x^*$ be a limit point of $\{\x^k\}$, and there exists a subsequence $\{\x^{k_q}\}$ such that $\x^{k_q} \rightarrow \x^*$ as $q\rightarrow \infty$. Since the functions $r_b$ are lower semi-continuous, we have for all $b$,
	\begin{align}
	\liminf_{q\rightarrow \infty} r_b(\x_b^{k_q}) \geq r_b(\x_b^*).\label{eq:lower semicontinuous}
	\end{align}
	From \eqref{opt:GS}, we have for all $k$, taking $\x_b = \x_b^*$ yields
	\begin{align*}
	r_b(\x_b^{k+1}) + &\inn{\nabla f_b^k(\x_b^k)}{\x_b^{k+1} - \x_b^k} +\frac{1}{\alpha^k}\d_b(\x_b^{k+1}, \x_b^k)\\
	&\leq r_b(\x_b^*) + \inn{\nabla f_b^k(\x^k_b)}{\x_b^* - \x_b^k} +\frac{1}{\alpha^k}\d_b(\x_b^*, \x_b^k),
	\end{align*}
	or equivalently,
	\begin{align*}
	r_b(\x_b^{k+1}) \leq&  r_b(\x_b^*) + \inn{\nabla f_b^k(\x^k_b)}{\x_b^* - \x_b^{k+1}}\\
	& +\frac{1}{\alpha^k}\d_b(\x_b^*, \x_b^k) - \frac{1}{\alpha^k}\d_b(\x_b^{k+1}, \x_b^k).
	\end{align*}
	Choosing $k = k_q - 1$ and letting $q\rightarrow \infty$ yields
	\begin{align}
	\limsup_{q\rightarrow \infty} r_b(\x_b^{k_q}) \leq  r_b(\x_b^*),\label{eq:upper continuous}
	\end{align}
	where we have used the facts that $\{\x^k\}$ are bounded, $\nabla f$ is continuous, and $\d_b(\x_b^{k+1},\x_b^k)\rightarrow 0$ as $k\rightarrow \infty$. For that reason, we also have $\x^{k_q}\rightarrow \x^*$ as $q\rightarrow \infty$. Thus, combining \eqref{eq:upper continuous} with \eqref{eq:lower semicontinuous}, we have
	\begin{align*}
	\lim\limits_{q\rightarrow \infty} r_b(\x_b^{k_q}) = r_b(\x^*_b).
	\end{align*}
	Furthermore, by the continuity of $f$, we obtain
	\begin{align*}
	\lim\limits_{q\rightarrow \infty} F(\x^{k_q}) =& \lim\limits_{q\rightarrow \infty} \left\{f(\x^{k_q}) +   \sum_{i=1}^s r_b(\x_b^{k_q}) \right\}\\
	=& f(\x^*) + \sum_{i=1}^s r_b(\x_b^*) = F(\x^*).
	\end{align*}
	From Proposition~\ref{prop:cyclic prop} (\emph{ii}), we have that $\w^{k_q}\in \partial F(\x^{k_q})$  and $\w^{k_q} \rightarrow 0$ as $q\rightarrow \infty$. The closeness of $\partial F$ implies $0\in \partial F(\x^*)$. Therefore, $\x^*$ is a critical point of $F$.
%\end{proof}
\vspace{-10px}

\subsection{Proof of Proposition~\ref{prop:subgradient equiavlent}}
%\begin{proof}
	We need to show $\nabla^P f(\x)\in  \partial F(\x)$, which is equivalent to
	\begin{align*}
	\nabla^P f(\x) - \nabla f(\x)\in \partial r(\x).
	\end{align*}
	With $r(\x) = \delta_+(\x)$, the subdifferential of $\delta_+$ at a point $\x$ is given by
	\begin{align*}
	\partial \delta_+(\x) = \{\bv: \inn{\bv}{\x}= 0, \bv\leq 0\}.
	\end{align*}
	\begin{itemize}
		\item[(\emph{i})] 	If $\x_i > 0$, then we have $\nabla^P_i f(\x) = \nabla_i f(\x)$, and hence $\nabla^P_i f(\x) - \nabla_i f(\x) = 0$.
		\item[(\emph{ii})] If $\x_i = 0$ and $\nabla_i f(\x) > 0$, then $\nabla^P_i f(\x) = 0$, and so $\nabla^P_i f(\x) - \nabla_i f(\x) < 0$.
		\item[(\emph{iii})] If $\x_i = 0$ and $\nabla_i f(\x) \leq 0$, then $\nabla^P_i f(\x) = \nabla_i f(\x) $, and so $\nabla^P_i f(\x) - \nabla_i f(\x) = 0$.
	\end{itemize}
	Clearly, $\inn{\x}{\nabla^P f(\x) -\nabla f(\x) } = 0$
	and $\nabla^P f(\x) -\nabla f(\x) \leq 0$, which completes the proof.
%\end{proof}
\vspace{-10px}
\subsection{Proof of Lemma~\ref{lemma:critical point}}
We start with the proof for part (i).
%\begin{itemize}

	(1)$\Longrightarrow$(2). Note that the necessary optimality condition is:
	\begin{align*}
	\frac{\partial f(\x^k)}{\partial \x_i} = 0, \quad &\text{if $\x^k_i > 0$};\\
	\frac{\partial f(\x^k)}{\partial \x_i} \geq 0, \quad &\text{if $\x^k_i = 0$}.
	\end{align*}
	The desired result is obtained directly from the definition of $\nabla^P f(\bx^k) $ in \eqref{eq:projected gradient}.
	
	(2)$\Longrightarrow$(3). From the definition of subgradient of a convex function, the subdifferential of $\delta_+(\x)$ is given by
	\begin{align*}
	\partial \delta_+(\x)=\{\bv: \inn{\bv}{\x} = 0, \bv\geq 0\}.
	\end{align*}
	It is clear that $\inn{\nabla f(\x^k)}{\x^k} = 0$ and $\nabla f(\x)\geq 0$. From the block structure of $f$ and $\delta_+$, we have $-\nabla_b f(\x^k)\in \delta_+(\x_b^k)$ for all $b$.
	
 (3)$\Longrightarrow$(4). As none of the blocks is valid from the definition, we obtain $\d_b^k = 0$ and hence we obtain the desired result.

 (4)$\Longrightarrow$(1). Fixed $b$, and assume $\x^k(\alpha)=\x^k$ for all $\alpha > 0$. Let $\cW$ denote the index set that contains all coordinates from block $b$. Then we must have
	\begin{align*}
	\d_i^k = 0, \quad\forall \x^k_i>0, i\in \cW,\\
	\d_i^k \leq 0, \quad\forall  \x^k_i=0, i\in \cW.
	\end{align*}
	Since $b$ block is valid, we have that
	\begin{itemize}
		\item if $\x^k_i>0$ and $i\in \cW$, then $\partial f(\x^k)/\partial \x_i \neq 0$ and $\partial f(\x^k)/\partial \x_i\cdot \d_i^k = 0$;
		\item if $\x_i^k = 0$ and $i\in \cW$, then $\partial f(\x^k)/\partial \x_i < 0$ and so $\partial f(\x^k)/\partial \x_i\cdot \d_i^k \geq 0$.
	\end{itemize}
	These two relations imply
	\begin{align}
	\inn{\nabla_b f(\x^k)}{\d_b^k}\geq 0.\label{lemma:stat eq1}
	\end{align}
	However, from the optimality of the subproblem \eqref{opt:two metric mapping}, we have
	\begin{align}
	\inn{\nabla_b f(\x^k) }{\d_b^k}  + D_h(\x_b^k + \d_b^k, \x_b^k)\leq 0.\label{lemma: negative}
	\end{align}
	The convexity of $h$ implies $D_h(\x_b^k + \d_b^k, \x_b^k) \geq 0$, and hence $\inn{\nabla_b f(\x^k) }{\d_b^k} \leq 0$. Combining \eqref{lemma: negative} and \eqref{lemma:stat eq1}, we have $\inn{\nabla_b f(\x^k)}{\d_b^k}= 0$ and so $\d_b^k = \textbf{0}$. From the optimality condition of \eqref{opt:two metric mapping}, the solution for $\d_b^k$ is given by
	\begin{align*}
	\d^k_b = \nabla h^*[\nabla h(\x^k_b) - \nabla_b f(\x^k)] - \x^k_b.
	\end{align*}
	Nothing that $\d_b=\textbf{0}$, we obtain $\nabla_b f(\x^k) = \textbf{0}$. Since this condition holds for arbitrary block, $\x^k$ is a critical point.
	
	To prove the part (ii), we suppose that $\x^k$ is not a critical point. Let $\cW^k$ be the index set that contains all coordinates from the $b$-th block. Consider two index sets:
	\begin{align*}
	\cA^k &= \{i\in \cW^k: (\x_i^k > 0 \wedge \d_i^k \neq 0) \vee (\x_i^k=0 \wedge \d_i^k >0) \},\\
	\cB^k &= \{i\in \cW^k: (\x_i^k > 0 \wedge \d_i^k = 0) \vee (\x_i^k=0 \wedge \d_i^k \leq 0) \}.
	\end{align*}
	Clearly, we have $\cW^k = \cA^k\cup \cB^k$. Moreover, we have for all $i\in \cB^k$,
	\begin{align*}
	\x_i^k(\alpha) = \x_i^k\quad \forall \alpha > 0.
	\end{align*}
	Thus, if $\cA^k=\emptyset$, then we cannot make any progress, \ie, $\x^{k}(\alpha) = \x^k$ for all $\alpha > 0$. We will next need show that the index set $\cA^k\neq \emptyset$.
	
	By contradiction, assume that $\cA^k=\emptyset$. Since the selected block is valid, we have
	\begin{subequations}
		\begin{align*}
		\frac{\partial f(\x^k)}{\partial \x_i}<0, \quad&\text{if $\x^k_i = 0$},\\
		\frac{\partial f(\x^k)}{\partial \x_i}\neq 0, \quad&\text{if $\x^k_i > 0$},
		\end{align*}
	\end{subequations}
	Taking the inner product of $\nabla_b f(\x^k)$ and $\d_b^k$ yields
	\begin{align}
	\inn{\nabla_b f(\x^k)}{\d_b^k} = \sum_{i\in \cB^k}\frac{\partial f(\x^k)}{\partial \x_i} \cdot \d_i^k \geq 0.\label{lemma:descent eq1}
	\end{align}
	However, the optimality of \eqref{opt:two metric mapping} implies $\inn{\nabla_b f(\x^k)}{\d_b^k} \leq -D_h(\x_b^k + \d_b^k, \x_b^k)$. The strict convexity of $h$ implies $\inn{\nabla_b f(\x_b^k)}{\d_b^k} < 0$, which contradicts \eqref{lemma:descent eq1}. Therefore, the index set $\cA^k\neq \emptyset$.
	
	Next, we will derive a feasible descent direction based on the index set $\cA^k$ so that the descent of the objective value is guaranteed. We define a stepsize $\alpha_1$ such that
	\begin{align}
	\alpha_1 = \sup\{\alpha: \x_i^k + \alpha \d_i^k\geq 0, i\in \cA^k\}.\label{lemma:descent eq2}
	\end{align}
	Clearly, the stepsize $\alpha_1$ is either a finite positive value or $+\infty$. We define a direction $\overline{\d}_b^k$ as follows
	\begin{align}\label{lemma: descent eq dhead}
	\overline{\d}_i^k = \begin{cases}
	\d_i^k, &\text{if $i\in \cA^k$}\\
	\textbf{0}, &\text{if $i\in \cB^k$}.
	\end{cases}
	\end{align}
	From \eqref{lemma:descent eq2}, we have
	\begin{align*}
	\left[\x_b^k + \alpha \overline{\d}_b^k\right]_+ = \x_b^k + \alpha \overline{\d}_b^k,\quad \forall \alpha \leq \alpha_1,
	\end{align*}
	which implies that $\overline{\d}_b^k$ is a feasible direction. As discussed before, we know that
	\begin{align*}
	\sum_{i\in \cB^k}\frac{\partial f(\x^k)}{\partial \x_i} \cdot \d_i^k \geq 0.
	\end{align*}
	Therefore, we obtain
	\begin{align}
	\inn{\nabla_b f(\x^k)}{\overline{\d}_b^k}&\leq \inn{\nabla_b f(\x^k)}{\d_b^k} \label{lemma:descent eq4}\\
	&\leq -D_h(\x_b^k + \d_b^k, \x_b^k) < 0.
	\end{align}
	Clearly, the derived direction $\overline{\d}_b^k$ is a feasible descent direction. With $\overline{\alpha}_k\leq \alpha_1$, there exists a scalar $\alpha^k \leq \overline{\alpha}_k$ for which
	\begin{align}
	\x_b^k(\alpha^k) = \left[\x_b^k + \alpha^k \d_b^k\right]_+ = \x_b^k + \alpha^k \overline{\d}_b^k,\label{lemma:descent eq3}
	\end{align}
	and the desired relation \eqref{lemma:descent eq0} is satisfied.
%\end{itemize}

\subsection{Proof of Theorem~\ref{thm:line search}}
	Let $\overline{\x}$ be a limit point of $\{\x^k\}$. Suppose that $\overline{x}$ is not a stationary point. Since $\{f(\x^k)\}$ is monotonically nonincreasing and $\inf_{\x\geq 0} f(\x) > -\infty$, the sequence must converge to a finite value. Since $f$ is continuous, $f(\overline{\x})$ is a limit point of $\{f(\x^k)\}$. Thus, it follows that the entire sequence $\{f(\x^k)\}$ converges to $f(\overline{\x})$, and
	\begin{align*}
	f(\x^k) - f(\x^{k+1}) \rightarrow 0.
	\end{align*}
	Moreover, by the definition of Armijo-like rule, we have
	\begin{align*}
	f(\x^k) - f(\x^{k+1}) \geq& -\sigma \inn{\nabla_b f(\x^k)}{\x_b^{k+1} - \x_b^k}\\
	=& -\sigma \alpha^k \inn{\nabla_b f(\x^k)}{\overline{\d}_b^k},
	\end{align*}
	where the equality follows from \eqref{lemma:descent eq3}. Therefore, the right hand side in the above relation tends to zero. Let $\{\x^{q}\}$ be the subsequence that converges to $\overline{\x}$ as $q\rightarrow \infty$. From \eqref{lemma:descent eq4}, we have
	\begin{align}
	\lim\limits_{q\rightarrow \infty}\alpha_{q}= 0.\label{thm:line search eq1}
	\end{align}
	Hence, by the definition of the Armijo-like rule, we must have for some $k_0\geq 0$
	\begin{align}
	f(\x^{q}) - f(\x^{q}(\alpha_{q}/\tau)) <
	-\sigma (\alpha_q/\tau) \inn{\nabla_b f(\x^q)}{\overline{\d}_b^q}, \quad \forall q\geq k_0,
	\label{thm:line search eq2}
	\end{align}
	\ie, the initial stepsize $\alpha_0$ will be reduced at least once for all $q\geq k_0$. Since $\{\x^k\}$ is bounded, it follows from \eqref{lemma:descent eq3} that $\{\overline{\d}^k\}$ is bounded. Therefore, there exists a subsequence $\{\d^{p}\}$ of $\{\d^{q}\}$ such that
	\begin{align*}
	\d^{p}\rightarrow \hat{\d}.
	\end{align*}
	From \eqref{thm:line search eq2}, we have
	\begin{align*}
	\frac{f(\x^{p}) - f(\x^{p}(\overline{\alpha}_p )) }{\overline{\alpha}_p}<
	-\sigma  \inn{\nabla_b f(\x^p)}{\overline{\d}_b^p},
	\end{align*}
	where $\overline{\alpha}_p = \alpha_{p}/\tau$. By using the mean value theorem, there exists some $\tilde{\alpha}_p\in [0, \overline{\alpha}_p]$ such that
	this relation is written as
	\begin{align*}
	\inn{\nabla_b  f(\x^{p}(\tilde{\alpha}_p )) }{\overline{\d}_b^p} <
	-\sigma  \inn{\nabla_b f(\x^p)}{\overline{\d}_b^p}.
	\end{align*}
	Taking limits in the above relation, we obtain
%	\begin{align*}
%	\inn{\nabla_b  f(\overline{\x} ) }{\hat{\d}_b} <
%	-\sigma  \inn{\nabla_b f(\overline{\x})}{\hat{\d}_b},
%	\end{align*}
%	or
	\begin{align*}
	0\leq (1-\sigma)  \inn{\nabla_b f(\overline{\x})}{\hat{\d}_b}.
	\end{align*}
	Since $\sigma < 1$, it follows that
	\begin{align*}
	\inn{\nabla_b f(\overline{\x})}{\hat{\d}_b} \geq 0,
	\end{align*}
	which contradicts that $\hat{\d}_b$ is a descent direction in \eqref{lemma:descent eq4} if $\bx^k$ is not a stationary point. This proves the desired result.

\subsection{Proof of Lemma~\ref{lemma:decrease}}
	Applying Lemma~\ref{lemma:nolip} for $(f_b, h_b)$ and nothing that $f_b(\x^{k+1}) = f(\x^{k+1})$, we have
	\begin{align}
	&f(\x^{k+1})-f(\x^k)\leq \inn{\nabla_b f(\x^k)}{\x^{k+1}-\x^k} + L_bD_h(\x_b^{k+1},\x_b^k)\nonumber\\
	&\overset{\eqref{lemma:descent eq3}}{\leq} \alpha^k\inn{\nabla_b f(\x^k)}{\overline{\d}_b^k} + L_bD_h(\x_b^{k+1},\x_b^k)
	\nonumber\\
	&\overset{\eqref{lemma:descent eq4}}{\leq}\alpha^k\inn{\nabla_b f(\x^k)}{\d_b^k} + L_bD_h(\x_b^{k+1},\x_b^k).\label{lemma:decrease eq1}
	\end{align}	
	Note that the optimality condition of \eqref{opt:two metric mapping} is given by
	\begin{align*}
	\nabla_b f(\x^k) + \nabla h_b(\x_b^k+\d_b^k) - \nabla h_b(\x_b^k) = 0.
	\end{align*}
	Taking the inner product of the left-hand side of the above relation with $\z_b^k-\x_b^k$ yields
	\begin{align*}
	&\inn{\nabla_b f(\x^k) }{z^k_b-\x^k_b}= \inn{\nabla h_b(\x_b^k+\d_b^k) - \nabla h_b(\x_b^k)}{\x_b^k-\z_b^k} \\
	=&D_h(\z_b^k,\x_b^k + \d^k_b)-D_h(\z_b^k,\x_b^k)-D_h(\x_b^k, \x_b^k+\d_b^k),
	\end{align*}
	where the second equality follows from $D_h(\x,\z) - D_h(\x,\y)-D_h(\y,\x) = \inn{\nabla h(\y)-\nabla(\z)}{\x-\y}$. Setting $\d_b^k = \z^k_b-\x^k_b$ yields
	\begin{align*}
	\inn{\nabla_b f(\x^k) }{\d^k_b}
	=&-D_h(\x_b^k + \d_b^k,\x_b^k)-D_h(\x_b^k, \x_b^k+\d_b^k)\\
	\leq& -(1+\beta)D_h(\x_b^k + \d_b^k,\x_b^k).
	\end{align*}
	From \eqref{lemma:decrease eq1}, it follows that
	\begin{align*}
	f(\x^{k+1})\leq f(\x^k) - \alpha^k(1+\beta)D_h(\x_b^k + \d_b^k,\x_b^k) + L_bD_h(\x_b^{k+1},\x_b^k).
	\end{align*}
	From Assumption~\ref{assume: convex} (\emph{ii}) and \ref{assume: smooth}, we have
	\begin{align*}
	f(\x^{k+1})\leq& f(\x^k) - \frac{\alpha^k(1+\beta)m_b}{2}\norm{\d_b^k}^2 + \frac{L_b M_b(\alpha^k)^2 }{2}\norm{\overline{\d}_b^k}^2\\
	\leq&f(\x^k) - \frac{\alpha^k(1+\beta)m_b}{2}\left(1-\frac{L_b M_b\alpha^k}{(1+\beta)m_b}\right)\norm{\d_b^k}^2,
	\end{align*}
	where the last inequality follows from \eqref{lemma: descent eq dhead}.

\subsection{Proof of Proposition~\ref{prop:subsequence}}
\begin{itemize}
	\item[(\emph{i})] Since $0 < \alpha^k\leq\frac{(1+\beta)m}{2LM}$, we have
	\begin{align*}
	f(\x^+)\leq& f(\x)-\frac{\alpha^k(1+\beta)m}{4}\norm{\d_b^k}^2\\
	\overset{\eqref{lemma:descent eq3}}{\leq}&f(\x)-\frac{(1+\beta)m}{4\alpha^k}\norm{\x_b^{k+1} -\x_b^k}^2\\
	=&f(\x)-\frac{LM}{2} \norm{\x_b^{k+1} -\x_b^k}^2.
	\end{align*}
	Using the fact that $\norm{\x_b^{k+1} -\x_b^k}^2 =\norm{\x^{k+1} -\x^k}^2$ leads to the desired result.
	\item[(\emph{ii})] Summing the inequalities in Proposition~\ref{prop:subsequence} (\emph{i}) for $k=0,1, \cdots, N$, we obtain:
	\begin{align}
	\sum_{k=0}^{N}\norm{\x^{k+1} - \x^k}^2\leq& \frac{2(f(\x^0) - f(\x^N))}{LM} \nonumber\\
	\leq&\frac{2(f(\x^0) - f^*)}{LM},\label{prop:subsequence sum}
	\end{align}
	where $f^* = \inf_{\x\geq 0} f(\x)$. Taking the limit as $N\rightarrow \infty$, we can conclude that $\{\norm{\x^{k+1}-\x^k}^2\}$ converges to zero.
	\item[(\emph{iii})] From \eqref{prop:subsequence sum}, we also obtain that
	\begin{align*}
	(N+1) \min_{0\leq k\leq N} \norm{\x^{k+1} - \x^k}^2
	\leq& \sum_{k=0}^{N}\norm{\x^{k+1} - \x^k}^2\\
	\leq& \frac{2(f(\x^0) - f^*)}{LM}.
	\end{align*}
	Dividing $N+1$ on both sides completes the proof.
\end{itemize}
\vspace{-10px}
\subsection{Proof of Theorem~\ref{prop:greedy}}
\begin{itemize}
	\item[(\emph{i})] Note that the optimality condition of \eqref{opt:two metric mapping} can be written as
	\begin{align*}
	\d_b^k =& \nabla h_b^*[\nabla h_b(\x_b^k) - \nabla_b f(\x^k)] - \x_b^k\\
	=&\nabla h_b^*[\nabla h_b(\x_b^k) - \nabla_b f(\x^k)] - \nabla h_b^*[\nabla_b h(\x_b^k)].
	\end{align*}
	Then we obtain:
	\begin{align}
	\norm{\d_b^k}^2 =& \norm{\nabla h_b^*[\nabla h_b(\x_b^k) - \nabla_b f(\x^k)] - \nabla h_b^*[\nabla_b h(\x_b^k)]}^2 \nonumber\\
	\geq&  \frac{1}{M}\norm{\nabla_b f(\x^k)}^2\label{prop:greedy direction}
	\end{align}
	It follows from \eqref{eq:greedy} that
	\begin{align*}
	\norm{\d_b^k}^2\geq \frac{1}{sM} \norm{ \nabla^P f(\x^k) }^2.
	\end{align*}
	Combining the above relation with \eqref{lemma:decrease eq inequality} yields
	\begin{align*}
%	f(\x^k) - f(\x^{k+1})\geq& \frac{\alpha^k(1+\beta)m}{2}\left(1-\frac{L_b M_b\alpha^k}{(1+\beta)m}\right)\frac{1}{sM} \norm{ \nabla^P f(\x^k) }^2\\
	f(\x^k) - f(\x^{k+1})\geq
	\frac{\alpha^k(1+\beta)m}{4sM}\norm{ \nabla^P f(\x^k) }^2
	\end{align*}
	\item[(\emph{ii})] Taking the telescopic sum over $k=0,1, \cdots, N$ yields
	\begin{align*}
	&(N+1) \min_{0\leq k\leq N} \norm{\nabla^P f(\x^k)}^2
	\leq \sum_{k=0}^{N} \norm{\nabla^Pf (\x^k)}\\
	\leq &\frac{4sM}{\alpha^k(1+\beta) m}\left(f(\x^k) - f(\x^N)\right)\\
	\leq &\frac{4sM}{\alpha^k(1+\beta) m}\left(f(\x^k) - f^*\right).
	\end{align*}
	Dividing both sides by $N+1$ yields the stated result.
	\item[(\emph{iii})] The desired result can be obtained by repeating the second part of the proof of Proposition~\ref{prop:cyclic convergence}, and so we omit the proof here for brevity.
\end{itemize}
\vspace{-15px}
\subsection{Proof of Theorem~\ref{prop:randomized}}
	\begin{itemize}
		\item[(\emph{i})] Since $\alpha^k \leq  \frac{(1+\beta)m}{2LM}$, we have
		\begin{align*}
		f(\x^k) - f(\x^{k+1}) &\geq \frac{\alpha^k(1+\beta)m}{4} \norm{ \d_b^k }^2\\
		&\overset{\eqref{prop:greedy direction}}{\geq}
		\frac{\alpha^k(1+\beta)m}{4M}\norm{\nabla_b f(\x^k)}^2
		\end{align*}
		Taking expectation on both sides of the above relation with respect to $b_k$ yields
		\begin{align*}
		f(\x^k) - \E_{b_k}f(\x^{k+1})\geq&
		\frac{\alpha^k(1+\beta)m}{4M} \E_{b_k}\norm{\nabla_b f(\x^k)}^2\\
		=&\frac{\alpha^k(1+\beta)m}{4M} \sum_{i=1}^s \frac{1}{s}\norm{\nabla_b f(\x^k)}^2\\
		=&\frac{\alpha^k(1+\beta)m}{4sM} \norm{\nabla^P f(\x^k)}^2
		\end{align*}
		We then take the expectation on both sides with respect to all variables $b_0, b_1, \cdots$, to obtain
		\begin{align*}
		f(\x^k) - \E f(\x^{k+1}) \geq  \frac{\alpha^k(1+\beta)m}{4sM} \E \norm{\nabla^P f(\x^k)}^2.
		\end{align*}
		\item[(\emph{ii})] Taking the telescopic sum for $k=0,1,\cdots, N$ yields
		\begin{align*}
		(N+1) \min_{0\leq k\leq N} \E \norm{\nabla^P f(\x^k)}\leq \frac{4sM }{\alpha^k (1+\beta) m}(f(\x^0) - f^*).
		\end{align*}
		Dividing both sides by $N+1$ completes the proof.
		\item[(\emph{iii})] The stated result can be obtained by repeating the second part of the proof of Proposition~\ref{prop:cyclic convergence}, and so we omit it.
	\end{itemize}
\vspace{-15px}
\subsection{Proof of Proposition~\ref{prop:nmf relative smooth}}
	Since $f_b$ and $h_b$ are both twice continuously differentiable, in order to ensure the convexity of $Lh_b- f_b$, it is sufficient to find $L > 0$ such that $L\nabla^2 h_b\succeq \nabla^2 f_b$. By a straightforward computation, we obtain that
	\begin{align*}
	\nabla^2 h_b(\bv_b) = \u_b^T\u_b I_b,\quad \nabla^2 f_b(\bv_b) = \u_b^T\u_b I_b.
	\end{align*}
	As $\nabla^2 h_b(\bv_b)  = \nabla^2 f_b(\bv_b)$, we have $L\nabla^2 h_b\succeq \nabla^2 f_b,\forall L\geq 1$.
%\clearpage
\small
%\vspace{-20px}
\onlytwo{\linespread{0.95}\normalsize}
\bibliographystyle{IEEE-unsorted}
\bibliography{refs}
\onecolumn
\section*{Supplemental}
\subsection*{Proof of Theorem~\ref{thm:Global convergence}}
Here, we first review the essential ingredients of the methodology \cite{bolte2014proximal}. To solve a general optimization problem in the form of $\eqref{opt:PG}$, we first define \textit{gradient-like descent sequence} as follows.
\begin{define}\label{def:gradient-like descent sequence}
	A sequence $\{\x^k\}$ is called gradient-like descent sequence for $F$ if the following three conditions hold
	\begin{itemize}
		\item[(\emph{i})] Sufficient decreases property. There exists a constant $\rho_1 > 0$ such that
		\begin{align}
		\rho_1\norm{\x^{k+1} - \x^k}\leq F(\x^k) -F(\x^{k+1}).
		\end{align}
		\item[(\emph{ii})] A subgradient lower bound for the iterates gap. There exists a scalar $\rho_2>0$ such that
		\begin{align}
		\norm{\w^{k+1}}\leq \rho_2\norm{\x^{k+1} - \x^k}
		\end{align}
		for some $\w^{k+1}\in \partial F(\x^{k+1})$.
		\item[(\emph{iii})] Let $\overline{\x}$ be a limit point of the subsequence ${\x^{k_q}}$, then $\limsup_{q\rightarrow \infty} F(\x^{k_q})\leq F(\overline{\x})$.
	\end{itemize}
\end{define}
The first two conditions are typical properties of a descent method. From Proposition~\ref{prop:subsequence}, Theorem~\ref{prop:greedy}, and Theorem~\ref{prop:randomized}, it follows that the first two conditions are satisfied for B2B method. The third condition is weak condition and trivially holds if $F$ is continuous. In the case of \eqref{opt:problem}, the third condition obviously holds due to Assumption~\ref{assump:lower bounded}. Let $\omega(\x^0)$ be the set of all limit points of $\{\x^k\}$.
\begin{lemma}\label{lemma:property of limit points}
	Let $\{\x^k\}$ be the sequence generated by Algorithm that is assumed to be bounded. The following assertions hold:
	\begin{itemize}
		\item[(\emph{i})] $\omega(\x^0)$ is nonempty and compact.
		\item[(\emph{ii})] $\omega(\x^0)\subset \crit F$.
		\item[(\emph{iii})]
		\begin{align}
		\lim\limits_{k\rightarrow \infty} \dist(\x^k, \omega(\x^0)) = 0.
		\end{align}
		\item[(\emph{iv})] $F$ is finite and constant on $\omega(\x^0)$.
	\end{itemize}
\end{lemma}
\begin{proof}
	\begin{itemize}
		\item[(\emph{i})] Since $\{\x^k\}$ is bounded, it has at least one limit point and so $\omega(\x^0)\neq \emptyset$. Let $\overline{\x}$ be a limit point of $\omega(\x^0)$. Then given $\varepsilon > 0$ there exists a point $\x^\prime\in \omega(\x^0)$ with $\norm{\x^\prime - \overline{\x}} < \varepsilon/2$. Note that $\x^\prime\in \omega(\x^0)$ is a limit point of $\{\x^k\}$. Thus there exists a point $\x^{\prime\prime}\in \{\x^k\}$ with $\norm{\x^{\prime} - \x^{\prime\prime} }< \varepsilon/2$. It then follows
		\begin{align*}
		\norm{\overline{\x} - \x^{\prime\prime}} \leq \norm{\overline{\x}-\x^{\prime}} + \norm{\x^{\prime} - \x^{\prime\prime} } < \varepsilon.
		\end{align*}
		Thus $\x$ is a limit point of $\{\x^k\}$, \ie, $x\in\omega(\x^0)$. We have shown $\omega(\x^0)$ contains all its limit points, indicating $\omega(\x^0)$ is closed. Since $\{\x^k\}$ is bounded, we have $\omega(\x^0)$ is compact.
		
		\item[(\emph{ii})] Let $\x^*$ be a limit point of $\{\x^k\}$, \ie, $\x^*\in\omega(\x^0)$, and there exist a subsequence $\{\x^{k_q}\}$ such that $\x^{k_q} \rightarrow \x^*$ as $q\rightarrow \infty$. Due to the continuity of $f$, we get
		\begin{align*}
		\lim\limits_{q\rightarrow \infty} f(\x^{k_q}) =  f(\x^*).
		\end{align*}
		If this subsequence $\{\x^{k_q}\}$ converges to a noncritical point, we must have, as in the proof of Theorem~\ref{thm:line search}, $f(\x^{k_q+1}) - f(\x^k_q)\rightarrow 0$. The relations in Proposition~\ref{prop:greedy}(i) and Proposition~\ref{prop:randomized}(i) imply $\nabla^P f(\x^{k_q})\rightarrow 0$. It follows from Lemma~\ref{lemma:critical point}(ii) that $\x^*$ is a critical point.
		
		\item[(\emph{iii})] The desired result follows directly from the definition of limit points.
		
		\item[(\emph{iv})] As $\{f(\x^k)\}$ is nonincreasing and $f^*>-\infty$, the sequence $\{f(\x^k)\}$ is convergent. Let $\lim\limits_{k\rightarrow \infty} f(\x^k) = \overline{f}$. Take $\x^*$ in $\omega(\x^0)$ and so $\x^{k_q}\rightarrow \x^*$ as $q\rightarrow \infty$. On one hand, we have $f(\x^{k_q}) \rightarrow \overline{f}
		$. On the other hand, we have $f(\x^{k_q}) \rightarrow f(\x^*)$. Hence, we have $f(\x^*) = \overline{f}$. Therefore, the restriction of $f$ to $\omega
		(\x^0)$ equals to $\overline{f}$.
	\end{itemize}
\end{proof}

The above, together with the so-called nonsmooth Kurdyka-\L ojasiewicz property \cite{bolte2007lojasiewicz} allows us to establish the global convergence results for B2B method. Recall the notion of Kurdyka-\L ojasiewicz (K\L) property. Given a set $S $, the distance from a point $\x$ to $\cS$ is defined by
\begin{align}
\dist(\x, \cS):= \inf\{\norm{\x-\y}, \y\in \cS\}.
\end{align}
We let $\dist(\x,\cS) \triangleq \infty$ for all $\x$ if $\cS=\emptyset$.

Let scalar $\eta > 0$. We let $\Phi_\eta$ denote the class of functions $\varphi$ as follows
\begin{align}
\Phi_\eta = \{\varphi \in C^0[0, \eta) \cap C^1(0, \eta): \varphi\geq 0, \varphi(0) = 0, \varphi \; \text{concave}, \;\text{and}\;\varphi^\prime > 0  \}.
\end{align}
Now we give the definition of Kurdyka-\L ojasiewicz property.
\begin{define}[Kurdyka-\L ojasiewicz property]\cite[Definition~6.2]{bolte2018first}\label{def:KL}
	A proper and lsc function $\phi:\rightarrow (-\infty, +\infty]$ has the K\L$\;$property locally at $\overline{\x} \in \dom \phi$ if there exist $\eta > 0$, $\varphi \in \Phi_\eta$, and a neighborhood $U(\overline{\x})$ such that
	\begin{align}
	\varphi^\prime(\phi(\x)  - \phi(\overline{\x})) \cdot \dist(0, \partial \phi) \geq 1
	\end{align}
	for all $x\in U(\overline{\x}) \cap [\phi(\overline{\x}) < \phi < \phi(\overline{\x}) + \eta ]$.
\end{define}

To establish the \textit{global} convergence of the proposed algorithm, we need to additionally assume the (nonsmooth) K\L$\;$property in Definition~\ref{def:KL} on the objective function $F$, which is stated as follows:
\begin{assump}
	The objective function $F$ satisfies the K\L$\;$property.
\end{assump}
To prove the main theorem,  we first invoke \cite[Lemma 6]{bolte2014proximal}.
\begin{lemma}\cite[Lemma 6]{bolte2014proximal}\label{lemma:uniform}
	Let $\Omega$ be a nonempty and compact set, and let $\phi:\rightarrow(-\infty, \infty]$ be a proper and lsc function. Assume $\phi$ is finite and constant on $\Omega$ and satisfies K\L$\;$property for every point in $\Omega$. Then, there exist $\epsilon > 0$ and $\eta > 0$ and $\varphi\in \Phi_\eta$ such that for all $\overline{\x}\in \Omega$, we have
	\begin{align}
	\varphi^\prime(\phi(\x) - \phi(\overline{\x}))\cdot \dist(0, \partial \phi(\x))\geq 1
	\end{align}
	for all $\x$ in the following intersection
	\begin{align}
	\{\x: \dist(\x, \Omega) < \epsilon \} \cap [\phi(\overline{\x}) < \phi <  \phi(\overline{\x})  + \eta ].
	\end{align}
\end{lemma}
%\begin{proof}
%	Denote by $\mu$ the value of $\phi$ on $\Omega$. As $\Omega$ is compact, it can be covered by a finite number of open balls $B(\overline{\x}_i, \epsilon_i)$ on which the K\L$\;$property holds on $\overline{\x}_i$ with $\overline{\x}_i\in\Omega$ for $i=1,\cdots, p$. For each $i$, we denote the corresponding concave function by $\varphi_i:[0, \eta_i) \rightarrow\reals_+$. Therefore, we have for each $\overline{\x}_i$
%	\begin{align}
%	\varphi_i^\prime(\phi(\x) - \phi(\overline{\x}_i))\cdot \dist(0, \partial \phi(\x))=\varphi_i^\prime(\phi(\x) -\mu )\cdot \dist(0, \partial \phi(\x))\geq 1,\label{eq:lemma5-0}
%	\end{align}
%	for all $ \x\in B(\overline{\x}_i, \epsilon_i)\cap [\phi(\overline{\x}_i) <\phi <  \phi(\overline{\x}_i) + \eta]$. Choose $\epsilon>0$ sufficiently small such that
%	\begin{align}
%	U_{\epsilon} := \{ \x\in : \dist(x,\Omega)\leq \epsilon \} \subset \bigcup_{i} B(\overline{\x}_i, \epsilon_i).
%	\end{align}
%	Set $\eta = \min\{\eta_i:i=1,\cdots,p\}$ and
%	\begin{align}
%	\varphi(\y) = \sum_{i=1}^p \varphi_i(\y), \quad \forall y\in [0, \eta).
%	\end{align}
%	In the view of \eqref{eq:lemma5-0}, we obtain for all $\x\in U_\epsilon\cap [\mu < \phi < \mu +\eta]$
%	\begin{align}
%	\varphi^\prime(\phi(\x) - \mu)\cdot \dist(0, \partial \phi(\x))=\sum_{i=1}^p \varphi_i^\prime(\phi(\x) -\mu )\cdot \dist(0, \partial \phi(\x))
%	\geq 1,
%	\end{align}
%	which completes the proof.
%\end{proof}
Now we are ready to establish the result of global convergence.
\begin{theorem}[Global convergence]
	Let $\{\x^k\}$ be the sequence generated by Algorithm~\ref{alg:B2B}, which is assumed to be bounded. Then, the sequence $\{\x^k\}$ has a finite length and converges to a critical point of $F$.
\end{theorem}
\begin{proof}
	We first show there exist an integer $k > 0$ such that the uniformized K\L property holds. Since $\{\x^k\}$ is bounded, there exists a subsequence $\{\x^{k_q}\}$ that converges to a limit point $\overline{\x}$. As with Lemma~\ref{lemma:property of limit points}(ii), we have
	\begin{align}
	\lim\limits_{q\rightarrow \infty} F(\x^{k_q}) = F(\overline{\x}).
	\end{align}
	From Proposition~\ref{prop:subsequence}(ii), the distance between two consecutive iterates shrinks to zero as $q\rightarrow \infty$. Thus, $\x^{k_q}\rightarrow \overline{\x}$, which implies $\x^{k_q-1}\rightarrow \overline{\x}$ as $q\rightarrow \infty$. With simple induction, we have
	\begin{align}
	\lim\limits_{k\rightarrow \infty} F(\x^k) = F(\overline{\x}).
	\end{align}
	Since $\{F(\x^k)\}$ is nonincreasing, given a $\eta>0$, there must exist an integer $k_0$ such that $F(\x^k) < F(\overline{\x}) + \eta$ for all $k > k_0$. It follows from Lemma~\ref{lemma:property of limit points}(iii) that given an $\epsilon > 0$, there exist an integer $k_1$ such that $\dist(\x^k, \omega(\x^0)) \leq \epsilon$ for all $k > k_1$. Therefore, setting $l:=\max\{k_0, k_1\}$, we obtain:
	\begin{align}
	\x^k\in \{\x^k: \dist(\x^k, \omega(\x^0)) \leq \epsilon\}\cap [F(\overline{\x}) < F(\x^k) < F(\overline{\x}) + \eta], \quad\forall k> l.
	\end{align}
	Since $\omega(\x^0)$ is nonempty and compact from Lemma~\ref{lemma:property of limit points}(i), and $F$ is finite and constant on $\omega(\x^0)$ from Lemma~\ref{lemma:property of limit points}(iv), the uniformized K\L property in Lemma~\ref{lemma:uniform} holds by setting $\Omega = \omega(\x^0)$. Hence, for all $k> l$, we have
	\begin{align}
	\varphi^\prime (F(\x^k) - F(\overline{\x})\cdot \dist(0, \partial F(\x^k)) \geq 1.\label{eq:thm3-0}
	\end{align}
	It follows from Definition~\ref{def:gradient-like descent sequence} that
	\begin{align}
	\varphi^\prime (F(\x^k) - F(\overline{\x}) ) \geq \frac{1}{\rho_2\norm{\x^k - \x^{k-1}}}.\label{eq:thm3-1}
	\end{align}
	From the concavity of $\varphi$, we have
	\begin{align}
	\varphi (F(\x^{k+1}) - F(\overline{\x}) ) \leq  \varphi (F(\x^{k}) - F(\overline{\x}) ) + \varphi^\prime (F(\x^k) - F(\overline{\x}) ) \cdot ( F(\x^{k+1}) - F(\x^k) ) .
	\end{align}
	For convenience, for all $p,q>l$, we define
	\begin{align}
	\Delta_{p,q} : =\varphi (F(\x^{p}) - F(\overline{\x}) ) -  \varphi (F(\x^{q}) - F(\overline{\x}) ).
	\end{align}
	It follows from \eqref{eq:thm3-0} and \eqref{eq:thm3-1} that
	\begin{align*}
	\Delta_{k,k+1}\geq\varphi^\prime (F(\x^k) - F(\overline{\x}) ) \cdot ( F(\x^{k+1}) - F(\x^k) )
	\geq\frac{ \norm{\x^{k+1} -\x^k}^2 }{\rho\norm{\x^k - \x^{k-1}}},
	\end{align*}
	where
	$\rho=\frac{\rho_2}{\rho_1}$.
	
	Using the fact that $2\sqrt{\alpha \beta} \leq \alpha + \beta$, we obtain
	\begin{align}
	2\norm{\x^{k+1} - \x^k} \leq 2\sqrt{\rho \norm{\x^k - \x^{k-1}} \cdot \Delta_{k,k+1} }
	\leq \norm{\x^k-\x^{k+1}} + \rho \Delta_{k,k+1}.\label{eq:thm3-2}
	\end{align}
	Summing over the inequality \eqref{eq:thm3-2} for $i=l+1,\cdots ,k$ yields
	\begin{align*}
	2\sum_{i=l+1}^k \norm{\x^{i+1} - \x^i}
	&\leq  \sum_{i=l+1}^k  \norm{\x^i-\x^{i+1}} + \rho \sum_{i=l+1}^k\Delta_{i,i+1}\\
	&= \norm{\x^{l+1} - \x^l} - \norm{\x^{k+1}-\x^k} + \sum_{i=l+1}^k  \norm{\x^{i+1}-\x^{i}} + \rho \sum_{i=l+1}^k\Delta_{i,i+1}\\
	&= \norm{\x^{l+1} - \x^l} + \sum_{i=l+1}^k  \norm{\x^{i+1}-\x^{i}} + \rho \sum_{i=l+1}^k\Delta_{i,i+1}\\
	&= \norm{\x^{l+1} - \x^l} + \sum_{i=l+1}^k  \norm{\x^{i+1}-\x^{i}} + \rho \Delta_{l+1,k+1},
	\end{align*}
	where the last inequality follows from $\Delta_{p,r} = \Delta_{p,q} + \Delta_{q, r} $. Since $\varphi\geq 0$, we obtain
	\begin{align}
	\sum_{i=l+1}^k \norm{\x^{i+1} - \x^i} \leq
	\norm{\x^{l+1} - \x^l}  + \rho \varphi(F(\x^{l+1}) -F(\overline{\x}) ) =: B_0.\label{eq:thm3-3}
	\end{align}
	Since the upper bound $B_0$ is constant, the length of the sequence $\{\x^k\}$ is finite, \ie,
	\begin{align}
	\sum_{k=0}^\infty\norm{\x^{k+1}- \x^k} <\infty.
	\end{align}
	It is also clear that Eq.\eqref{eq:thm3-3} implies that the sequence $\{\x^k\}$ is a Cauchy sequence. In particular, with $q > p > l$, we have
	\begin{align}
	\norm{\x^q - \x^p}\leq \sum_{i=p}^{q-1}\norm{\x^{i+1} - \x^i}\leq  \sum_{i=l+1}^{q-1}\norm{\x^{i+1} - \x^i} \leq B_0.
	\end{align}
	Since \eqref{eq:thm3-3} implies that $\sum_{k=l+1}^{\infty}\norm{\x^{k+1}-\x^k}$ converges to zero as $l\rightarrow\infty$, it follows that $\{\x^k\}$ is a Cauchy sequence and hence it is a convergence sequence. The result follows immediately from Lemma~\ref{lemma:property of limit points}(ii).
\end{proof}
\end{document}